\documentclass[a4paper]{amsart}
\usepackage[utf8]{inputenc}
\usepackage[T1]{fontenc}

\usepackage{amsmath,amsthm,amssymb,amscd,bbm,comment, mathrsfs, mathabx, thmtools,mathtools,graphicx,enumitem,stmaryrd}
\usepackage{multicol}

\DeclareRobustCommand{\SkipTocEntry}[5]{}

\usepackage{xparse}
\usepackage[table]{xcolor}
\definecolor{blue}{rgb}{0.38, 0.51, 0.71}
\definecolor{red}{RGB}{175, 49, 39}
\definecolor{green}{RGB}{146, 227, 95}
\usepackage[english]{babel}%
\usepackage[colorlinks, linktoc=page, linkcolor={red}, citecolor={green!65!black}, urlcolor={blue}]{hyperref}
\addto\extrasenglish{%
}
\usepackage{anyfontsize}
\SetSymbolFont{stmry}{bold}{U}{stmry}{m}{n}
\allowdisplaybreaks[4]
\usepackage[all]{xy}
\usepackage[normalem]{ulem}
\usepackage{bm, tensor}

\usepackage{tikz,tikz-cd}

\usetikzlibrary{matrix}
\usetikzlibrary{fit}
\usetikzlibrary{calc}

\usetikzlibrary{backgrounds}
\usetikzlibrary{arrows}
\usetikzlibrary{shapes,shapes.geometric,shapes.misc}

\tikzstyle{tikzfig}=[baseline=-0.25em,scale=0.5]

\pgfkeys{/tikz/tikzit fill/.initial=0}
\pgfkeys{/tikz/tikzit draw/.initial=0}
\pgfkeys{/tikz/tikzit shape/.initial=0}
\pgfkeys{/tikz/tikzit category/.initial=0}

\pgfdeclarelayer{edgelayer}
\pgfdeclarelayer{nodelayer}
\pgfsetlayers{background,edgelayer,nodelayer,main}

\tikzstyle{none}=[inner sep=0mm]

\newcommand{\tikzfig}[1]{%
{\tikzstyle{every picture}=[tikzfig]
\IfFileExists{#1.tikz}
  {\input{#1.tikz}}
  {%
    \IfFileExists{./figures/#1.tikz}
      {\input{./figures/#1.tikz}}
      {\tikz[baseline=-0.5em]{\node[draw=red,font=\color{red},fill=red!10!white] {\textit{#1}};}}%
  }}%
}

\tikzstyle{every loop}=[]

\makeatletter
\def\slashedarrowfill@#1#2#3#4#5{%
	$\m@th\thickmuskip0mu\medmuskip\thickmuskip\thinmuskip\thickmuskip
	\relax#5#1\mkern-7mu%
	\cleaders\hbox{$#5\mkern-2mu#2\mkern-2mu$}\hfill
	\mathclap{#3}\mathclap{#2}%
	\cleaders\hbox{$#5\mkern-2mu#2\mkern-2mu$}\hfill
	\mkern-7mu#4$%
}
\def\rightslashedarrowfill@{%
	\slashedarrowfill@\relbar\relbar\mapstochar\rightarrow}
\newcommand\xslashedrightarrow[2][]{%
	\ext@arrow 0055{\rightslashedarrowfill@}{#1}{#2}}
\makeatother

\makeatletter
\newcommand{\ostar}{\mathbin{\mathpalette\make@circled\star}}
\newcommand{\make@circled}[2]{%
	\ooalign{$\m@th#1\smallbigcirc{#1}$\cr\hidewidth$\m@th#1#2$\hidewidth\cr}%
}
\newcommand{\smallbigcirc}[1]{%
	\vcenter{\hbox{\scalebox{0.77778}{$\m@th#1\bigcirc$}}}%
}
\makeatother

\usepackage{adjustbox}

\DeclareFontFamily{U}{min}{}
\DeclareFontShape{U}{min}{m}{n}{<-> dmjhira}{}
\DeclareMathAlphabet\EuRoman{U}{eur}{m}{n}
\SetMathAlphabet\EuRoman{bold}{U}{eur}{b}{n}
\newcommand{\euler}{\EuRoman}

\tikzcdset{scale cd/.style={every label/.append style={scale=#1},
		cells={nodes={scale=#1}}}}

\calclayout

\usepackage[bb=dsserif]{mathalpha}

\usetikzlibrary{calc}
\usetikzlibrary{decorations.pathmorphing}
\tikzset{curve/.style={settings={#1},to path={(\tikztostart)
			.. controls ($(\tikztostart)!\pv{pos}!(\tikztotarget)!\pv{height}!270:(\tikztotarget)$)
			and ($(\tikztostart)!1-\pv{pos}!(\tikztotarget)!\pv{height}!270:(\tikztotarget)$)
			.. (\tikztotarget)\tikztonodes}},
	settings/.code={\tikzset{quiver/.cd,#1}
		\def\pv##1{\pgfkeysvalueof{/tikz/quiver/##1}}},
	quiver/.cd,pos/.initial=0.35,height/.initial=0}

\tikzset{tail reversed/.code={\pgfsetarrowsstart{tikzcd to}}}
\tikzset{2tail/.code={\pgfsetarrowsstart{Implies[reversed]}}}
\tikzset{2tail reversed/.code={\pgfsetarrowsstart{Implies}}}
\tikzset{no body/.style={/tikz/dash pattern=on 0 off 1mm}}

\ProvideDocumentCommand{\hypersetup}{m}{}
\definecolor{blue(pigment)}{rgb}{0.2, 0.2, 0.6}
\definecolor{americanrose}{rgb}{1.0, 0.01, 0.24}
\definecolor{nicegreen}{rgb}{0.0, 0.5, 0.0}
\definecolor{deepmagenta}{rgb}{0.8, 0.0, 0.8}
\definecolor{deepcarrotorange}{rgb}{0.91, 0.41, 0.17}
\definecolor{cadetgrey}{rgb}{0.57, 0.64, 0.69}

\newtheorem{theoremm}{Theorem}[section]

\declaretheorem[style=plain,name=Theorem,numberlike=theoremm]{theorem}
\declaretheorem[style=plain,name=Theorem,numbered=no]{theorem*}

\declaretheorem[style=plain,name=Lemma,numberlike=theoremm]{lemma}
\declaretheorem[style=plain,name=Proposition,numberlike=theoremm]{proposition}
\declaretheorem[style=plain,name=Corollary,numberlike=theoremm]{corollary}

\declaretheorem[style=definition,name=Definition,numberlike=theorem]{definition}

\declaretheorem[style=remark,name=Example,numberlike=theorem]{example}
\declaretheorem[style=remark,name=Remark,numberlike=theorem]{remark}
\declaretheorem[style=remark,name=Notation,numberlike=theorem]{notation}

\newcommand{\on}[1]{\operatorname{#1}}
\newcommand{\setj}[1]{\left\{ #1 \right\}}

\newcommand{\kotimes}{\otimes_{\Bbbk}}

\newcommand{\thru}{\mathtt{th}}
\newcommand{\obj}[1]{\underline{\mathbf{#1}}}
\newcommand{\mobij}[1]{#1}
\newcommand{\mobijd}[1]{#1}
\newcommand{\mobijs}[1]{\underline{\mathbf{#1}}}
\newcommand{\mobijt}[1]{#1}

\newcommand{\uqsl}{{U}_{q}(\mathfrak{sl}_{2})}
\newcommand{\uqg}{{U}_{q}(\mathfrak{g})}

\DeclareMathSymbol{\blackdiamond}{\mathbin}{mathb}{"0C}

\tikzstyle{blob}=[fill=white, draw=black, shape=circle, inner sep=1.5pt]
\tikzstyle{fullblob}=[fill=black, draw=black, shape=circle, tikzit fill=black, tikzit draw=black, inner sep=1.5pt]
\tikzstyle{jonesrectangle}=[fill=white, draw=black, shape=rectangle, minimum width=1cm, minimum height=0.6cm]

\tikzstyle{bluearrow}=[draw={rgb,255: red,0; green,0; blue,255}, thick, <-]
\tikzstyle{bieski}=[-, draw={rgb,255: red,41; green,28; blue,218}, dashed, thick]
\tikzstyle{thickstrand}=[-, thick]
\tikzstyle{redarrow}=[draw={rgb,255: red,255; green,0; blue,0}, thick, <-]
\tikzstyle{blueline}=[draw={rgb,255: red,0; green,0; blue,255}, thick, -]
\tikzstyle{background}=[-, dashed, draw={rgb,255: red,50; green,50; blue,50}, thin]
\tikzstyle{blue arrow}=[->, draw=blue, thick]
\tikzstyle{redline}=[-, draw=red, thick]
\tikzstyle{arrow}=[->, thick]

\begin{document}
	\title{A Coboundary Temperley--Lieb Category for \texorpdfstring{\lowercase{$\mathfrak{sl}_{2}$}}{sl2}-Crystals}
	
	\author{Moaaz Alqady}
	\address{M.A., Department of Mathematics, University of Oregon, Fenton Hall, Eugene, OR 97403, USA}
	\email{malqady@uoregon.edu}
	
	\author{Mateusz Stroi{\' n}ski}
	\address{M.S., Department of Mathematics, Uppsala University, Box. 480, SE-75106, Uppsala, Sweden}
	\email{mateusz.stroinski@math.uu.se}
	
	\begin{abstract}
       By considering a suitable renormalization of the Temperley--Lieb category, we study its specialization to the case $q=0$. Unlike the $q\neq 0$ case, the obtained monoidal category, $\mathcal{TL}_0(\Bbbk)$, is not rigid or braided. 
       We provide a closed formula for the Jones--Wenzl projectors in $\mathcal{TL}_0(\Bbbk)$ and give semisimple bases for its endomorphism algebras. We explain how to obtain the same basis using the representation theory of finite inverse monoids, via the associated M{\" o}bius inversion.
       We then describe a coboundary structure on $\mathcal{TL}_0(\Bbbk)$ and show that its idempotent completion is coboundary monoidally equivalent to the category of $\mathfrak{sl}_{2}$-crystals. This gives a diagrammatic description of the commutor for $\mathfrak{sl}_{2}$-crystals defined by Henriques and Kamnitzer and of the resulting action of the cactus group.
       We also study fiber functors of $\mathcal{TL}_0(\Bbbk)$ and discuss how they differ from the $q\neq 0$ case.
	\end{abstract}

	\maketitle
	
	\tableofcontents
    \addtocontents{toc}{\SkipTocEntry}
\subsection*{Acknowledgements}\label{sec:acknowledgements}
 The authors would like to thank Victor Ostrik for introducing them to the idea of a diagrammatic presentation for $\mathfrak{sl}_{2}$-crystals, and for many valuable discussions; Volodymyr Mazorchuk, for suggesting the connection between Jones-Wenzl projectors and M{\" o}bius inversion; Iva Halacheva for discussions about the cactus group action on crystals; and Alexis Langlois-R{\' e}millard for many helpful comments on the manuscript.
	
\section{Introduction}

Diagrammatically defined algebras and monoidal categories are an increasingly important tool in algebraic and categorical representation theory. Among the early examples of such structures are the Temperley--Lieb algebras and Temperley--Lieb categories. They find applications in many different settings, such as knot invariants (\cite{Ka}) and Soergel bimodules (\cite{E2}), and help formalize connections between the areas in which they arise. However, perhaps their most important, {\it defining} property, is the equivalence 
\[
\mathbf{Fund}(\uqsl)
 \simeq \mathcal{TL}_{q}(\Bbbk),
\]
between the Temperley--Lieb category $\mathcal{TL}_{q}(\Bbbk)$ and the category of tensor powers of the fundamental representation of $\uqsl$, providing a diagrammatic interpretation of the latter category. In many cases, this interpretation gives a better understanding of the representation theory of $\uqsl$, and it often greatly facilitates computations. This is to a large extent because the above equivalence is monoidal, and so the tensor product of $\uqsl$-modules is captured by horizontal concatenation of diagrams. Finding diagrammatic categories with similar property for $\uqg$ for more general $\mathfrak{g}$ is an extensively studied problem, initially posed in \cite{Ku}, where it was solved for $\mathfrak{g}$ of rank $2$. It has since been solved in type $A$ in \cite{CKM} and recently also in types $C$ in \cite{BERT}. For ongoing progress in types $B$ and $D$, see \cite{BW},\cite{BT}, and see \cite{SW} for results in type $F$.

Another fundamental idea in the theory of quantum groups is that of crystal bases, as introduced by Kashiwara (\cite{K1}). For finite-dimensional semisimple $\mathfrak{g}$, and finite-dimensional modules over $\uqg$, crystal bases can be obtained as a specialization of the canonical bases of \cite{Lu}, by setting $q=0$. However, in order to make this claim precise (and well-defined), a certain renormalization needs to be imposed. Working more abstractly, we obtain the category $\mathfrak{g}\mathbf{-Crys}$ of $\mathfrak{g}$-crystals. Crucially, this category is monoidal, and hence can be used to find bases for tensor products of $\uqg$-modules.

In this article, we connect these two aspects of representation theory of quantum groups in the simplest case, $\mathfrak{g} = \mathfrak{sl}_{2}$, by defining a diagrammatic monoidal category $\mathcal{TL}_{0}(\Bbbk)$ which is a specialization of the Temperley--Lieb category at $q=0$. Similarly to the definition of $\mathfrak{g}$-crystals, in order to make the specialization claim well-defined, a renormalization of the Temperley--Lieb category is required. We describe it in detail in \autoref{def:renormalization}. As one of our main results, \autoref{thm:equivalence}, we show that the completion of $\mathcal{TL}_{0}(\Bbbk)$ under direct sums and direct summands, which we denote by $\mathbf{CrysTL}$,
is monoidally equivalent to the category $\mathfrak{sl}_{2}\mathbf{-Crys}$.

We remark that the category $\mathcal{TL}_{0}(\Bbbk)$ we study has been previously considered by Virk in \cite{V}, in the context of an action of $\mathfrak{sl}_{2}\mathbf{-Crys}$ on a graded version of category $\mathcal{O}$ for $\mathfrak{sl}_{2}$, and, much more recently, in the recent work of Etingof and Penneys (\cite{EP}), where the unusual features of the monoidal structure on $\mathcal{TL}_{0}(\Bbbk)$ are used to illustrate the necessity of braiding in \cite[Theorem~1.1]{EP}; see \cite[Section~3.1]{EP} for details.

Beyond an elementary, detailed proof of the monoidal equivalence between
$\mathbf{CrysTL}$ and $\mathfrak{sl}_{2}\mathbf{-Crys}$, as well as a sketch of a different argument in \autoref{rem:alternativeequivalence}, we give elementary proofs of other structural properties of $\mathbf{CrysTL}$:
\begin{enumerate}
 \item\label{case1} In \autoref{TLBasis}, we establish a basis theorem for $\mathcal{TL}_{0}(\Bbbk)$, showing that Temperley--Lieb diagrams give bases for hom-spaces.
 \item\label{case2} In \autoref{JWuniqueness}, we find a closed formula for its Jones--Wenzl projectors, stated in \autoref{def:JW}.
 \item\label{case3} In \autoref{cor:semisimple}, we give a representation-theoretic proof of its semisimplicity. 
\end{enumerate}
  In the cases \eqref{case1} and \eqref{case3}, the proofs resemble those for the case of generic $q$, respectively $q\neq 0$, but the zigzag relation for $q=0$ implies significant simplifications. In case \eqref{case2}, the formula we obtain for Jones--Wenzl projectors is both more explicit and simpler in form than those known for $q \neq 0$, see \cite{Mo} and \cite{FK}.

Additionally, we observe that the endomorphism algebras in $\mathcal{TL}_{0}(\Bbbk)$ are easily identified as reduced (or {\it contracted}) monoid algebras. We connect this observation to the approach of Smith (\cite{Sm}) to $\mathfrak{g}$-crystals via categories enriched in $\mathbf{Set}_{\ast}$, the category of pointed sets. We show that the monoids in question are finite inverse monoids, which allows us to recover our descriptions of both the Jones--Wenzl projectors and of the induced semisimple bases purely in terms of representation theory of finite inverse monoids, more precisely, in terms of M{\" o}bius inversion for such monoids (see \cite[Chapter~9]{St}).

While many algebraic and representation-theoretic properties of $\mathcal{TL}_{0}(\Bbbk)$ are as nice, and often easier to derive, as the (generic) case $q \neq 0$, essentially the opposite is true for the monoidal structure of $\mathcal{TL}_{0}(\Bbbk)$. In particular:
\begin{enumerate}
    \item\label{introrigid} $\mathcal{TL}_{0}(\Bbbk)$ is not rigid (\autoref{prop:notrigid});
    \item\label{introbraided} $\mathcal{TL}_{0}(\Bbbk)$ is not braided 
    (\autoref{notbraided}).
\end{enumerate}
In fact, the results of \cite{EP} show that \eqref{introrigid} implies \eqref{introbraided}, establishing a connection between the two. We study further the consequences of these two properties, the former by investigating the fiber functors for $\mathbf{CrysTL}$ in \autoref{subsec:fiberfunctors}, and the latter by providing a diagrammatic description of the coboundary category structure on $\mathcal{TL}_{0}(\Bbbk)$. This structure extends uniquely to all of $\mathbf{CrysTL}$.

While the category $\mathfrak{g}\mathbf{-Crys}$ does not admit a braiding, perhaps the most interesting categorical structure it is endowed with is its {\it coboundary structure}. 
A coboundary structure on a monoidal category $\mathcal{C}$ consists of a {\it commutor} $\sigma_{A,B}: A \otimes B \rightarrow B \otimes A$, invertible for all $A,B$ in $\mathcal{C}$, natural in both $A$ and $B$, and such that it defines an action of the {\it cactus group} $J_{n}$ on $A^{\otimes n}$, just like symmetric and braided monoidal categories define actions of the symmetric groups and braid groups respectively. 
In \cite[Section~3]{D}, Drinfeld gives a construction which, given a quasitriangular, topologically free Hopf algebra over $\Bbbk[[\hbar]]$, produces a coboundary Hopf algebra. This yields a commutor on $\mathbf{Rep}(\uqg)$, and, as described in \cite{KT}, a commutor on $\mathfrak{g}\mathbf{-Crys}$.
Following an idea of Berenstein, a more explicit, combinatorial construction of a commutor for $\mathfrak{g}\mathbf{-Crys}$ (which is closely connected, but not equal, to that in \cite{KT}, see there for details), and further also for $\mathbf{Rep}(\uqg)$, was given by Henriques and Kamnitzer in \cite{HKam}.

In \autoref{sec:commutor}, we make the coboundary structure of \cite{HKam} on
$\mathfrak{sl}_{2}\mathbf{-Crys}$ more accessible by determining the corresponding coboundary structure on the diagrammatic category $\mathcal{TL}_{0}(\Bbbk)$. Our description, formulated in \autoref{TLCob}, is given directly by closed formulas in terms of the semisimple basis we obtain in \autoref{EndAlgebras} by using Jones--Wenzl projectors. 

We establish the coboundary axioms purely combinatorially, using diagrammatic properties of $\mathcal{TL}_{0}(\Bbbk)$ we obtain in the other sections. We then show that the coboundary structure we define on $\mathcal{TL}_{0}(\Bbbk)$ makes the equivalence of \autoref{thm:equivalence} a coboundary equivalence.
Towards that end, in \autoref{ProjEmb} we describe the bijection induced by this equivalence between the various summands of objects in $\mathcal{TL}_{0}(\Bbbk)$ and the direct summands of the tensor powers of the crystal underlying the defining $2$-dimensional representations of $\uqsl$.

In \autoref{cor:intervalreversal} we also calculate the interval-reversing morphisms (see \cite[Section~3.1]{HKam}) in the coboundary structure for $\mathcal{TL}_{0}(\Bbbk)$, the diagrams for which have a rather intuitive visual form. 
Finally, we remark that our description of the commutor involves an automorphism $\kappa_{m,n}$ of the set $\mathscr{D}_{m+n}$ of the set of cap diagrams whose domain has $m+n$ strands, and in turn a bijection on the summands of the $(m+n)^\text{th}$ tensor power of the fundamental representation of $\uqsl$. While $\kappa_{m,n}$ too seem quite visually intuitive, the authors are not aware of a representation-theoretic interpretation for these bijections.

A fiber functor on $\mathcal{TL}_{0}(\Bbbk)$ defines a bilinear form $\mathsf{b}$, realizing the cap diagram. In the case $q\neq 0$ studied in \cite{EO}, and similarly for some other rigid diagrammatic categories (see \cite{Tu}), the zig-zag relation allows one to reduce the characterization of fiber functors to just equivalence relations of bilinear forms which satisfy an additional property imposed by the circle evaluation relation. In particular, the images of the cup and the cap diagrams determine each other - we give a brief account of this in \autoref{Turaev}.

This reduction cannot be made in the case of $\mathcal{TL}_{0}(\Bbbk)$. The bilinear form $\mathsf{b}$ must be degenerate, and the $2$-tensor $\mathsf{t}$ realizing the cup diagram lies in the tensor product of the left and the right radicals of the form. Circle evaluation entails $\mathsf{b}(\mathsf{t}) = 1$, which allows us 
to decompose a fiber functor $(\mathsf{b},\mathsf{t})$ as an inflation of a fiber functor $(\mathsf{b}',\mathsf{t}')$ by a pair $(\mathsf{b}'',\mathsf{t}'')$ such that $\mathsf{b}''(\mathsf{t}'') = 0$ and all the degenerate blocks of $\mathsf{b}'$ are two-dimensional, see~\autoref{prop:inflationreduction}. 
We thus study the fiber functors of the latter kind, and show that, fixing $\mathsf{b}'$, their isomorphism classes correspond to the orbits of the conjugacy action of $\on{GL}(V)$ on $\mathfrak{sl}(V)$, where $V$ is the left radical of $\mathsf{b}''$, see~\autoref{cor:GIT}.

This shows that both the problem of classifying fiber functors, as well as the structure of the moduli space they form (see \cite[Section~3.1]{EO}) is more complicated than in the case $q\neq 0$. Indeed, the case $\mathsf{b}'' = 0$ gives us a moduli space of choices of $\mathsf{t}''$, which, as an affine GIT quotient, is of dimension $m-1$ (\autoref{cor:GIT}), while the dimension of the total space on which $\mathsf{b}$ is defined equals $2m$. Again, this contrasts the generic case $q \neq 0$ where $\mathsf{b,t}$ determine each other.

Furthermore, in \autoref{notagroupoid} we show that for $q = 0$ the {\it category} of fiber functors does not form a groupoid. We give examples of non-invertible morphisms of fiber functors, and in \autoref{prop:HopfModules} explain how the existence of such morphisms in the rigid case would contradict the Fundamental Theorem of Hopf Modules. Finally, we give examples of some operations for fiber functors on $\mathcal{TL}_{0}(\Bbbk)$, indicating further structure on the category they form.

The rest of the document is organized as follows. \autoref{sec:prelim} contains the necessary preliminaries and definitions for $\mathfrak{g}$-crystals, coboundary categories and Temperley--Lieb categories for $q\neq 0$. We also recall the commutor of \cite{HKam}. In \autoref{sec:crysTL}, we define the category $\mathcal{TL}_{0}(\Bbbk)$, establish its semisimplicity (\autoref{cor:semisimple}), the basis theorem (\autoref{TLBasis}), the description of Jones--Wenzl projectors (\autoref{JWuniqueness}). 
In \autoref{sec:mobius}, we describe the relation between M{\" o}bius inversion for inverse monoids and the semisimple basis for $\mathcal{TL}_{0}(\Bbbk)$. In \autoref{sec:equiv}, we establish a monoidal equivalence between the category of $\mathfrak{sl}_{2}$-crystals and (the Cauchy completion) of $\mathcal{TL}_{0}(\Bbbk)$ (\autoref{thm:equivalence}). In \autoref{sec:commutor} we define the commutor for $\mathcal{TL}_{0}(\Bbbk)$ (\autoref{TLCob}), and in \autoref{sec:coboundaryequiv}, we show that the equivalence of \autoref{thm:equivalence} is a coboundary equivalence (\autoref{coboundaryF}).
Finally, in \autoref{subsec:fiberfunctors}, we study the category fiber functors of $\mathcal{TL}_0(\Bbbk)$.

\section{Preliminaries}\label{sec:prelim}
	
	\subsection{$\mathfrak{g}$-Crystals}
	
	Given any complex reductive Lie algebra $\mathfrak{g}$, a $\mathfrak{g}$-crystal should be thought of as a combinatorial model for a finite-dimensional representation of $\mathfrak{g}$.
	Let $\Lambda$ denote the weight lattice of $\mathfrak{g}$, $\Lambda_+$ its set of dominant weights, $I$ the vertex set of the corresponding Dynkin diagram, $\{\alpha_i\}_{i\in I}$ the set of its simple roots, and $\{\alpha_i^\vee\}_{i\in I}$ the set of its simple coroots.
	
	\begin{definition}\label{def:gcrystal}
		A \emph{$\mathfrak{g}$-crystal} is a finite set $B$ together with maps 
        $\{e_i,f_i: B\to B\amalg \{0\}\}_{i\in I}$, $\{\varepsilon_i,\phi_i: B\to \mathbb{Z}\}_{i\in I}$ 
        and
        $\mathrm{wt}:B\to \Lambda$,
        such that, for all $i\in I$, and $b,b'\in B$,
		\begin{enumerate}
			\item\label{semiregular} $\varepsilon_i(b) = \max\{n\,|\,e_i^n(b)\neq 0\}$ and $\phi(b) = \max\{n\,|\,f_i^n(b)\neq 0\}$;
			\item $\phi_i(b)-\varepsilon_i(b) = \langle \mathrm{wt(b),\alpha_i^\vee\rangle}$;
			\item if $e_i(b)\neq 0$, then $\mathrm{wt}(e_i(b)) = \mathrm{wt}(b)+\alpha_i$; similarly if $f_i(b)\neq 0$, then $\mathrm{wt}(f_i(b)) = \mathrm{wt}(b)-\alpha_i$;
			\item $b'=e_i(b)$ if and only if $b=f_i(b')$.
		\end{enumerate}
	The maps $e_i$ and $f_i$ are called the \emph{Kashiwara operators} and are a renormalizations of the usual Chevalley operators of $\mathfrak{g}$.
	\end{definition}

	\begin{definition}
		Let $A$ and $B$ be two $\mathfrak{g}$-crystals. A \emph{morphism of crystals} $\psi:A\to B$ is a (set-theoretic) function $\psi:A\amalg \{0\} \to B \amalg \{0\}$ such that, for all $a\in A$ and $i\in I$,
		\begin{enumerate}
			\item $\psi(0)=0$
			\item $\mathrm{wt}(\psi(a)) = \mathrm{wt}(a)$
			\item $\psi(e_i\cdot a) = e_i\cdot \psi(a)$ and $\psi(f_i\cdot a) = f_i\cdot \psi(a)$.
		\end{enumerate}
	\end{definition}
    
 \begin{remark}
    In some settings a more general notion of crystals (see \cite[Definition~4.5.1]{HK}) is considered, which omits property~\eqref{semiregular} in \autoref{def:gcrystal}. The crystals we consider would then be referred to as \emph{normal} or \emph{semiregular} crystals. Similarly, there is also a more general notion of a morphism of crystals (see \cite[Definition~4.5.5]{HK}), and the notion we use is sometimes referred to as \emph{strict} morphisms. The less general notions we consider are motivated by the representation theory of quantum groups.
 \end{remark}
	
	Perhaps the most remarkable property of crystals is the existence of a tensor product operation on them, which allows us to understand and decompose tensor products of representations of $U_q(\mathfrak{g})$ according to simple combinatorial rules. 
	
	\begin{definition} \label{crysTensor}
		Given two crystals $A$ and $B$, the tensor product $A\otimes B$ is defined as the set $A\times B$ whose crystal structure is given by
		\begin{enumerate}
			\item $\mathrm{wt}(a\otimes b) = \mathrm{wt}(a) + \mathrm{wt}(b)$
			\item $e_i(a\otimes b) = 
			\begin{cases}
				e_i a \otimes b & \text{if } \phi_i(a)\geq \varepsilon_i(b)\\
				a \otimes e_i b & \text{if } \phi_i(a) < \varepsilon_i(b)
			\end{cases}$
			\item $f_i(a\otimes b) = 
			\begin{cases}
				f_i a \otimes b & \text{if } \phi_i(a) > \varepsilon_i(b)\\
				a \otimes f_i b & \text{if } \phi_i(a) \leq \varepsilon_i(b),
			\end{cases}$
		\end{enumerate}
		for all $a\in A$ and $b\in B$.
	\end{definition}
	
	One may think of a crystal $B$ more visually as an edge-colored directed graph with vertex set $B$, and where there is an edge of color $i\in I$ from $b$ to $b'$ if $f_i(b) = b'$.
	Such a graph is called the crystal graph of $B$, and it carries the same information as the crystal itself.
	A crystal is called $\emph{connected}$ if its crystal graph is connected.
	
	\begin{example}
		The following crystal graph corresponds to the adjoint representation of $\mathfrak{sl}_3$ and it can be seen in the weight lattice (with the two bases vectors of the 0-weight space identified).
		
		
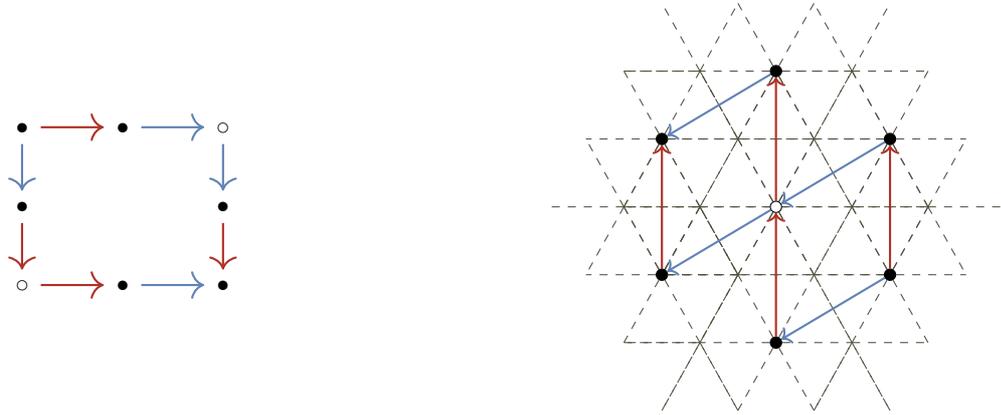
\begin{figure}[!htb]
			\centering
			\begin{minipage}{.5\textwidth}
				\centering
\[\begin{tikzcd}
	\bullet & \bullet & \circ \\
	\bullet && \bullet \\
	\circ & \bullet & \bullet
	\arrow[draw={rgb,255:red,255;green,0;blue,0}, thick, from=1-1, to=1-2]
	\arrow[draw={rgb,255:red,0;green,0;blue,255}, thick, from=1-1, to=2-1]
	\arrow[draw={rgb,255:red,0;green,0;blue,255}, thick, from=1-2, to=1-3]
	\arrow[color={rgb,255:red,0;green,0;blue,255}, thick, from=1-3, to=2-3]
	\arrow[draw={rgb,255:red,255;green,0;blue,0}, thick, from=2-1, to=3-1]
	\arrow[color={rgb,255:red,255;green,0;blue,0}, thick, from=2-3, to=3-3]
	\arrow[color={rgb,255:red,255;green,0;blue,0}, thick, from=3-1, to=3-2]
	\arrow[color={rgb,255:red,0;green,0;blue,255}, thick, from=3-2, to=3-3]
\end{tikzcd}\]
			\end{minipage}%
			\begin{minipage}{0.5\textwidth}
				\centering
                \[
				\begin{tikzpicture}
	\begin{pgfonlayer}{nodelayer}
		\node [style=none] (18) at (0.5, 2.7) {};
		\node [style=none] (26) at (-0.5, 0.9) {};
		\node [style=none] (31) at (-1, 1.8) {};
		\node [style=none] (34) at (0.5, -0.9) {};
		\node [style=none] (38) at (0.5, 0.9) {};
		\node [style=none] (40) at (-1, 0) {};
		\node [style=none] (44) at (-0.5, -0.9) {};
		\node [style=none] (46) at (-0.5, 0.9) {};
		\node [style=none] (50) at (-0.5, -0.9) {};
		\node [style=none] (54) at (-0.5, -0.9) {};
		\node [style=none] (55) at (-1, -1.8) {};
		\node [style=none] (56) at (0.5, -0.9) {};
		\node [style=blob] (61) at (0, 0) {};
		\node [style=none] (62) at (1, 1.8) {};
		\node [style=none] (63) at (0.5, 0.9) {};
		\node [style=none] (64) at (2, 1.8) {};
		\node [style=none] (67) at (0.5, 2.7) {};
		\node [style=none] (68) at (1.5, 2.7) {};
		\node [style=none] (71) at (2.5, 0.9) {};
		\node [style=none] (72) at (0.5, 0.9) {};
		\node [style=none] (74) at (1, 1.8) {};
		\node [style=none] (75) at (2, 1.8) {};
		\node [style=none] (78) at (3, 0) {};
		\node [style=none] (80) at (2.5, -0.9) {};
		\node [style=none] (82) at (2.5, 0.9) {};
		\node [style=none] (84) at (1, -1.8) {};
		\node [style=none] (85) at (2.5, -0.9) {};
		\node [style=none] (86) at (0.5, -0.9) {};
		\node [style=none] (87) at (2, -1.8) {};
		\node [style=none] (88) at (1, 0) {};
		\node [style=none] (89) at (2, 0) {};
		\node [style=none] (91) at (-2.5, -0.9) {};
		\node [style=none] (93) at (-3, 0) {};
		\node [style=none] (95) at (-2.5, 0.9) {};
		\node [style=none] (99) at (-0.5, 0.9) {};
		\node [style=none] (100) at (-2.5, 0.9) {};
		\node [style=none] (101) at (-1, 0) {};
		\node [style=none] (102) at (-2, 1.8) {};
		\node [style=none] (103) at (-1, 1.8) {};
		\node [style=fullblob] (104) at (1.5, 0.9) {};
		\node [style=fullblob] (105) at (1.5, -0.9) {};
		\node [style=fullblob] (106) at (0, 1.8) {};
		\node [style=fullblob] (107) at (0, -1.8) {};
		\node [style=none] (115) at (-2.5, -0.9) {};
		\node [style=none] (116) at (-1, -1.8) {};
		\node [style=none] (119) at (-1, 0) {};
		\node [style=none] (121) at (-2, 0) {};
		\node [style=none] (122) at (-0.5, -0.9) {};
		\node [style=fullblob] (123) at (-1.5, -0.9) {};
		\node [style=fullblob] (124) at (-1.5, 0.9) {};
		\node [style=none] (130) at (-1.5, 2.7) {};
		\node [style=none] (131) at (-0.5, 2.7) {};
		\node [style=none] (133) at (-1, 1.8) {};
		\node [style=none] (137) at (-2, 1.8) {};
		\node [style=none] (138) at (-0.5, 0.9) {};
		\node [style=none] (144) at (0.5, -0.9) {};
		\node [style=none] (146) at (0.5, -2.7) {};
		\node [style=none] (147) at (1, -1.8) {};
		\node [style=none] (148) at (2, -1.8) {};
		\node [style=none] (150) at (1.5, -2.7) {};
		\node [style=none] (157) at (-0.5, -0.9) {};
		\node [style=none] (158) at (-1.5, -2.7) {};
		\node [style=none] (159) at (-1, -1.8) {};
		\node [style=none] (161) at (-2, -1.8) {};
		\node [style=none] (162) at (-0.5, -2.7) {};
	\end{pgfonlayer}
	\begin{pgfonlayer}{edgelayer}
		\draw [style=background] (26.center) to (31.center);
		\draw [style=background] (40.center) to (44.center);
		\draw [style=background] (40.center) to (46.center);
		\draw [style=background] (26.center) to (40.center);
		\draw [style=background] (54.center) to (56.center);
		\draw [style=background] (54.center) to (55.center);
		\draw [style=background] (54.center) to (61);
		\draw [style=background] (40.center) to (61);
		\draw [style=background] (26.center) to (61);
		\draw [style=background] (61) to (40.center);
		\draw [style=background] (61) to (44.center);
		\draw [style=background] (61) to (56.center);
		\draw [style=background] (61) to (46.center);
		\draw [style=background] (62.center) to (64.center);
		\draw [style=background] (62.center) to (63.center);
		\draw [style=background] (62.center) to (68.center);
		\draw [style=background] (62.center) to (67.center);
		\draw [style=background] (61) to (63.center);
		\draw [style=background] (26.center) to (72.center);
		\draw [style=background] (72.center) to (26.center);
		\draw [style=background] (72.center) to (74.center);
		\draw [style=background] (72.center) to (61);
		\draw [style=background] (86.center) to (50.center);
		\draw [style=background] (86.center) to (84.center);
		\draw [style=background] (86.center) to (61);
		\draw [style=background] (89.center) to (78.center);
		\draw [style=background] (89.center) to (88.center);
		\draw [style=background] (89.center) to (80.center);
		\draw [style=background] (89.center) to (82.center);
		\draw [style=background] (86.center) to (88.center);
		\draw [style=background] (88.center) to (89.center);
		\draw [style=background] (88.center) to (34.center);
		\draw [style=background] (88.center) to (38.center);
		\draw [style=background] (72.center) to (88.center);
		\draw [style=background] (88.center) to (61);
		\draw [style=background] (61) to (88.center);
		\draw [style=background] (54.center) to (101.center);
		\draw [style=background] (104) to (71.center);
		\draw [style=background] (104) to (72.center);
		\draw [style=background] (104) to (75.center);
		\draw [style=background] (104) to (74.center);
		\draw [style=background] (72.center) to (104);
		\draw [style=background] (62.center) to (104);
		\draw [style=background] (89.center) to (104);
		\draw [style=background] (104) to (88.center);
		\draw [style=background] (104) to (89.center);
		\draw [style=background] (88.center) to (104);
		\draw [style=background] (105) to (85.center);
		\draw [style=background] (105) to (86.center);
		\draw [style=background] (105) to (84.center);
		\draw [style=background] (105) to (87.center);
		\draw [style=background] (105) to (89.center);
		\draw [style=background] (105) to (88.center);
		\draw [style=background] (86.center) to (105);
		\draw [style=background] (89.center) to (105);
		\draw [style=background] (88.center) to (105);
		\draw [style=background] (26.center) to (106);
		\draw [style=background] (106) to (18.center);
		\draw [style=background] (106) to (31.center);
		\draw [style=background] (106) to (46.center);
		\draw [style=background] (62.center) to (106);
		\draw [style=background] (106) to (63.center);
		\draw [style=background] (72.center) to (106);
		\draw [style=background] (106) to (74.center);
		\draw [style=background] (54.center) to (107);
		\draw [style=background] (86.center) to (107);
		\draw [style=background] (119.center) to (121.center);
		\draw [style=background] (119.center) to (122.center);
		\draw [style=background] (40.center) to (121.center);
		\draw [style=background] (121.center) to (101.center);
		\draw [style=background] (121.center) to (93.center);
		\draw [style=background] (121.center) to (91.center);
		\draw [style=background] (121.center) to (95.center);
		\draw [style=background] (122.center) to (116.center);
		\draw [style=background] (122.center) to (119.center);
		\draw [style=background] (40.center) to (123);
		\draw [style=background] (54.center) to (123);
		\draw [style=background] (119.center) to (123);
		\draw [style=background] (123) to (122.center);
		\draw [style=background] (123) to (115.center);
		\draw [style=background] (123) to (116.center);
		\draw [style=background] (121.center) to (123);
		\draw [style=background] (122.center) to (123);
		\draw [style=background] (123) to (119.center);
		\draw [style=background] (123) to (121.center);
		\draw [style=background] (26.center) to (124);
		\draw [style=background] (40.center) to (124);
		\draw [style=background] (124) to (99.center);
		\draw [style=background] (124) to (100.center);
		\draw [style=background] (124) to (101.center);
		\draw [style=background] (124) to (103.center);
		\draw [style=background] (124) to (102.center);
		\draw [style=background] (124) to (121.center);
		\draw [style=background] (121.center) to (124);
		\draw [style=background] (137.center) to (133.center);
		\draw [style=background] (138.center) to (133.center);
		\draw [style=background] (106) to (131.center);
		\draw [style=background] (133.center) to (131.center);
		\draw [style=background] (133.center) to (130.center);
		\draw [style=background] (133.center) to (137.center);
		\draw [style=background] (133.center) to (138.center);
		\draw [style=background] (146.center) to (147.center);
		\draw [style=background] (150.center) to (147.center);
		\draw [style=background] (147.center) to (148.center);
		\draw [style=background] (147.center) to (146.center);
		\draw [style=background] (147.center) to (144.center);
		\draw [style=background] (147.center) to (150.center);
		\draw [style=background] (158.center) to (159.center);
		\draw [style=background] (161.center) to (159.center);
		\draw [style=background] (162.center) to (159.center);
		\draw [style=background] (123) to (161.center);
		\draw [style=background] (159.center) to (158.center);
		\draw [style=background] (159.center) to (157.center);
		\draw [style=background] (159.center) to (161.center);
		\draw [style=background] (159.center) to (162.center);
		\draw [style=background] (107) to (162.center);
		\draw [style=background] (107) to (159.center);
		\draw [style=background] (107) to (146.center);
		\draw [style=background] (107) to (147.center);
		\draw [style=bluearrow] (124) to (106);
		\draw [style=bluearrow] (123) to (61);
		\draw [style=bluearrow] (61) to (104);
		\draw [style=redarrow] (61) to (107);
		\draw [style=redarrow] (104) to (105);
		\draw [style=redarrow] (124) to (123);
		\draw [style=bluearrow] (107) to (105);
		\draw [style=redarrow] (106) to (61);
	\end{pgfonlayer}
\end{tikzpicture}
				\]
			\end{minipage}
		\end{figure}
		
	\end{example} 
	\begin{example}
		The following tensor product computation for $\mathfrak{sl}_2$-crystals shows the decomposition
		$$V(2) \otimes V(3) \cong V(5) \oplus V(3)  \oplus V(1),$$
		where $V(\lambda)$ is the highest weight representation of $\mathfrak{sl}_2$ of weight $\lambda$.
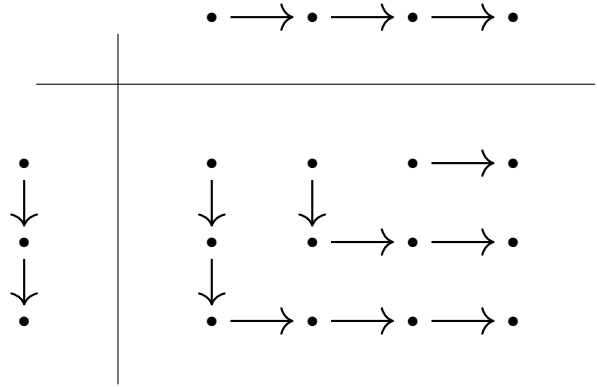
		\[\begin{tikzcd}
			& {} & \bullet & \bullet & \bullet & \bullet \\
			{} &&&&&& {} \\
			\bullet && \bullet & \bullet & \bullet & \bullet \\
			\bullet && \bullet & \bullet & \bullet & \bullet \\
			\bullet && \bullet & \bullet & \bullet & \bullet \\
			& {}
			\arrow[no head, from=1-2, to=6-2]
			\arrow[thick, from=1-3, to=1-4]
			\arrow[thick, from=1-4, to=1-5]
			\arrow[thick, from=1-5, to=1-6]
			\arrow[no head, from=2-1, to=2-7]
			\arrow[thick, from=3-1, to=4-1]
			\arrow[thick, from=3-3, to=4-3]
			\arrow[thick, from=3-4, to=4-4]
			\arrow[thick, from=3-5, to=3-6]
			\arrow[thick, from=4-1, to=5-1]
			\arrow[thick, from=4-3, to=5-3]
			\arrow[thick, from=4-4, to=4-5]
			\arrow[thick, from=4-5, to=4-6]
			\arrow[thick, from=5-3, to=5-4]
			\arrow[thick, from=5-4, to=5-5]
			\arrow[thick, from=5-5, to=5-6]
		\end{tikzcd}\]
	\end{example}

	Connected graphs should be thought of as analogs of irreducible representations, and thus the problem of decomposing some tensor product of representations into irreducible under this correspondence amounts to decomposing the crystal graph of the tensor crystal into connected components.
	
	This correspondence, however, is not perfect.
	A crystal $B$ is called of highest weight $\lambda\in \Lambda_+$ is there exists $b_0\in B$ with $\mathrm{wt}(b_0) = \lambda$, and such that $B$ is generated by $b_0$ under the repeated actions of the maps $f_i$; such $b_0$ is then called a \emph{highest weight element}.
	Unfortunately, not every connected crystal is of highest weight, and there exists non-isomorphic highest weight crystals of the same weight $\lambda$.
	This issue can be fixed by considering families of crystals instead.
	
	Let $\mathcal B = \{B_\lambda \,|\, \lambda \in \Lambda_+\}$ be a family of crystals, where $B_\lambda$ is a crystal of highest weight $\lambda$.
	Fixing inclusions of crystals $\iota_{\lambda, \mu}: B_{\lambda+\mu}\to B_{\lambda}\otimes B_{\mu}$ for every $\lambda, \mu\in \Lambda_+$, we say that $(\mathcal B, (\iota)_{\lambda, \mu \,\in\, \Lambda_+})$ is a \emph{closed family} of crystals.
	It turns out that there is a \emph{unique} closed family of crystals (see \cite[Section~6.4]{J}), and so we may unambiguously refer to $B_\lambda$ as members of this unique family.
	
	We can now define the category of crystals.
	
	\begin{definition}
		The category $\mathfrak{g}\mathbf{-Crys}$ is the $\mathbb{k}$-linear category whose objects are crystals $B$ such that each connected component of $B$ is $B_\lambda$ for some $\lambda\in\Lambda_+$, and whose morphisms are $\mathbb{k}$-linear combinations of crystal morphisms. It is semisimple and monoidal, where the monoidal structure is given by the tensor product of crystals, and the direct sum of two crystals is their disjoint union.
	\end{definition}
	
	\subsection{Coboundary Categories and Cactus Actions}\label{CobCats}	

    In this section, we will review the basic definitions and properties of a coboundary cateogory, and in the next section we discuss the explicit commutor on $\mathfrak{g}\mathbf{-Crys}$, following the definitions and notation given in \cite{HKam}.
	\begin{definition}[{\cite[Section~3]{D}}]
		A \emph{coboundary category} is a monoidal category $\mathcal C$ together with a natural isomorphism, called the \emph{commutor}, $\sigma_{A,B}:A\otimes B\to B\otimes A$ satisfying
		\begin{enumerate}
			\item Symmetry Axiom: $\sigma_{B,A}\circ\sigma_{A,B} = 1_{A\otimes B}$.
			\item Unit Axiom: $\eta_L^{-1}\circ \eta_R = \sigma_{\mathbb 1,A}$, where $\eta_R:\mathbb{1}\otimes A \to A \leftarrow A\otimes \mathbb{1}:\eta_L$ are the unit isomorphisms of $\mathcal C$.
			\item Cactus Axiom: The following diagram commutes, where $\alpha$ is the associator of $\mathcal C$.
			\[\begin{tikzcd}
				{A\otimes(B\otimes C)} && {A\otimes(C\otimes B)} && {(C\otimes B)\otimes A} \\
				\\
				{(A\otimes B)\otimes C} && {(B\otimes A)\otimes C} && {C\otimes (B\otimes A)}
				\arrow["{1\otimes\sigma_{B,C}}", from=1-1, to=1-3]
				\arrow["{\alpha_{A,B,C}}"', from=1-1, to=3-1]
				\arrow["{\sigma_{A,C\otimes B}}", from=1-3, to=1-5]
				\arrow["{\sigma_{A,B}\otimes 1}"', from=3-1, to=3-3]
				\arrow["{\sigma_{B\otimes A, C}}"', from=3-3, to=3-5]
				\arrow["{\alpha_{C,B,A}}"', from=3-5, to=1-5]
			\end{tikzcd}\]
		\end{enumerate}
	\end{definition}
	
	Given a coboundary category $\mathcal C$ and objects $A_1,\dots, A_n\in \mathcal C$, using the commutor, we can define natural isomorphisms $\sigma_{p,r,q}$ for $1\leq p\leq r<q\leq n$, via
	$$(\sigma_{p,r,q})_{A_1,\dots,A_n} = 1_{A_1\otimes\dots\otimes A_{p-1}}\otimes \sigma_{A_p\otimes \dots\otimes A_r, A_{r+1}\otimes\dots\otimes A_q}\otimes 1_{A_{q+1}\otimes\dots\otimes A_n},$$
	which commute the factors $A_p\otimes \dots\otimes A_r$ over $A_{r+1}\otimes \dots\otimes A_q$.
	These in turn define the so-called ``interval-reversal'' morphisms
	\begin{equation*}
		\begin{split}
			s_{p,q}:A_1\otimes &\dots\otimes A_{p-1}\otimes A_p \otimes A_{p+1} \otimes \dots, \otimes A_{q-1}\otimes A_{q}\otimes A_{q+1}\otimes\dots\otimes A_n \to\\
			&A_1\otimes \dots\otimes A_{p-1}\otimes A_q \otimes A_{q-1} \otimes \dots, \otimes A_{p+1}\otimes A_{p}\otimes A_{q+1}\otimes\dots\otimes A_n
		\end{split}
	\end{equation*}
	defined recursively by $s_{p,p+1}=\sigma_{p,p,p+1}$ and $s_{p,q} = \sigma_{p,p,q}\circ s_{p+1, q}$.
	
	Conversely, given $s_{p,q}$, one can recover the commutor morphisms by noting that $\sigma_{p,r,q} = s_{p,q}\circ s_{p+1, q}\circ s_{p,r}$ (see~\cite{HKam}, Lemma 3] for a proof), and then observing that the commutor $\sigma_{A,B}$ may be thought of simply as $\sigma_{1,1,2}$. 
	
	The morphisms $s_{p,q}$ satisfy the following relations:
	\begin{enumerate}
		\item $s_{p,q}^2 = 1$ for every $1\leq p < q \leq n$.
		\item $s_{p,q}s_{k,l} = s_{k,l}s_{p,q}$ for all $1\leq p < q \leq n$ and $1\leq k < l \leq n$ satisfying $[p,q]\cap[k,l] = \emptyset$.
		\item $s_{p,q}s_{k,l} = s_{p+q-l,p+q-k}s_{p,q}$ for all $1\leq p < q \leq n$ and $1\leq k < l \leq n$ satisfying $[k,l]\subset [p,q]$.
	\end{enumerate}
	These relations may be visualized as follows:
	
	
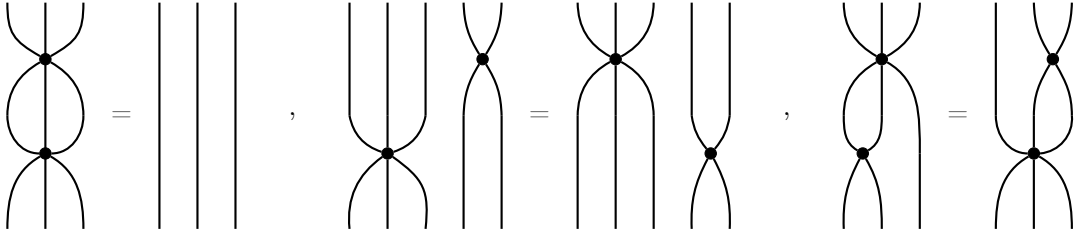
\[\begin{tikzpicture}
	\begin{pgfonlayer}{nodelayer}
		\node [style=none] (0) at (-7, 1.5) {};
		\node [style=none] (1) at (-6.5, 1.5) {};
		\node [style=none] (2) at (-6, 1.5) {};
		\node [style=fullblob] (3) at (-6.5, 0.75) {};
		\node [style=none] (4) at (-7, 0) {};
		\node [style=none] (5) at (-6.5, 0) {};
		\node [style=none] (6) at (-6, 0) {};
		\node [style=fullblob] (7) at (-6.5, -0.5) {};
		\node [style=none] (8) at (-7, -1.5) {};
		\node [style=none] (9) at (-6.5, -1.5) {};
		\node [style=none] (10) at (-6, -1.5) {};
		\node [style=none] (11) at (-5, 1.5) {};
		\node [style=none] (12) at (-4.5, 1.5) {};
		\node [style=none] (13) at (-4, 1.5) {};
		\node [style=none] (14) at (-5, -1.5) {};
		\node [style=none] (15) at (-4.5, -1.5) {};
		\node [style=none] (16) at (-4, -1.5) {};
		\node [style=none] (17) at (-5.5, 0) {};
		\node [style=none] (18) at (-5.5, 0) {=};
		\node [style=none] (19) at (-2.5, 1.5) {};
		\node [style=none] (20) at (-2, 1.5) {};
		\node [style=none] (21) at (-1.5, 1.5) {};
		\node [style=none] (22) at (-2.5, -1.5) {};
		\node [style=none] (23) at (-2, -1.5) {};
		\node [style=none] (24) at (-1.5, -1.5) {};
		\node [style=fullblob] (25) at (-2, -0.5) {};
		\node [style=none] (26) at (-2.5, 0) {};
		\node [style=none] (27) at (-2, 0) {};
		\node [style=none] (28) at (-1.5, 0) {};
		\node [style=none] (29) at (-1, 1.5) {};
		\node [style=none] (30) at (-0.5, 1.5) {};
		\node [style=fullblob] (31) at (-0.75, 0.75) {};
		\node [style=none] (32) at (-1, 0) {};
		\node [style=none] (33) at (-0.5, 0) {};
		\node [style=none] (34) at (-1, -1.5) {};
		\node [style=none] (35) at (-0.5, -1.5) {};
		\node [style=none] (36) at (0, 0) {=};
		\node [style=none] (37) at (0.5, 1.5) {};
		\node [style=none] (38) at (1, 1.5) {};
		\node [style=none] (39) at (1.5, 1.5) {};
		\node [style=none] (40) at (0.5, -1.5) {};
		\node [style=none] (41) at (1, -1.5) {};
		\node [style=none] (42) at (1.5, -1.5) {};
		\node [style=fullblob] (43) at (1, 0.75) {};
		\node [style=none] (44) at (0.5, 0) {};
		\node [style=none] (45) at (1, 0) {};
		\node [style=none] (46) at (1.5, 0) {};
		\node [style=none] (47) at (2, 1.5) {};
		\node [style=none] (48) at (2.5, 1.5) {};
		\node [style=fullblob] (49) at (2.25, -0.5) {};
		\node [style=none] (50) at (2, 0) {};
		\node [style=none] (51) at (2.5, 0) {};
		\node [style=none] (52) at (2, -1.5) {};
		\node [style=none] (53) at (2.5, -1.5) {};
		\node [style=none] (54) at (4, 1.5) {};
		\node [style=none] (55) at (4.5, 1.5) {};
		\node [style=none] (56) at (5, 1.5) {};
		\node [style=none] (57) at (5.5, 0) {=};
		\node [style=none] (58) at (6, 1.5) {};
		\node [style=none] (59) at (6.5, 1.5) {};
		\node [style=none] (60) at (4, 0) {};
		\node [style=none] (61) at (4.5, 0) {};
		\node [style=none] (62) at (5, -0.5) {};
		\node [style=none] (63) at (4, -1.5) {};
		\node [style=none] (64) at (4.5, -1.5) {};
		\node [style=none] (65) at (5, -1.5) {};
		\node [style=none] (66) at (7, 1.5) {};
		\node [style=none] (67) at (6, 0) {};
		\node [style=none] (68) at (6.5, 0) {};
		\node [style=none] (69) at (7, 0) {};
		\node [style=none] (70) at (6, -1.5) {};
		\node [style=none] (71) at (6.5, -1.5) {};
		\node [style=none] (72) at (7, -1.5) {};
		\node [style=fullblob] (73) at (4.5, 0.75) {};
		\node [style=fullblob] (74) at (4.25, -0.5) {};
		\node [style=fullblob] (75) at (6.75, 0.75) {};
		\node [style=fullblob] (76) at (6.5, -0.5) {};
		\node [style=none] (78) at (-3.25, 0) {,};
		\node [style=none] (79) at (3.25, 0) {,};
	\end{pgfonlayer}
	\begin{pgfonlayer}{edgelayer}
		\draw [style=thickstrand] (1.center) to (9.center);
		\draw [style=thickstrand] [in=150, out=-90, looseness=1.25] (0.center) to (3);
		\draw [style=thickstrand] [in=270, out=30, looseness=1.25] (3) to (2.center);
		\draw [style=thickstrand] [in=90, out=-30] (3) to (6.center);
		\draw [style=thickstrand] [in=90, out=-150] (3) to (4.center);
		\draw [style=thickstrand] [in=270, out=180] (7) to (4.center);
		\draw [style=thickstrand] [in=0, out=-90] (6.center) to (7);
		\draw [style=thickstrand] [in=90, out=-30] (7) to (10.center);
		\draw [style=thickstrand] [in=90, out=-150] (7) to (8.center);
		\draw [style=thickstrand] (11.center) to (14.center);
		\draw [style=thickstrand] (12.center) to (15.center);
		\draw [style=thickstrand] (13.center) to (16.center);
		\draw [style=thickstrand] [bend right=15] (29.center) to (31);
		\draw [style=thickstrand] [bend left=15] (30.center) to (31);
		\draw [style=thickstrand] [bend right=15] (31) to (32.center);
		\draw [style=thickstrand] [bend left=15] (31) to (33.center);
		\draw [style=thickstrand] (33.center) to (35.center);
		\draw [style=thickstrand] (32.center) to (34.center);
		\draw [style=thickstrand] (19.center) to (26.center);
		\draw [style=thickstrand] (20.center) to (27.center);
		\draw [style=thickstrand] (21.center) to (28.center);
		\draw [style=thickstrand] [bend left] (28.center) to (25);
		\draw [style=thickstrand] (27.center) to (25);
		\draw [style=thickstrand] [bend right] (26.center) to (25);
		\draw [style=thickstrand] [bend right] (25) to (22.center);
		\draw [style=thickstrand] (23.center) to (25);
		\draw [style=thickstrand] [bend right, looseness=1.25] (24.center) to (25);
		\draw [style=thickstrand] [bend right] (37.center) to (43);
		\draw [style=thickstrand] (38.center) to (43);
		\draw [style=thickstrand] [bend left] (39.center) to (43);
		\draw [style=thickstrand] [bend right] (43) to (44.center);
		\draw [style=thickstrand] (43) to (45.center);
		\draw [style=thickstrand] [bend right] (46.center) to (43);
		\draw [style=thickstrand] (45.center) to (41.center);
		\draw [style=thickstrand] (44.center) to (40.center);
		\draw [style=thickstrand] (46.center) to (42.center);
		\draw [style=thickstrand] [bend left=15] (49) to (50.center);
		\draw [style=thickstrand] [bend right=15] (49) to (51.center);
		\draw [style=thickstrand] [bend right=15] (49) to (52.center);
		\draw [style=thickstrand] [bend right=15] (53.center) to (49);
		\draw [style=thickstrand] (50.center) to (47.center);
		\draw [style=thickstrand] (51.center) to (48.center);
		\draw [style=thickstrand] [in=150, out=-90] (54.center) to (73);
		\draw [style=thickstrand] (55.center) to (73);
		\draw [style=thickstrand] [in=30, out=-90] (56.center) to (73);
		\draw [style=thickstrand] [in=90, out=-90] (73) to (61.center);
		\draw [style=thickstrand] [in=90, out=-150] (73) to (60.center);
		\draw [style=thickstrand] [in=90, out=-30] (73) to (62.center);
		\draw [style=thickstrand] [in=30, out=-90] (61.center) to (74);
		\draw [style=thickstrand] [in=270, out=150] (74) to (60.center);
		\draw [style=thickstrand] [bend right=15] (74) to (63.center);
		\draw [style=thickstrand] [in=300, out=90] (64.center) to (74);
		\draw [style=thickstrand] [bend right=15] (75) to (66.center);
		\draw [style=thickstrand] [bend left=15] (75) to (59.center);
		\draw [style=thickstrand] [in=90, out=-60] (75) to (69.center);
		\draw [style=thickstrand] [in=90, out=-120] (75) to (68.center);
		\draw [style=thickstrand] (58.center) to (67.center);
		\draw [style=thickstrand] [in=180, out=-90] (67.center) to (76);
		\draw [style=thickstrand] (68.center) to (76);
		\draw [style=thickstrand] [bend left=45] (69.center) to (76);
		\draw [style=thickstrand] [in=330, out=90] (72.center) to (76);
		\draw [style=thickstrand] (76) to (71.center);
		\draw [style=thickstrand] [in=90, out=-150] (76) to (70.center);
		\draw [style=thickstrand] (65.center) to (62.center);
	\end{pgfonlayer}
\end{tikzpicture}
\]
	
	\begin{definition}
		The group generated by elements $s_{p,q}$ for $1\leq p < q\leq n$ satisfying the relations (1),(2), and (3) above is called \emph{the $n$-fruit cactus group}, denoted by $J_n$.
	\end{definition}
	
	\setcounter{MaxMatrixCols}{12}
	
	The cactus group admits a natural map $J_n\to S_n$, denoted $\rho\to\widehat{\rho}$ extending $s_{p,q}\mapsto \widehat{s_{p,q}}$, where 
	$$\widehat{s_{p,q}} = \begin{pmatrix}
		1 & \dots & p-1 & p & p+1 & \dots & q-1 & q & q+1 & \dots & n \\
		1 & \dots & p-1 & q & q-1 & \dots & p+1 & p & q+1 & \dots & n \end{pmatrix}\in S_n.$$
	
	The kernel of this map is the \emph{pure cactus group} and is identified with the fundamental group of the Deligne-Mumford compactification $\overline{M_0}^{n+1}(\mathbb R)$ of the moduli space of real genus-0 curves with $n$ market points \cite[Theorem~9]{HKam}.\footnote{The namesake of the cactus group comes from this geometric viewpoint, as the points of the moduli space $\overline{M_0}^{n+1}(\mathbb R)$ look somewhat like cacti from the genus Opuntia (see \cite[Section~3.2]{HKam}).}
	
	So far, we have abused notation by using $s_{p,q}$ to denote generators of the group $J_n$ as well as natural isomorphisms $$(s_{p,q})_{A_1,\dots,A_n}:A_1\otimes \dots \otimes A_n \to A_{\widehat{s_{p,q}}(1)}\otimes \dots \otimes A_{\widehat{s_{p,q}}(n)}.$$
	Whenever confusion is possible, we shall distinguish the two by denoting the latter by $s_{p,q}^{\mathcal C}$.
	
	By expressing an element of $\rho\in J_n$ as a product of the generators $s_{p,q}$, we get a natural isomorphism by composing the corresponding morphisms $s_{p,q}^{\mathcal C}$, which is denoted by 
	$$\tau(\rho; A_1, \dots, A_n):A_1\otimes \dots \otimes A_n \to A_{\widehat{\rho}(1)}\otimes \dots \otimes A_{\widehat{\rho}(n)}.$$
	
	\begin{theorem} \cite[Theorem~7]{HKam} 
		Let $\mathcal C$ be a coboundary category, and let $A_1,\dots, A_n\in \mathcal C$. If $\rho\in J_n$, the natural isomorphisms $\tau(\rho; A_1, \dots, A_n)$ defined above satisfy
		$$\tau(\rho'; A_{\widehat{\rho}(1)}, \dots, A_{\widehat{\rho}(n)})\circ\tau(\rho; A_1, \dots, A_n)=\tau(\rho\rho'; A_1, \dots, A_n).$$ Moreover, these natural isomorphisms are exactly the ones which can be generated by using the commutor of $\mathcal C$.
	\end{theorem}
	
	We conclude this section with the definition of a \emph{coboundary functor}, following \cite[Definition~23.3.1]{Ya}.
	\begin{definition} \label{CobFunctor}
		Let $\mathcal C$ and $\mathcal D$ be two coboundary categories with commutors $\sigma^{\mathcal C}$ and $\sigma^{\mathcal D}$, respectively. 
		Given a monoidal functor $F:\mathcal C \to \mathcal D$ with structure isomorphisms $J_{A,B}:F(A\otimes B)\to F(A)\otimes F(B)$. We say $F$ is a \emph{coboundary functor} if the following diagram commutes for every objects $A$ and $B$ in $\mathcal C$.
		\[\begin{tikzcd}
			{F(A\otimes B)} && {F(A)\otimes F(B)} \\
			\\
			{F(B\otimes A)} && {F(B)\otimes F(A)}
			\arrow["{J_{A,B}}", from=1-1, to=1-3]
			\arrow["{F\left(\sigma^{\mathcal C}_{A,B}\right)}"', from=1-1, to=3-1]
			\arrow["{\sigma^{\mathcal D}_{F(A),F(B)}}", from=1-3, to=3-3]
			\arrow["{J_{B,A}}"', from=3-1, to=3-3]
		\end{tikzcd}\]
	\end{definition}
	
	\begin{remark}
		It is easy to see that a coboundary structure on a category $\mathcal C$ is the same data as a monoidal structure $\sigma_{X,Y}$ on the identity functor $\on{id}:\mathcal C \to \mathcal{C}^{\on{op}}$ (where $\mathcal{C}^{\on{op}}$ is the same category with the opposite tensor product $X\otimes^{\on{op}}Y:=Y\otimes X$) satisying $\sigma_{Y,X}\circ \sigma_{X,Y} =\on{id}_{X\otimes Y}$ \cite[8.3.25]{EGNO}.
		From that perspective, our definition of a coboundary functor simply means a functor $F:\mathcal C\to \mathcal D$ compatible with this data, i.e. the following diagram commutes.
		\[\begin{tikzcd}
			{\mathcal C} & {\mathcal D} \\
			{\mathcal C^{\on{op}}} & {\mathcal D^{\on{op}}}
			\arrow["{(F,J)}", from=1-1, to=1-2]
			\arrow["{(\on{id}, \sigma^\mathcal C)}"', from=1-1, to=2-1]
			\arrow["{(\on{id}, \sigma^\mathcal D)}", from=1-2, to=2-2]
			\arrow["{(F,J^{\on{op}})}"', from=2-1, to=2-2]
		\end{tikzcd}\]
	\end{remark}
	
	\subsection{$\mathfrak{g}\mathbf{-Crys}$ as a Coboundary Category}
	
	Let $\theta:I\to I$ be the Dynkin diagram automorphism such that $\alpha_{\theta(i)} = -w_0 \cdot \alpha_i$, where $w_0$ is the longest element in the Weyl group of $\mathfrak{g}$ acting on simple roots by reflections.
	
	Given a crystal $B_\lambda$ of highest weight $\lambda\in \Lambda_+$, let $\overline{B_\lambda}$ be the crystal with underlying set $\{\overline{b}:b\in B_\lambda\}$ and where crystal structure is given by
	$$e_i(\overline{b}) = \overline{f_{\theta(i)}b}, \quad\quad f_i(\overline{b}) = \overline{e_{\theta(i)}b}, \quad\quad  \mathrm{wt}(\overline{b}) = \overline{w_0(b)}.$$
	
	Henriques and Kamnitzer have shown that $\overline{B_\lambda}$ is also a highest weight crystal of heighest weight $\lambda$ \cite[Lemma~2]{HKam}. Thus, by Schur's Lemma, we have a crystal isomorphism $\overline{B_\lambda}\to B_\lambda$. 
	Composing this isomorphism with the \emph{map of sets} (which is not a morphism of crystals) $B_\lambda\to \overline{B_\lambda}$ given by $b\mapsto \overline{b}$, we get a map of sets $\xi_{B_\lambda}:B_\lambda\to B_\lambda$.

	Next, one extends the set-theoretic maps $\xi_{B_\lambda}$ to any crystal (not necessarily highest weight) by taking $\xi_B: B\to B$ on a crystal $B$ to be given locally on each connected component $B_\lambda$ by $\xi_{B_\lambda}$.
	
	\begin{theorem} \cite[Theorem~6]{HKam}
		The commutor maps $\sigma_{A,B}:A\otimes B \to B\otimes A$ defined by $$(a,b)\mapsto\xi_{B\otimes A}(\xi_B(b),\xi_A(a))$$ is an isomorphism, natural in $A$ and $B$, endowing the category $\mathfrak{g}\mathbf{-Crys}$ with the structure of a coboundary category.
	\end{theorem}
	
	\begin{example}
		For $\mathfrak{g=sl}_2$, crystal graphs of $B_\lambda$ are path graphs with vertex set $\{b_k: k=0,1,\dots,\lambda\}$ and arrows $b_k\to b_{k+1}$ for $k=0,1,\dots, \lambda-1$. 
		The graph of $\overline{B_\lambda}$ is given by reversing the direction of all arrows in $B_\lambda$.
		The map $\xi_{B_\lambda}$ is thus given by $b_k\mapsto b_{\lambda-k}$.
	\end{example}
	
	It is natural to wonder whether the coboundary structure on the category $\mathfrak{g}\mathbf{-Crys}$ comes from a coboundary structure in the category $\mathcal U_q(\mathfrak{g})$-\textbf{Mod} of finite-dimensional representations of the quantum group of $\mathfrak{g}$.
	This is indeed the case, where the commutor for the representation category of a quantum group is defined via an automorphism $\xi:{U}_q(\mathfrak{g}) \to {U}_q(\mathfrak{g})$ resembling the definition of $\xi$ above (see \cite[Section~2.4]{HKam} for details).

	\subsection{The Temperley--Lieb Category}
	
For the remainder of this document, let $\Bbbk$ be a field.	
	
	\begin{definition}\label{def:renormalization}
		Let $\widetilde{\mathcal{TL}}_{q}(\Bbbk)$ be the strict $\Bbbk[q,q^{-1}]$-linear monoidal category defined as follows: 
		\begin{itemize}
			\item $\on{Ob}\widetilde{\mathcal{TL}}_{q}(\Bbbk) = \mathbb{N} = \setj{\obj{0},\obj{1},\obj{2},\ldots}$, with $\mathbb{1}_{\widetilde{\mathcal{TL}}_{q}(\Bbbk)} = \obj{0}$ and $\obj{m} \otimes \obj{n} = \obj{m+n}$;
			\item the morphisms of $\widetilde{\mathcal{TL}}_{q}(\Bbbk)$ are generated by $\on{cup}\in \on{Hom}_{\widetilde{\mathcal{TL}}_{q}(\Bbbk)}(\obj{0},\obj{2})$ and $\on{cap}\in \on{Hom}_{\widetilde{\mathcal{TL}}_{q}(\Bbbk)}(\obj{2},\obj{0})$ which we depict by, and identify with, string diagrams
			\[
			\raisebox{4pt}{$\on{cup}=\; $}
			\begin{tikzpicture}
	\begin{pgfonlayer}{nodelayer}
		\node [style=none] (0) at (-0.5, 0.5) {};
		\node [style=none] (1) at (0, 0) {};
		\node [style=none] (2) at (0.5, 0.5) {};
	\end{pgfonlayer}
	\begin{pgfonlayer}{edgelayer}
		\draw [style=thickstrand, bend right=45] (0.center) to (1.center);
		\draw [style=thickstrand, bend left=45] (2.center) to (1.center);
	\end{pgfonlayer}
\end{tikzpicture}\quad
			\raisebox{4pt}{ and 
				$\on{cap}=\;$ }
			\begin{tikzpicture}
	\begin{pgfonlayer}{nodelayer}
		\node [style=none] (0) at (-0.5, 0) {};
		\node [style=none] (1) at (0, 0.5) {};
		\node [style=none] (2) at (0.5, 0) {};
	\end{pgfonlayer}
	\begin{pgfonlayer}{edgelayer}
		\draw [style=thickstrand, bend left=45] (0.center) to (1.center);
		\draw [style=thickstrand, bend right=45] (2.center) to (1.center);
	\end{pgfonlayer}
\end{tikzpicture}
			\]
			\item the generators $\on{cup}$ and $\on{cap}$ satisfy the following relations:
			\[
			\begin{tikzpicture}
	\begin{pgfonlayer}{nodelayer}
		\node [style=none] (0) at (-1.5, 0.75) {};
		\node [style=none] (1) at (-1.25, 1) {};
		\node [style=none] (2) at (-1, 0.75) {};
		\node [style=none] (3) at (-2, 0.5) {};
		\node [style=none] (4) at (-1.75, 0.25) {};
		\node [style=none] (5) at (-1.5, 0.5) {};
		\node [style=none] (6) at (-1, 0) {};
		\node [style=none] (7) at (-2, 1.25) {};
		\node [style=none] (8) at (-0.5, 0.75) {$=$};
		\node [style=none] (9) at (0, 0) {};
		\node [style=none] (10) at (0, 1.25) {};
		\node [style=none] (11) at (1.5, 0.75) {};
		\node [style=none] (12) at (1.25, 1) {};
		\node [style=none] (13) at (1, 0.75) {};
		\node [style=none] (14) at (2, 0.5) {};
		\node [style=none] (15) at (1.75, 0.25) {};
		\node [style=none] (16) at (1.5, 0.5) {};
		\node [style=none] (17) at (1, 0) {};
		\node [style=none] (18) at (2, 1.25) {};
		\node [style=none] (19) at (0.5, 0.75) {$=$};
	\end{pgfonlayer}
	\begin{pgfonlayer}{edgelayer}
		\draw [style=thickstrand, bend left=45] (0.center) to (1.center);
		\draw [style=thickstrand, bend right=45] (2.center) to (1.center);
		\draw [style=thickstrand, bend right=45] (3.center) to (4.center);
		\draw [style=thickstrand, bend left=45] (5.center) to (4.center);
		\draw [style=thickstrand] (5.center) to (0.center);
		\draw [style=thickstrand] (6.center) to (2.center);
		\draw [style=thickstrand] (3.center) to (7.center);
		\draw [style=thickstrand] (9.center) to (10.center);
		\draw [style=thickstrand, bend right=45] (11.center) to (12.center);
		\draw [style=thickstrand, bend left=45] (13.center) to (12.center);
		\draw [style=thickstrand, bend left=45] (14.center) to (15.center);
		\draw [style=thickstrand, bend right=45] (16.center) to (15.center);
		\draw [style=thickstrand] (16.center) to (11.center);
		\draw [style=thickstrand] (17.center) to (13.center);
		\draw [style=thickstrand] (14.center) to (18.center);
	\end{pgfonlayer}
\end{tikzpicture}
			\]
			and
			\[
\begin{tikzpicture}
	\begin{pgfonlayer}{nodelayer}
		\node [style=none] (1) at (0, 1) {};
		\node [style=none] (3) at (-0.5, 0.5) {};
		\node [style=none] (4) at (0, 0) {};
		\node [style=none] (5) at (0.5, 0.5) {};
		\node [style=none] (6) at (1.5, 0.5) {$=[2]_{q}\on{id}_{0}$};
	\end{pgfonlayer}
	\begin{pgfonlayer}{edgelayer}
		\draw [style=thickstrand, bend right=45] (3.center) to (4.center);
		\draw [style=thickstrand, bend left=45] (5.center) to (4.center);
		\draw [style=thickstrand, bend right=45] (5.center) to (1.center);
		\draw [style=thickstrand, bend left=45] (3.center) to (1.center);
	\end{pgfonlayer}
\end{tikzpicture}
			\]
		\end{itemize}
	\end{definition}
	
	\begin{remark}
		The object $\obj{1}$ is self-dual, where the morphisms $\on{cup}$ and $\on{cap}$ provide the unit and the counit of an adjunction $(\obj{1},\obj{1},\on{cup},\on{cap})$. The category $\widetilde{\mathcal{TL}}_{q}(\Bbbk)$ in fact, is a well-studied category, see e.g. \cite{EO}, \cite{Tu}; in particular, it is braided and the dimension of $\obj{1}$ (i.e. the trace of its identity morphism) equals $[2]_{q}$. Indeed, $\widetilde{\mathcal{TL}}_{q}(\Bbbk)$ is the {\it free} strict $\Bbbk[q,q^{-1}]$-linear monoidal category on a self-dual object of dimension $[2]_{q}$, as is explained, and heavily used in \cite{EO}. 
	\end{remark}

	The following result is folklore and, at least in the non-quantum setting, dates back to \cite{RTW}. A much more general version of it is proven in detail in \cite[Theorem~2.58]{E1}. 
    Let $\mathbf{Fund}(\uqsl)$ be the full monoidal subcategory of $\mathbf{Rep}(\uqsl)$ given by tensor powers of the fundamental representation. 
	
	\begin{theorem}\label{uqsl2TL}
		Let $q \in \Bbbk^{\times}$. There is a monoidal equivalence $\mathcal{TL}_{q}(\Bbbk) \simeq \mathbf{Fund}(\uqsl)$. 
	\end{theorem}
	We now briefly formulate some well-known facts about additive and Karoubi envelopes, in a language closer to enriched category theory.
        We say that a $\Bbbk$-linear category $\mathcal{A}$ is {\it Cauchy complete} if it admits finite direct sums and is idempotent complete. Indeed, finite direct sums and retracts are precisely the {\it absolute colimits} for $\Bbbk$-linear functors, i.e. the colimits preserved by any such functor. 
	Given a $\Bbbk$-linear category $\mathcal{A}$ let $\mathcal{A}^{\mathtt{C}}$ denote the {\it Cauchy completion of $\mathcal{A}$}, given by the subcategory of $\mathbf{Fun}_{\Bbbk}(\mathcal{A},\mathbf{Vec}_{\Bbbk})$ consisting of absolute colimits of representable presheaves. It can be identified with the Karoubi envelope of the additive envelope of $\mathcal{A}$ - see \cite{BD}, \cite{Ri} for detailed accounts. The restricted Yoneda embedding yields a full and faithful functor $\iota: \mathcal{A} \hookrightarrow \mathcal{A}^{\mathtt{C}}$, such that for a Cauchy complete $\mathcal{E}$, the functor 
    \begin{equation}\label{eq:CauchyEquiv}
    \mathbf{Fun}_{\Bbbk}(\mathcal{A}^{\mathtt{C}},\mathcal{E})\xrightarrow{- \circ \iota} \mathbf{Fun}_{\Bbbk}(\mathcal{A,E})
    \end{equation}
     is an equivalence. A monoidal structure on $\mathcal{A}$ extends essentially uniquely to a monoidal structure on $\mathcal{A}^{\mathtt{C}}$ via Day convolution, and so $\iota$ is strong monoidal, and the equivalence~\eqref{eq:CauchyEquiv} lifts to an equivalence
     \begin{equation}\label{eq:MonCauchy}
    \mathbf{MonFun}_{\Bbbk}(\mathcal{A}^{\mathtt{C}},\mathcal{E})\xrightarrow{- \circ \iota} \mathbf{MonFun}_{\Bbbk}(\mathcal{A,E})
    \end{equation}
    between categories of (strong) monoidal functors to a Cauchy complete monoidal category $\mathcal{E}$.
    
    In particular, for two (monoidal) categories $\mathcal{A,B}$ such that $\mathcal{A}^{\mathtt{C}} \simeq \mathcal{B}^{\mathtt{C}}$, we have $\mathbf{Fun}_{\Bbbk}(\mathcal{A,E})\simeq \mathbf{Fun}_{\Bbbk}(\mathcal{B,E})$ for any (monoidal) Cauchy complete $\mathcal{E}$, and a similar equivalence for monoidal functors. In fact, already the case $\mathcal{E} = \mathbf{Vec}_{\Bbbk}$ suffices for the converse. These equivalent conditions are equivalent to $\mathcal{A,B}$ being equivalent in the bicategory of profunctors.
     
	\begin{corollary}
		Since any finite-dimensional $\uqsl$-module is a direct summand of a tensor power of the fundamental representation, we find that $\mathbf{Rep}(\uqsl) \simeq \mathbf{Fund}(\uqsl)^{\mathtt{C}} \simeq \mathcal{TL}_{q}(\Bbbk)^{\mathtt{C}}$ by \autoref{uqsl2TL}. 
	\end{corollary}
	
	\section{The Crystal Temperley--Lieb Category}\label{sec:crysTL}
	
	\subsection{Renormalizing the Temperley--Lieb Category}
	We now define a renormalization of $\widetilde{\mathcal{TL}}_{q}(\Bbbk)$, which extends to a $\Bbbk[q]$-linear category. This already mirrors the development of $\mathfrak{g}$-crystals: we introduce a renormalization which allows us to set $q = 0$.
	\begin{definition}
		Let $\overline{\mathcal{TL}}_{q}(\Bbbk)$ be the strict $\Bbbk[q]$-linear monoidal category defined as follows: 
		\begin{itemize}
			\item $\on{Ob}\overline{\mathcal{TL}}_{q}(\Bbbk) = \mathbb{N} = \setj{\obj{0},\obj{1},\obj{2},\ldots}$, with $\mathbb{1}_{\overline{\mathcal{TL}}_{q}(\Bbbk)} = \obj{0}$ and $\obj{m} \otimes \obj{n} = \obj{m+n}$;
			\item the morphisms of $\overline{\mathcal{TL}}_{q}(\Bbbk)$ are generated by $\on{cup}\in \on{Hom}_{\overline{\mathcal{TL}}_{q}(\Bbbk)}(\obj{0},\obj{2})$ and $\on{cap}\in \on{Hom}_{\overline{\mathcal{TL}}_{q}(\Bbbk)}(\obj{2},\obj{0})$, which we depict by, and identify with, string diagrams
			\[
			\raisebox{4pt}{$\on{cup}=\; $}
			\begin{tikzpicture}
	\begin{pgfonlayer}{nodelayer}
		\node [style=none] (0) at (-0.5, 0.5) {};
		\node [style=none] (1) at (0, 0) {};
		\node [style=none] (2) at (0.5, 0.5) {};
	\end{pgfonlayer}
	\begin{pgfonlayer}{edgelayer}
		\draw [style=thickstrand, bend right=45] (0.center) to (1.center);
		\draw [style=thickstrand, bend left=45] (2.center) to (1.center);
	\end{pgfonlayer}
\end{tikzpicture}\quad
			\raisebox{4pt}{ and 
				$\on{cap}=\;$ }
			\begin{tikzpicture}
	\begin{pgfonlayer}{nodelayer}
		\node [style=none] (0) at (-0.5, 0) {};
		\node [style=none] (1) at (0, 0.5) {};
		\node [style=none] (2) at (0.5, 0) {};
	\end{pgfonlayer}
	\begin{pgfonlayer}{edgelayer}
		\draw [style=thickstrand, bend left=45] (0.center) to (1.center);
		\draw [style=thickstrand, bend right=45] (2.center) to (1.center);
	\end{pgfonlayer}
\end{tikzpicture}
			\]
			\item the generators $\on{cup}$ and $\on{cap}$ satisfy the following relations:
			\[
			\begin{tikzpicture}
	\begin{pgfonlayer}{nodelayer}
		\node [style=none] (0) at (-1.5, 0.75) {};
		\node [style=none] (1) at (-1.25, 1) {};
		\node [style=none] (2) at (-1, 0.75) {};
		\node [style=none] (3) at (-2, 0.5) {};
		\node [style=none] (4) at (-1.75, 0.25) {};
		\node [style=none] (5) at (-1.5, 0.5) {};
		\node [style=none] (6) at (-1, 0) {};
		\node [style=none] (7) at (-2, 1.25) {};
		\node [style=none] (9) at (0, 0) {};
		\node [style=none] (10) at (0, 1.25) {};
		\node [style=none] (11) at (1.5, 0.75) {};
		\node [style=none] (12) at (1.25, 1) {};
		\node [style=none] (13) at (1, 0.75) {};
		\node [style=none] (14) at (2, 0.5) {};
		\node [style=none] (15) at (1.75, 0.25) {};
		\node [style=none] (16) at (1.5, 0.5) {};
		\node [style=none] (17) at (1, 0) {};
		\node [style=none] (18) at (2, 1.25) {};
		\node [style=none] (19) at (0.5, 0.75) {$=$};
		\node [style=none] (20) at (-0.5, 0.75) {$=q\cdot$};
	\end{pgfonlayer}
	\begin{pgfonlayer}{edgelayer}
		\draw [style=thickstrand, bend left=45] (0.center) to (1.center);
		\draw [style=thickstrand, bend right=45] (2.center) to (1.center);
		\draw [style=thickstrand, bend right=45] (3.center) to (4.center);
		\draw [style=thickstrand, bend left=45] (5.center) to (4.center);
		\draw [style=thickstrand] (5.center) to (0.center);
		\draw [style=thickstrand] (6.center) to (2.center);
		\draw [style=thickstrand] (3.center) to (7.center);
		\draw [style=thickstrand] (9.center) to (10.center);
		\draw [style=thickstrand, bend right=45] (11.center) to (12.center);
		\draw [style=thickstrand, bend left=45] (13.center) to (12.center);
		\draw [style=thickstrand, bend left=45] (14.center) to (15.center);
		\draw [style=thickstrand, bend right=45] (16.center) to (15.center);
		\draw [style=thickstrand] (16.center) to (11.center);
		\draw [style=thickstrand] (17.center) to (13.center);
		\draw [style=thickstrand] (14.center) to (18.center);
	\end{pgfonlayer}
\end{tikzpicture}
			\]
			and
			\[\begin{tikzpicture}
	\begin{pgfonlayer}{nodelayer}
		\node [style=none] (0) at (0, 1) {};
		\node [style=none] (1) at (-0.5, 0.5) {};
		\node [style=none] (2) at (0, 0) {};
		\node [style=none] (3) at (0.5, 0.5) {};
		\node [style=none] (4) at (1.75, 0.5) {$=(q^{2}+1)\on{id}_{0}.$};
	\end{pgfonlayer}
	\begin{pgfonlayer}{edgelayer}
		\draw [style=thickstrand, bend right=45] (1.center) to (2.center);
		\draw [style=thickstrand, bend left=45] (3.center) to (2.center);
		\draw [style=thickstrand, bend right=45] (3.center) to (0.center);
		\draw [style=thickstrand, bend left=45] (1.center) to (0.center);
	\end{pgfonlayer}
\end{tikzpicture}
\]
		\end{itemize}
	\end{definition}
	
	Observe that the interchange law for monoidal categories implies that two diagrams which differ only by a rectilinear isotopy and represent morphisms in $\widetilde{\mathcal{TL}}_{q}(\Bbbk)$ necessarily represent the same morphism. 
    The same holds for $\overline{\mathcal{TL}}_{q}(\Bbbk)$; in both cases, this is an immediate consequence of the interchange law for monoidal categories.
	
	\begin{definition}
		Given an {\it invertible} scalar $a \in \Bbbk^{\times}$, the strict monoidal $\Bbbk$-linear category $\widetilde{\mathcal{TL}}_{a}(\Bbbk)$ is defined by substituting $a$ in place of $q$, for every occurence of the latter.
		
		Similarly, given {\it any} scalar $a \in \Bbbk$, the strict monoidal $\Bbbk$-linear category $\overline{\mathcal{TL}}_{a}(\Bbbk)$ is defined by substituting $a$ in place of $q$, for every occurence of the latter.
	\end{definition}
	
	\begin{remark}
		The $\Bbbk$-algebra homomorphism $\on{ev}_{a}: \Bbbk[q,q^{-1}] \rightarrow \Bbbk$ determined by sending $q$ to $a$ yields the induction functor $\Bbbk[q,q^{-1}]\!\on{-Mod} \rightarrow \mathbf{Vec}_{\Bbbk}$, which in turn defines a monoidal $2$-functor $\on{Ev}_{a}: \mathbf{Cat}_{\Bbbk[q,q^{-1}]} \rightarrow \mathbf{Cat}_{\Bbbk}$. 
        The monoidal category $\widetilde{\mathcal{TL}}_{a}(\Bbbk)$ is obtained as $\on{Ev}_{a}\big(\widetilde{\mathcal{TL}}_{q}(\Bbbk)\big)$. Similar considerations apply to $\overline{\mathcal{TL}}_{a}(\Bbbk)$.
	\end{remark}
	
	\begin{lemma}\label{NaFunctor}
		For $a \in \Bbbk^{\times}$, there is a strict monoidal functor $N_{a}: \widetilde{\mathcal{TL}}_{a}(\Bbbk) \rightarrow \overline{\mathcal{TL}}_{a}(\Bbbk)$, which is identity on objects, and is determined on morphisms by the assignments $N_{a}(\on{cup}) = \on{cup}$ and $N_{a}(\on{cap}) = \frac{1}{a}\cdot \on{cap}$.
		
		Similarly, there is a strict monoidal functor $D_{a}: \overline{\mathcal{TL}}_{a}(\Bbbk) \rightarrow \widetilde{\mathcal{TL}}_{a}(\Bbbk)$, which is identity on objects, and is determined on morphisms by the assignments $D_{a}(\on{cup}) = \on{cup}$ and $D_{a}(\on{cap}) = a\cdot \on{cap}$.
	\end{lemma}
	
	\begin{proof}
		We verify that that $D_{a}$ is well-defined by checking if it respects the defining relations for  are satisfied: 
		\[\begin{tikzpicture}
	\begin{pgfonlayer}{nodelayer}
		\node [style=none] (1) at (0.5, 1) {};
		\node [style=none] (3) at (0, 0.5) {};
		\node [style=none] (5) at (0, 0.5) {};
		\node [style=none] (6) at (-0.5, 0.5) {};
		\node [style=none] (7) at (-0.25, 0.25) {};
		\node [style=none] (8) at (0.5, 0) {};
		\node [style=none] (9) at (-0.5, 1.5) {};
		\node [style=none] (11) at (2.5, 0.5) {};
		\node [style=none] (12) at (2, 0.5) {};
		\node [style=none] (14) at (2, 0.5) {};
		\node [style=none] (15) at (1.5, 0.5) {};
		\node [style=none] (16) at (1.75, 0.25) {};
		\node [style=none] (17) at (2.5, 0) {};
		\node [style=none] (18) at (2.5, 1) {};
		\node [style=none] (19) at (2, 1) {};
		\node [style=none] (20) at (2.25, 1.25) {};
		\node [style=none] (21) at (2, 1) {};
		\node [style=none] (22) at (1.5, 1) {};
		\node [style=none] (23) at (1.5, 1.5) {};
		\node [style=none] (24) at (0.5, 1) {};
		\node [style=none] (25) at (0, 1) {};
		\node [style=none] (26) at (0.25, 1.25) {};
		\node [style=none] (27) at (0, 1) {};
		\node [style=none] (28) at (1, 0.75) {$=$};
		\node [style=none] (29) at (2, 0.75) {$\circ$};
		\node [style=none] (30) at (3.25, 1.25) {$\xmapsto{N_{a}}$};
		\node [style=none] (31) at (5.5, 1) {};
		\node [style=none] (32) at (5, 1) {};
		\node [style=none] (33) at (5.25, 1.25) {};
		\node [style=none] (34) at (5, 1) {};
		\node [style=none] (35) at (4.5, 1) {};
		\node [style=none] (36) at (4.5, 1.5) {};
		\node [style=none] (37) at (4, 1.25) {$\frac{1}{a}\cdot$};
		\node [style=none] (44) at (5.5, 0.5) {};
		\node [style=none] (45) at (5, 0.5) {};
		\node [style=none] (46) at (5, 0.5) {};
		\node [style=none] (47) at (4.5, 0.5) {};
		\node [style=none] (48) at (4.75, 0.25) {};
		\node [style=none] (49) at (5.5, 0) {};
		\node [style=none] (50) at (5, 0.75) {$\circ$};
		\node [style=none] (51) at (3.25, 0.25) {$\xmapsto{N_{a}}$};
		\node [style=none] (52) at (6.25, 0.75) {$= \frac{1}{a}\cdot $};
		\node [style=none] (53) at (8, 1) {};
		\node [style=none] (54) at (7.5, 0.5) {};
		\node [style=none] (55) at (7.5, 0.5) {};
		\node [style=none] (56) at (7, 0.5) {};
		\node [style=none] (57) at (7.25, 0.25) {};
		\node [style=none] (58) at (8, 0) {};
		\node [style=none] (59) at (7, 1.5) {};
		\node [style=none] (60) at (8, 1) {};
		\node [style=none] (61) at (7.5, 1) {};
		\node [style=none] (62) at (7.75, 1.25) {};
		\node [style=none] (63) at (7.5, 1) {};
		\node [style=none] (64) at (8.75, 0.75) {$=$};
		\node [style=none] (65) at (9.25, 1.5) {};
		\node [style=none] (66) at (9.25, 0) {};
		\node [style=none] (67) at (9.75, 0.75) {$\stackrel{N_{a}}{\mapsfrom}$};
		\node [style=none] (68) at (10.5, 1.5) {};
		\node [style=none] (69) at (10.5, 0) {};
	\end{pgfonlayer}
	\begin{pgfonlayer}{edgelayer}
		\draw [style=thickstrand, bend left=45] (7.center) to (6.center);
		\draw [style=thickstrand, bend right=45] (7.center) to (5.center);
		\draw [style=thickstrand] (9.center) to (6.center);
		\draw [style=thickstrand] (1.center) to (8.center);
		\draw [style=thickstrand, bend left=45] (16.center) to (15.center);
		\draw [style=thickstrand, bend right=45] (16.center) to (14.center);
		\draw [style=thickstrand] (11.center) to (17.center);
		\draw [style=thickstrand, bend right=45] (20.center) to (19.center);
		\draw [style=thickstrand, bend left=45] (20.center) to (18.center);
		\draw [style=thickstrand] (23.center) to (22.center);
		\draw [style=thickstrand, bend right=45] (26.center) to (25.center);
		\draw [style=thickstrand, bend left=45] (26.center) to (24.center);
		\draw [style=thickstrand] (27.center) to (5.center);
		\draw [style=thickstrand, bend right=45] (33.center) to (32.center);
		\draw [style=thickstrand, bend left=45] (33.center) to (31.center);
		\draw [style=thickstrand] (36.center) to (35.center);
		\draw [style=thickstrand, bend left=45] (48.center) to (47.center);
		\draw [style=thickstrand, bend right=45] (48.center) to (46.center);
		\draw [style=thickstrand] (44.center) to (49.center);
		\draw [style=thickstrand, bend left=45] (57.center) to (56.center);
		\draw [style=thickstrand, bend right=45] (57.center) to (55.center);
		\draw [style=thickstrand] (59.center) to (56.center);
		\draw [style=thickstrand] (53.center) to (58.center);
		\draw [style=thickstrand, bend right=45] (62.center) to (61.center);
		\draw [style=thickstrand, bend left=45] (62.center) to (60.center);
		\draw [style=thickstrand] (63.center) to (55.center);
		\draw [style=thickstrand] (65.center) to (66.center);
		\draw [style=thickstrand] (68.center) to (69.center);
	\end{pgfonlayer}
\end{tikzpicture}\]
		and
		\[\begin{tikzpicture}
	\begin{pgfonlayer}{nodelayer}
		\node [style=none] (0) at (0, 1.25) {};
		\node [style=none] (1) at (-0.5, 0.75) {};
		\node [style=none] (2) at (0, 0.25) {};
		\node [style=none] (3) at (0.5, 0.75) {};
		\node [style=none] (5) at (1.5, 0.5) {};
		\node [style=none] (6) at (2, 0) {};
		\node [style=none] (7) at (2.5, 0.5) {};
		\node [style=none] (8) at (1.5, 1) {};
		\node [style=none] (9) at (2, 1.5) {};
		\node [style=none] (10) at (2.5, 1) {};
		\node [style=none] (11) at (3.5, 0.5) {};
		\node [style=none] (12) at (4, 0) {};
		\node [style=none] (13) at (4.5, 0.5) {};
		\node [style=none] (14) at (3.5, 1) {};
		\node [style=none] (15) at (4, 1.5) {};
		\node [style=none] (16) at (4.5, 1) {};
		\node [style=none] (17) at (4, 0.75) {$\circ$};
		\node [style=none] (18) at (1, 0.75) {$=$};
		\node [style=none] (19) at (3, 0.75) {$=$};
		\node [style=none] (20) at (5.5, 1.25) {$\xmapsto{N_{a}}$};
		\node [style=none] (22) at (5.5, 0.25) {$\xmapsto{N_{a}}$};
		\node [style=none] (23) at (6.75, 0.5) {};
		\node [style=none] (24) at (7.25, 0) {};
		\node [style=none] (25) at (7.75, 0.5) {};
		\node [style=none] (26) at (6.75, 1) {};
		\node [style=none] (27) at (7.25, 1.5) {};
		\node [style=none] (28) at (7.75, 1) {};
		\node [style=none] (29) at (7.25, 0.75) {$\circ$};
		\node [style=none] (30) at (11.25, 0.75) {$= \frac{1}{a}(a^{2}+1)\on{id}_{0} = [2]_{a}\on{id}_{0}\stackrel{N_{a}}{\mapsfrom}[2]_{a}\on{id}_{0}$.};
		\node [style=none] (31) at (6.25, 1.25) {$\frac{1}{a} \cdot$};
	\end{pgfonlayer}
	\begin{pgfonlayer}{edgelayer}
		\draw [style=thickstrand, bend right=45] (1.center) to (2.center);
		\draw [style=thickstrand, bend left=45] (3.center) to (2.center);
		\draw [style=thickstrand, bend right=45] (3.center) to (0.center);
		\draw [style=thickstrand, bend left=45] (1.center) to (0.center);
		\draw [style=thickstrand, bend right=45] (5.center) to (6.center);
		\draw [style=thickstrand, bend left=45] (7.center) to (6.center);
		\draw [style=thickstrand, bend left=45] (8.center) to (9.center);
		\draw [style=thickstrand, bend right=45] (10.center) to (9.center);
		\draw [style=thickstrand] (8.center) to (5.center);
		\draw [style=thickstrand] (10.center) to (7.center);
		\draw [style=thickstrand, bend right=45] (11.center) to (12.center);
		\draw [style=thickstrand, bend left=45] (13.center) to (12.center);
		\draw [style=thickstrand, bend left=45] (14.center) to (15.center);
		\draw [style=thickstrand, bend right=45] (16.center) to (15.center);
		\draw [style=thickstrand, bend right=45] (23.center) to (24.center);
		\draw [style=thickstrand, bend left=45] (25.center) to (24.center);
		\draw [style=thickstrand, bend left=45] (26.center) to (27.center);
		\draw [style=thickstrand, bend right=45] (28.center) to (27.center);
	\end{pgfonlayer}
\end{tikzpicture}
\]
		
		The proof of well-definedness of $N_{a}$ is analogous.
	\end{proof}
	
	\begin{proposition}
		The functors $D_{a},N_{a}$ are mutually inverse strict monoidal isomorphisms of categories.
	\end{proposition}
	
	\begin{proof}
       	The claim follows from 
        \[
        (D_{a}\circ N_{a})(\on{cup}) = D_{a}(\on{cup}) = \on{cup}\text{ and } (D_{a}\circ N_{a})(\on{cup}) = D_{a}(\frac{1}{a}\on{cup}) = \frac{1}{a} D_{a}(\on{cup}) = \frac{1}{a} a \on{cup} = \on{cup},
        \]
        and a similar calculation for $N_{a} \circ D_{a}$. 
	\end{proof}
	
	\begin{notation}
		For $a \in \Bbbk^{\times}$, we denote by $\mathcal{TL}_{a}(\Bbbk)$ the unique up to isomorphism category given by the common value of $\overline{\mathcal{TL}}_{a}(\Bbbk)$ and $\widetilde{\mathcal{TL}}_{q}(\Bbbk)$. 
        Further, we let $\mathcal{TL}_{0}(\Bbbk) := \overline{\mathcal{TL}}_{0}(\Bbbk)$ and $\mathbf{CrysTL}:=\mathcal{TL}_0(\Bbbk)^{\mathtt{C}}$, its Cauchy completion.
	\end{notation}
	
	We now establish a basis theorem for $\mathcal{TL}_{0}(\Bbbk)$. While the basis is given by Temperley--Lieb diagrams just like for $\mathcal{TL}_{a}(\Bbbk)$ with $a \neq 0$, and hence also the dimensions of the $\on{Hom}$-spaces coincide with that case, the proofs can be somewhat different (essentially due to the zig-zag relation), and in some cases simpler. 
    We include them for completeness.
	
	\begin{definition}
		Let $\overline{(-)}: \mathcal{TL}_{q}(\Bbbk)^{\on{op}} \rightarrow \mathcal{TL}_{q}(\Bbbk)$ be the contravariant monoidal equivalence determined by the assignments $\overline{\on{cup}} = \on{cap}$ and $\overline{\on{cap}} = \on{cup}$. 
	\end{definition}
	
	\begin{definition}
		We say that a non-zero morphism of $\mathcal{TL}_{0}(\Bbbk)$ is a {\it Temperley--Lieb diagram}, if it can be obtained from the generators $\on{cap}, \on{cup}$ and identity morphisms by tensoring and composition.
		
		Further, we define a {\it cup diagram} as a non-zero morphism of $\mathcal{TL}_{0}(\Bbbk)$ which can be obtained from the generator $\on{cup}$ and identity morphisms by tensoring and composition, and we define {\it cap diagrams} similarly.
	\end{definition}

     \begin{definition}
	Given any Temperley--Lieb diagram $\euler{x}$, we define
	$$ K (\euler{x}) = \{(i,j):\text{$\euler{x}$ has a cap joining $i$ and $j$}\},\quad 
	\overline{K}(\euler{x}) = \{(i,j):\text{$\euler{x}$ has a cup joining $i$ and $j$}\}.$$
     \end{definition}
	
	\begin{remark}
		Formally, there is a category $\mathcal{TL}_{\ast}$ enriched over the category $\mathbf{Set}_{\ast}$ of pointed sets, defined similarly to $\mathcal{TL}_{0}$, where the zig-zag evaluates to $\ast$ rather than to $0$. 
        Temperley--Lieb diagrams can be defined as the morphisms of (the category underlying) $\mathcal{TL}_{\ast}$, viewed as a subcategory of (the category underlying) $\mathcal{TL}_{0}(\Bbbk)$. 
        This $\mathbf{Set}_{\ast}$-enriched perspective is used in \cite{Sm}.
	\end{remark}
	
	We refer to morphisms of the form $(a,\on{cup},b) := \on{id}_{\obj{a}} \otimes \on{cup} \otimes \on{id}_{\obj{b}}$ or of the form $(a,\on{cap},b) := \on{id}_{\obj{a}} \otimes \on{cap} \otimes \on{id}_{\obj{b}}$ as {\it basic cups} and {\it basic caps}, respectively. 
    By definition, any Temperley--Lieb diagram can be written as a finite composition of basic caps and basic cups.

    The {\it number of through-strands} $\thru(\euler{x})$ of a Temperley--Lieb diagram $\euler{x}$ is defined by extending assignments $\thru(\on{cap}) = 0 = \thru(\on{cup})$ and $\thru(\on{id}_{\obj{m}}) = m$ by $\thru(\euler{x} \otimes \euler{y}) = \thru(\euler{x}) + \thru(\euler{y})$. 
    Of the relations defining $\mathcal{TL}_{0}(\Bbbk)$, the zig-zag relation is the only subhomogeneous with respect to $\thru$, and therefore we generally only have $\thru(\euler{x} \circ \euler{y}) \leq \min(\thru(\euler{x}), \thru(\euler{y}))$. 
    As a result, we find a chain of ideals $\thru_{\leq 0} \subseteq \thru_{\leq 1} \subseteq \cdots$ in the category $\mathcal{TL}_{0}(\Bbbk)$, such that $\mathcal{TL}_{0}(\Bbbk) = \bigcup_{i=0}^{\infty} \thru_{\leq i}$, where $\thru_{\leq i}$ is given by linear combinations of Temperley--Lieb diagrams with at most $i$ through-strands. 
    The ideal $\thru_{\leq i}$ can also be described as that generated by the morphisms with $i$ through-strands. For a general morphism $f$ of $\mathcal{TL}_{0}(\Bbbk)$, we set $\thru(f) = \min\setj{k \; | \; f \in \thru_{\leq k}}$. 
    In particular, for a linear combination $\sum_{k=1}^{n} \lambda_{i} \euler{x}_{i}$, we have $\thru(\sum_{k=1}^{n} \lambda_{i} \euler{x}_{i}) = \max_{i}\setj{\thru(\euler{x}_{i})}$.
	
	It is easy to verify that for any cup diagram $\euler{x}$, we have $\overline{\euler{x}} \circ \euler{x} = \on{id}_{\obj{\thru(\euler{x})}}$. 
	
	\begin{lemma}\label{linindepcupcap}
		Fix $k,m \in \mathbb{N}$. Consider a pair of cup diagrams $\euler{x},\euler{y} \in \on{Hom}_{\mathcal{TL}_{0}(\Bbbk)}(\obj{m},\obj{m+2k})$. We have 
		\[
		\thru(\overline{\euler{x}} \circ \euler{y}) \leq m = \thru(\euler{x}) = \thru(\euler{y}),
		\]
		with equality if and only if $\euler{x} = \euler{y}$.
	\end{lemma}
	
	\begin{proof}
		In the case $\euler{x} = \euler{y}$, we use the previously made observation $\overline{\euler{x}} \circ \euler{x} = \on{id}_{\obj{\thru(\euler{x})}}$, to conclude the equality.
		
		Assume now that $\euler{x} \neq \euler{y}$. We have three cases to consider:
		\begin{enumerate}
			\item there is $i \in \setj{1,\ldots, m+2k}$ such that the strand in the $i^\text{th}$ outgoing position of $\euler{x}$ is a through-strand, but the strand in the $i^\text{th}$ outgoing position of $\euler{y}$ is not;
			\item there is $i \in \setj{1,\ldots,m +2k}$ such that the strand in $i^\text{th}$ outoing position in $\euler{y}$ is a through-strand, but in $\euler{x}$ it is not;
			\item there is $i \in \setj{1,\ldots,m +2k}$ such that the strand in $i^\text{th}$ outgoing position in $\euler{x}$ is cupped to $j^\text{th}$ outgoing position, while in $\euler{y}$ it is capped to $(j')^\text{th}$ outgoing position, with $j\neq j'$.
		\end{enumerate}
		
		We begin by treating the first case. Let $(i,j) \in \overline{K}(\euler{y})$. If the strand in the $j^\text{th}$ outgoing position in $\euler{x}$ is a through-strand, then $(i,j) \in \overline{K}(\overline{\euler{x}} \circ \euler{y})$, and thus $\thru(\overline{\euler{x}} \circ \euler{y}) < m = \thru(\overline{\euler{x}})$.
		
		If $(j,k) \in \overline{K}(\euler{x})$, where by assumption $k\neq i$, then the cup $(i,j) \in \overline{K}(\euler{y})$ is joined to the cap $(j,k) \in K(\overline{\euler{x}})$, and thus the composite $\overline{\euler{x}} \circ \euler{y}$ is zero by the zig-zag relation. The second case follows similarly to the first, by symmetry.
		
		In the last case, $\overline{\euler{x}} \circ \euler{y} = 0$ by the zig-zag relation.
	\end{proof}
	
	\begin{lemma}\label{triangularbasis}
		Any Temperley--Lieb diagram can be uniquely written as $\euler{x}_{\on{cup}} \circ \euler{x}_{\on{cap}}$, where $\euler{x}_{\on{cup}}$ is a cup diagram and $\euler{x}_{\on{cap}}$ is a cap diagram.
	\end{lemma}
	
	\begin{proof}
		By definition of $\mathcal{TL}_{0}(\Bbbk)$, we may write $\euler{x} = \euler{x}_{n} \circ \cdots \circ \euler{x}_{1}$, where $\euler{x}_{i}$, for $i = 1,\ldots,n$, are basic cups and basic caps. 
        If there is $m \in \setj{1,\ldots,n}$ such that $\euler{x}_{i}$ is a basic cap for $i \leq m$ and a basic cup for $i > m$, then $\euler{x}_{\on{cap}} := \euler{x}_{n} \circ \cdots \circ \euler{x}_{m+1}$ and $\euler{x}_{\on{cup}} := \euler{x}_{m} \circ \cdots \circ \euler{x}_{1}$ provide the claimed decomposition. 
		
		Otherwise, let $k$ be minimal such that $\euler{x}_{k}$ is a basic cap, but there is $j < k$ such that $\euler{x}_{j}$ is a basic cup. Let $\euler{x}_{k} := (a_{k},\on{cap},b_{k})$ and $\euler{x}_{k-1} := (a_{k-1},\on{cup},b_{k-1})$. We have a trichotomy:
		\begin{enumerate}
			\item $a_{k}=a_{k-1}$, and $\euler{x}_{k} \circ \euler{x}_{k-1} = \on{id}_{\obj{a_{k}+b_{k}}}$, by circle evaluation;
			\item $|a_{k} - a_{k-1}| = 1$, and $\euler{x}_{k} \circ \euler{x}_{k-1} = 0$ by zig-zag evaluation;
			\item $|a_{k} - a_{k-1}| \geq 2$ and $\euler{x}_{k} \circ \euler{x}_{k-1} = \euler{x}'_{k} \circ \euler{x}'_{k-1}$, where $\euler{x}_{k}'$ is a basic cup and $\euler{x}_{k-1}'$ is a basic cap.
		\end{enumerate}
		In the second case we obtain a contradiction, since $\euler{x}$ was assumed to be non-zero. 
        In the other two cases we have decreased either the number of basic caps with basic cups below them, or the height of the minimal basic cap with basic cups below it. 
        Basic cups with basic caps can be treated analogously, and the existence of decomposition follows by induction.
		
		To see uniqueness, observe first that $\thru(\euler{x}_{\on{cup}}) = \thru(\euler{x}) = \thru(\euler{x}_{\on{cap}})$. Writing $\euler{x}_{\on{cup}} \circ \euler{x}_{\on{cap}} = \euler{x} = \euler{x}'_{\on{cup}} \circ \euler{x}'_{\on{cap}}$, we find 
		\[
		\thru(\overline{\euler{x}'_{\on{cup}}} \circ \euler{x}_{\on{cup}} \circ \euler{x}_{\on{cap}} \circ \overline{\euler{x}_{\on{cap}}}) = \thru(\euler{x}) \leq \min\setj{\thru(\overline{\euler{x}'_{\on{cup}}} \circ \euler{x}_{\on{cup}}), \thru(\euler{x}_{\on{cap}} \circ \overline{\euler{x}_{\on{cap}}})} = \thru(\overline{\euler{x}'_{\on{cup}}} \circ \euler{x}_{\on{cup}}).
		\]
		showing that $\thru(\overline{\euler{x}'_{\on{cup}}} \circ \euler{x}_{\on{cup}}) = \thru(\euler{x})$, which by \autoref{linindepcupcap} proves $\euler{x}_{\on{cup}} = \euler{x}'_{\on{cup}}$. One shows $\euler{x}_{\on{cap}} = \euler{x}'_{\on{cap}}$ similarly.
	\end{proof}

    \begin{notation}
        Henceforth, whenever we write $\euler{x} = \euler{u} \circ \euler{v}$, we mean that $\euler{u}=\euler{x}_{\on{cup}}$ and $\euler{v} = \euler{x}_{\on{cap}}$.
    \end{notation}
	
	\begin{corollary}\label{TLBasis}
		Temperley--Lieb diagrams give a basis for $\on{Hom}_{\mathcal{TL}_{0}(\Bbbk)}(\obj{m},\obj{n})$. Hence, $\on{dim}\on{Hom}_{\mathcal{TL}_{0}(\Bbbk)}(\obj{m},\obj{n})$ is $C_{\frac{m+n}{2}}$, the $\left(\frac{m+n}{2}\right)^\text{th}$ Catalan number.
	\end{corollary}
	
	\begin{proof}
		Temperley--Lieb diagrams clearly span $\on{Hom}_{\mathcal{TL}_{0}(\Bbbk)}(\obj{m},\obj{n})$.
		To see their linear independence, let $\euler{u}$ be a cup diagram, let $\euler{v}$ be a cap diagram, and let $\euler{u}\circ \euler{v} = \sum_{i =1}^{n} \lambda_{i} \euler{u}_{i} \circ \euler{v}_{i}$, for cup diagrams $\euler{u}_{i}$ and cap diagrams $\euler{v}_{i}$. 
		
		We have 
		\[
		\thru(\euler{u} \circ \euler{v}) = \thru(\overline{\euler{u}} \circ \euler{u} \circ \euler{v} \circ \overline{\euler{v}}) = \max\setj{\thru(\overline{\euler{u}} \circ \euler{u}_{i} \circ \euler{v} \circ \overline{\euler{v}_{i}})}.
		\]
		Thus there must be $i$ such that $\thru(\euler{u} \circ \euler{v}) = \thru(\overline{\euler{u}} \circ \euler{u}_{i} \circ \euler{v}_{i} \circ \overline{\euler{v}})$.
		This implies $\thru(\euler{u}_{i}) = \thru(\euler{u})$ and $\thru(\euler{v}_{i}) = \thru(\euler{v})$, since composition can only decrease the number of through-strands. Thus the domains of $\euler{u}$ and $\euler{u}_{i}$ are the same, and, by assumption, so are the codomains. 
        We may then use \autoref{linindepcupcap} to find $\euler{u} = \euler{u}_{i}$, and $\euler{v} = \euler{v}_{i}$.
		Thus the coefficient of $\euler{u}\circ \euler{v}$ in $\sum_{i =1}^{\euler{v}} \lambda_{i} \euler{u}_{i} \circ \euler{v}_{i}$ must be non-zero, proving linear independence.
	\end{proof}

	\subsection{Jones--Wenzl Projectors}
	
	Similarly to the basis theorem, the category $\mathcal{TL}_{0}(\Bbbk)$ admits Jones--Wenzl projectors, behaving similarly to those for $\mathcal{TL}_{a}(\Bbbk)$ for $a \neq 0$. 
    However, in the case $a = 0$, we give an explicit closed formula for the Jones--Wenzl projectors, much simpler than the formulas for their coefficients in the case $a \neq 0$, as determined in \cite[Section~3.3]{FK}, \cite{Mo}, \cite{Oc}. 
	
	Given $n \in \mathbb{N}$, we say that a subset $I$ of $\setj{1,\ldots,n}$ is {\it apt} if $n \not\in I$ and $i \in I$ implies $i-1,i+1 \not\in I$. Given an apt subset $I$ of $\setj{1,\ldots,n}$, we define $\on{cap}_{I,n} \in \on{End}_{\mathcal{TL}_{0}(\Bbbk)}(\obj{n})$ as the cap diagram consisting of caps in positions $(i,i+1)$, for $i \in I$.
	
	\begin{definition}\label{def:JW}
		We define the {\it $n^\text{th}$ Jones--Wenzl projector} $\euler{j}_{n} \in \on{End}_{\mathcal{TL}_{0}(\Bbbk)}(\obj{n})$ as 
		\[
		\sum_{I\subseteq \setj{1,\ldots,n} \text{ apt}} (-1)^{|I|} \on{cup}_{I,n}\circ  \on{cap}_{I,n} =: \sum_{I\subseteq \setj{1,\ldots,n} \text{ apt}} (-1)^{|I|} c_{I,n}.
		\]
	\end{definition}
	
	\begin{example}
		The Jones--Wenzl projector $\euler{j}_{4}$ on four strands is given by 
		\[
		\begin{aligned}
			&\begin{tikzpicture}
	\begin{pgfonlayer}{nodelayer}
		\node [style=none] (0) at (0, 0) {};
		\node [style=none] (1) at (0.25, 0) {};
		\node [style=none] (2) at (0.125, 0.25) {};
		\node [style=none] (3) at (0.5, 0) {};
		\node [style=none] (4) at (0.75, 0) {};
		\node [style=none] (5) at (0.625, 0.25) {};
		\node [style=none] (6) at (0, 1) {};
		\node [style=none] (7) at (0.25, 1) {};
		\node [style=none] (8) at (0.125, 0.75) {};
		\node [style=none] (9) at (0.5, 1) {};
		\node [style=none] (10) at (0.75, 1) {};
		\node [style=none] (11) at (0.625, 0.75) {};
		\node [style=none] (12) at (1.25, 0.5) {$-$};
		\node [style=none] (13) at (1.75, 0) {};
		\node [style=none] (14) at (2, 0) {};
		\node [style=none] (16) at (2.25, 0) {};
		\node [style=none] (17) at (2.5, 0) {};
		\node [style=none] (18) at (2.375, 0.25) {};
		\node [style=none] (19) at (1.75, 1) {};
		\node [style=none] (20) at (2, 1) {};
		\node [style=none] (22) at (2.25, 1) {};
		\node [style=none] (23) at (2.5, 1) {};
		\node [style=none] (24) at (2.375, 0.75) {};
		\node [style=none] (25) at (3.5, 0) {};
		\node [style=none] (29) at (4.25, 0) {};
		\node [style=none] (31) at (3.75, 1) {};
		\node [style=none] (32) at (4, 1) {};
		\node [style=none] (33) at (3.875, 0.75) {};
		\node [style=none] (37) at (3, 0.5) {$-$};
		\node [style=none] (38) at (4.25, 1) {};
		\node [style=none] (39) at (3.5, 1) {};
		\node [style=none] (40) at (3.75, 0) {};
		\node [style=none] (41) at (4, 0) {};
		\node [style=none] (42) at (3.875, 0.25) {};
		\node [style=none] (43) at (4.75, 0.5) {$-$};
		\node [style=none] (44) at (5.75, 0) {};
		\node [style=none] (45) at (6, 0) {};
		\node [style=none] (46) at (5.25, 0) {};
		\node [style=none] (47) at (5.5, 0) {};
		\node [style=none] (48) at (5.375, 0.25) {};
		\node [style=none] (49) at (5.75, 1) {};
		\node [style=none] (50) at (6, 1) {};
		\node [style=none] (51) at (5.25, 1) {};
		\node [style=none] (52) at (5.5, 1) {};
		\node [style=none] (53) at (5.375, 0.75) {};
		\node [style=none] (54) at (7.5, 0) {};
		\node [style=none] (55) at (7.75, 0) {};
		\node [style=none] (56) at (7, 0) {};
		\node [style=none] (57) at (7.25, 0) {};
		\node [style=none] (59) at (7.5, 1) {};
		\node [style=none] (60) at (7.75, 1) {};
		\node [style=none] (61) at (7, 1) {};
		\node [style=none] (62) at (7.25, 1) {};
		\node [style=none] (64) at (6.5, 0.5) {$+$};
	\end{pgfonlayer}
	\begin{pgfonlayer}{edgelayer}
		\draw [style=thickstrand, in=0, out=90, looseness=1.25] (1.center) to (2.center);
		\draw [style=thickstrand, in=-180, out=90, looseness=1.25] (0.center) to (2.center);
		\draw [style=thickstrand, in=0, out=90, looseness=1.25] (4.center) to (5.center);
		\draw [style=thickstrand, in=-180, out=90, looseness=1.25] (3.center) to (5.center);
		\draw [style=thickstrand, in=0, out=-90, looseness=1.25] (7.center) to (8.center);
		\draw [style=thickstrand, in=180, out=-90, looseness=1.25] (6.center) to (8.center);
		\draw [style=thickstrand, in=0, out=-90, looseness=1.25] (10.center) to (11.center);
		\draw [style=thickstrand, in=180, out=-90, looseness=1.25] (9.center) to (11.center);
		\draw [style=thickstrand, in=0, out=90, looseness=1.25] (17.center) to (18.center);
		\draw [style=thickstrand, in=-180, out=90, looseness=1.25] (16.center) to (18.center);
		\draw [style=thickstrand, in=0, out=-90, looseness=1.25] (23.center) to (24.center);
		\draw [style=thickstrand, in=180, out=-90, looseness=1.25] (22.center) to (24.center);
		\draw [style=thickstrand] (19.center) to (13.center);
		\draw [style=thickstrand] (20.center) to (14.center);
		\draw [style=thickstrand, in=0, out=-90, looseness=1.25] (32.center) to (33.center);
		\draw [style=thickstrand, in=180, out=-90, looseness=1.25] (31.center) to (33.center);
		\draw [style=thickstrand] (39.center) to (25.center);
		\draw [style=thickstrand] (38.center) to (29.center);
		\draw [style=thickstrand, in=0, out=90, looseness=1.25] (41.center) to (42.center);
		\draw [style=thickstrand, in=-180, out=90, looseness=1.25] (40.center) to (42.center);
		\draw [style=thickstrand, in=0, out=90, looseness=1.25] (47.center) to (48.center);
		\draw [style=thickstrand, in=-180, out=90, looseness=1.25] (46.center) to (48.center);
		\draw [style=thickstrand, in=0, out=-90, looseness=1.25] (52.center) to (53.center);
		\draw [style=thickstrand, in=180, out=-90, looseness=1.25] (51.center) to (53.center);
		\draw [style=thickstrand] (49.center) to (44.center);
		\draw [style=thickstrand] (50.center) to (45.center);
		\draw [style=thickstrand] (59.center) to (54.center);
		\draw [style=thickstrand] (60.center) to (55.center);
		\draw [style=thickstrand] (61.center) to (56.center);
		\draw [style=thickstrand] (62.center) to (57.center);
	\end{pgfonlayer}
\end{tikzpicture} \\   
			& = c_{\setj{1,3},4} - c_{\setj{3},4} - c_{\setj{2},4} - c_{\setj{1},4} + c_{\varnothing, 4}.
		\end{aligned}
		\]
	\end{example}
	
	The following claim is easy to verify:
	
	\begin{lemma}
		Let $I,I'$ be apt subsets of $\setj{1,\ldots,n}$. We have a dichotomy:
		\begin{enumerate}
			\item There is $i \in I'$ such that $i+1 \in I$ or $i-1 \in I$, and $c_{I} \circ c_{I'} = 0$;
			\item The set $I \cup I'$ is an apt subset of $\setj{1,\ldots,n}$ and $c_{I} \circ c_{I'} = c_{I \cup I'}$.
		\end{enumerate}
	\end{lemma}
	
	\begin{proposition}
		The Jones--Wenzl projectors are idempotent: $\euler{j}_{n}^{2} = \euler{j}_{n}$.
	\end{proposition}
	
	\begin{proof}
		We have
		\[
		\begin{aligned}
			\left(\sum_{I\subseteq \setj{1,\ldots,n} \text{ apt}} (-1)^{|I|} c_{I,n}\right)^{2} = \sum_{(I,I') \text{ s.t. } I \cup I' \text{ apt}} (-1)^{|I| + |I'|} c_{I \cup I'} = \sum_{J \text{ apt}} \sum_{(I,I') \text{ s.t. } I \cup I' = J} (-1)^{|S|+|S'|} c_{J}.
		\end{aligned}
		\]
		Thus, the coefficient of $c_{J}$ in $\euler{j}_{n}^{2}$ is
		\[
		\sum_{(I,I') \text{ s.t. } I \cup I' = J} (-1)^{|S|+|S'|},
		\]
		and it suffices to show that it is equal to $(-1)^{|J|}$. We have
		\[
		\begin{aligned}
			\sum_{(I,I') \text{ s.t. } I \cup I' = J} (-1)^{|S|+|S'|} = \sum_{|I \cap I'| = 0}^{|I \cap I'| = |J|} (-1)^{|J| + |I \cap I'|} \binom{|J|}{|I \cap I'|} 2^{|J| - |S \cap S'|} = (1-2)^{|J|} = (-1)^{|J|},
		\end{aligned}
		\]
		where the first equality follows from counting the pairs $(I,I')$ by first determining the cardinality $|I \cap I'|$, then choosing the set $I \cap I'$, which can be done in $\binom{|J|}{|I \cap I'|}$ ways, and then for each element in $J \setminus (I \cap I')$, choosing whether it belongs to $I$ or to $I'$ -- this is an exclusive disjunction, and thus we can do so in $2^{|J| - |S \cap S'|}$ ways. 
        The second equality is an application of the binomial theorem.
	\end{proof}
	
	\begin{proposition}\label{annihilation}
		The Jones--Wenzl projectors are annihilated by basic cups and caps: for any $i \in \setj{1,\ldots,n-1}$, we have
		\[
		\on{cap}_{\setj{i},n} \circ\, \euler{j}_{n} = 0 = \euler{j}_{n} \circ \,\on{cup}_{\setj{i},n}.
		\]
        In other words, $\euler{j}_{n}$ is annihilated under both post- and precomposition with the ideal $\thru_{<n}$.
	\end{proposition}
	
	\begin{proof}
		We show that $\on{cap}_{\setj{i},n} \circ \euler{j}_{n} = 0$; the proof for cups is analogous. Let $I$ be an apt subset of $\setj{1,\ldots,n}$. If $i-1 \in I$ or $i+1 \in I$, then $\on{cap}_{\setj{i},n} \circ c_{I,n} = 0$, by the zig-zag relation. Thus,
		\[
		\on{cap}_{\setj{i},n} \circ \,\euler{j}_{n} = \sum_{I \text{ apt; }i-1,i+1 \not\in I} (-1)^{|I|} c_{I,n}.
		\]
		Clearly, we have a bijection
		\[
		\begin{aligned}
			\setj{I \subseteq \setj{1,\ldots,n} \; | \; I \text{ is apt and } i-1,i,i+1 \not\in I} &\leftrightarrow  \setj{I \subseteq \setj{1,\ldots,n} \; | \; I \text{ is apt, } i-1,i+1 \not\in I \text{ and } i \in I} \\
			I &\mapsto I \cup \setj{i} \\
			I\setminus\setj{i} &\mapsfrom I
		\end{aligned},
		\]
		which partitions the set $\setj{I \subseteq \setj{1,\ldots,n} \; | \; I \text{ is apt and } i-1,i+1 \not\in J}$ into two subsets of the same cardinality, which we denote by $N$. 
        Then we have 
		\[
		\begin{aligned}
			\on{cap}_{\setj{i},n} \circ\, \euler{j}_{n} &= \on{cap}_{\setj{i},n} \circ \left(\sum_{I \text{ apt; }i-1,i,i+1 \not\in I} (-1)^{|I|} c_{I,n} + \sum_{I \text{ apt; }i-1,i,i+1 \not\in I} (-1)^{|I|+1} c_{I \cup \setj{i},n}\right) \\
			&= \sum_{I \text{ apt; }i-1,i,i+1 \not\in I} (-1)^{|I|} c_{I\cup \setj{i},n} + \sum_{I \text{ apt; }i-1,i,i+1 \not\in I} (-1)^{|I|+1} c_{I \cup \setj{i},n} = 0.
		\end{aligned}
		\]
	\end{proof}
	We now show that the sequence of projectors $(\euler{j}_{n})_{n=0}^{\infty}$ satisfies the recursive relation described in \cite{We}. 
	\begin{proposition}
		The Jones--Wenzl projectors satisfy the Jones--Wenzl recursion -- the following equality holds:
		\[\begin{tikzpicture}[scale=0.5]
	\begin{pgfonlayer}{nodelayer}
		\node [style=none] (0) at (1.25, 0) {};
		\node [style=none] (1) at (-0.25, 0) {};
		\node [style=jonesrectangle] (2) at (0.5, 2.5) {$\euler{j}_{n}$};
		\node [style=none] (3) at (-0.25, 2) {};
		\node [style=none] (4) at (1.25, 2) {};
		\node [style=none] (5) at (0.5, 1) {$\cdots$};
		\node [style=none] (6) at (-0.25, 3) {};
		\node [style=none] (7) at (1.25, 3) {};
		\node [style=none] (8) at (1.25, 5) {};
		\node [style=none] (9) at (-0.25, 5) {};
		\node [style=none] (10) at (2, 2.5) {$=$};
		\node [style=jonesrectangle] (11) at (3.5, 2.5) {$\euler{j}_{n-1}$};
		\node [style=jonesrectangle] (12) at (7.75, 3.75) {$\euler{j}_{n-1}$};
		\node [style=none] (13) at (0.5, 4) {$\cdots$};
		\node [style=none] (14) at (2.75, 3) {};
		\node [style=none] (15) at (4.25, 3) {};
		\node [style=none] (16) at (4.25, 2) {};
		\node [style=none] (17) at (2.75, 2) {};
		\node [style=none] (18) at (2.75, 0) {};
		\node [style=none] (19) at (4.25, 0) {};
		\node [style=none] (20) at (4.25, 5) {};
		\node [style=none] (21) at (2.75, 5) {};
		\node [style=none] (22) at (8, 4.25) {};
		\node [style=none] (23) at (7, 4.25) {};
		\node [style=none] (24) at (7, 3.25) {};
		\node [style=none] (26) at (8.5, 0) {};
		\node [style=none] (27) at (7, 0) {};
		\node [style=none] (28) at (7, 5) {};
		\node [style=none] (29) at (8, 5) {};
		\node [style=none] (30) at (8.5, 4.25) {};
		\node [style=none] (31) at (8.5, 5) {};
		\node [style=none] (32) at (5, 5) {};
		\node [style=none] (33) at (5, 0) {};
		\node [style=jonesrectangle] (34) at (7.75, 1.25) {$\euler{j}_{n-1}$};
		\node [style=none] (35) at (8, 1.75) {};
		\node [style=none] (36) at (7, 1.75) {};
		\node [style=none] (37) at (7, 0.75) {};
		\node [style=none] (38) at (8.5, 0.75) {};
		\node [style=none] (39) at (8.5, 1.75) {};
		\node [style=none] (43) at (9, 1.75) {};
		\node [style=none] (44) at (8.75, 2.25) {};
		\node [style=none] (45) at (9, 0) {};
		\node [style=none] (46) at (8.5, 3.25) {};
		\node [style=none] (47) at (9, 3.25) {};
		\node [style=none] (48) at (8.75, 2.75) {};
		\node [style=none] (49) at (9, 5) {};
		\node [style=none] (50) at (8, 3.25) {};
		\node [style=none] (51) at (7.5, 2.5) {$\cdots$};
		\node [style=none] (52) at (7.5, 4.7) {$\cdots$};
		\node [style=none] (53) at (8, 0.75) {};
		\node [style=none] (54) at (8, 0) {};
		\node [style=none] (55) at (7.5, 0.275) {$\cdots$};
		\node [style=none] (56) at (3.5, 4) {$\cdots$};
		\node [style=none] (57) at (3.5, 1) {$\cdots$};
		\node [style=none] (58) at (6, 2.5) {$-$};
	\end{pgfonlayer}
	\begin{pgfonlayer}{edgelayer}
		\draw [style=thickstrand] (4.center) to (0.center);
		\draw [style=thickstrand] (3.center) to (1.center);
		\draw [style=thickstrand] (9.center) to (6.center);
		\draw [style=thickstrand] (8.center) to (7.center);
		\draw [style=thickstrand] (29.center) to (22.center);
		\draw [style=thickstrand] (31.center) to (30.center);
		\draw [style=thickstrand] (17.center) to (18.center);
		\draw [style=thickstrand] (16.center) to (19.center);
		\draw [style=thickstrand] (21.center) to (14.center);
		\draw [style=thickstrand] (20.center) to (15.center);
		\draw [style=thickstrand] (32.center) to (33.center);
		\draw [style=thickstrand, in=0, out=90] (43.center) to (44.center);
		\draw [style=thickstrand, in=180, out=90] (39.center) to (44.center);
		\draw [style=thickstrand, in=0, out=-90] (47.center) to (48.center);
		\draw [style=thickstrand, in=-180, out=-90] (46.center) to (48.center);
		\draw [style=thickstrand] (43.center) to (45.center);
		\draw [style=thickstrand] (49.center) to (47.center);
		\draw [style=thickstrand] (23.center) to (28.center);
		\draw [style=thickstrand] (37.center) to (27.center);
		\draw [style=thickstrand] (38.center) to (26.center);
		\draw [style=thickstrand] (50.center) to (35.center);
		\draw [style=thickstrand] (24.center) to (36.center);
		\draw [style=thickstrand] (53.center) to (54.center);
	\end{pgfonlayer}
\end{tikzpicture}\]
	\end{proposition}
	
	\begin{proof}
		
		First, for any $m$, let 
		\[
		\begin{aligned}
			&A_{m} := \setj{I \subseteq \setj{1,\ldots,m} \; | \; I \text{ is apt and } m-2,m-1 \not\in I}; \\
			&B_{m} := \setj{I \subseteq \setj{1,\ldots,m} \; | \; I \text{ is apt and } m-2 \in I}; \\
			&C_{m} := \setj{I \subseteq \setj{1,\ldots,m} \; | \; I \text{ is apt and } m-1 \in I}.
		\end{aligned}
		\]
		Clearly, $\setj{I \subseteq \setj{1,\ldots,m} \; | \; I \text{ is apt}}$ equals the disjoint union $A_{m} \sqcup B_{m} \sqcup C_{m}$. Finally, we denote $A_{m} \sqcup B_{m}$ by $A\!B_{m}$.
		
		Second, for $\euler{x} \in \on{End}_{\mathcal{TL}_{0}(\Bbbk)}(\obj{m})$, let
        \[
        \mathtt{r}(\euler{x}) := \euler{x} \otimes \on{id}_{\obj{1}} - (\euler{x} \otimes \on{id}_{\obj{1}}) \circ \on{cap}_{\setj{m},m+1} \circ \on{cup}_{\setj{m},m+1} \circ (\euler{x} \otimes \on{id}_{\obj{1}}) \in \on{End}_{\mathcal{TL}_{0}(\Bbbk)}(\obj{m+1}),
        \]
         so that the Jones--Wenzl recursion we aim to show can be written as $\euler{j}_{n+1} = \mathtt{r}(\euler{j}_{n})$. 
		
		It is easy to check that for $I \in C_{n}$, we have $\mathtt{r}(c_{I,n}) = c_{I,n+1} - 0 = c_{I,n+1}$. Moreover, $I \in B_{n+1}$ and this establishes a bijection between morphisms of the form $c_{I,n}$ with $I \in C_{n}$ and morphisms of the form $c_{J,n+1}$ with $J \in B_{n+1}$. For $I \in A\!B_{n}$ we have $\mathtt{r}(c_{I,n}) = c_{I,n+1} - c_{I \cup \setj{n},n+1}$. Moreover, $I \in A_{n+1}$ and $I \cup \setj{n} \in C_{n+1}$, and there are bijections
		\[
		\begin{aligned}
			A\!B_{n} &\leftrightarrow A_{n+1}, \\
			I &\leftrightarrow I
		\end{aligned}
		\]
		and 
		\[
		\begin{aligned}
			A\!B_{n} &\leftrightarrow C_{n+1}. \\
			I &\mapsto I \cup \setj{n}
		\end{aligned}
		\]
		Thus, 
		\[
		\begin{aligned}
			\mathtt{r}(\euler{j}_{n}) &= \mathtt{r}\left(\sum_{I \in A\!B_{n}} (-1)^{|I|} c_{I,n} + \sum_{I \in C_{n}} (-1)^{|I|} c_{I,n}\right) = \mathtt{r}\left(\sum_{I \in A\!B_{n}} (-1)^{|I|} c_{I,n}\right) + \mathtt{r}\left(\sum_{I \in C_{n}} (-1)^{|I|} c_{I,n}\right) \\
			&=\sum_{I \in B_{n+1}} c_{I,n+1} + \sum_{I \in A\!B_{n}} (-1)^{|I|} (c_{I,n+1} - c_{I \cup \setj{n}, n+1}) \\
			&= \sum_{I \in B_{n+1}} c_{I,n+1} + \sum_{I \in A_{n+1}} (-1)^{|I|} c_{I,n+1} + \sum_{I \in A\!B_{n}} (-1)^{|I|+1} c_{I\cup \setj{n},n+1} \\
			&= \sum_{I \in B_{n+1}} c_{I,n+1} + \sum_{I \in A_{n+1}} (-1)^{|I|} c_{I,n+1} + \sum_{I \in C_{n+1}} (-1)^{|J|} c_{J,n+1} = \euler{j}_{n+1},
		\end{aligned}
		\]
		which establishes the result.
	\end{proof}
	
	\begin{remark}
		The idempotence of the morphisms $\euler{j}_{n}$ could also be proved similarly to the standard proof for $\mathcal{TL}_{a}(\Bbbk)$ where $a \neq 0$, namely by first proving the absorption property and then calculating directly. 
        Again, the calculations would be slightly different, and in some cases simpler due to additional cancellation entailed by the zig-zag relation.
	\end{remark}
	
	\begin{lemma} \label{JWuniqueness}
		The morphism $\euler{j}_{n} \in \on{End}_{\mathcal{TL}_{0}(\Bbbk)}(n)$ is the unique endomorphism of $n$ such that:
		\begin{enumerate}
			\item It is of the form $\on{id}_{\obj{n}} + x$, where $x \in \thru_{< n}(\obj{n},\obj{n})$;    
			\item It is idempotent;
			\item It is annihilated by basic caps and basic cups, in the sense of \autoref{annihilation}.
		\end{enumerate}
	\end{lemma}
	
	\begin{proof}
		This uniqueness result follows in the exact same way as it does for $\mathcal{TL}_{a}(\Bbbk)$ where $a \neq 0$.
	\end{proof}

	\subsection{Properties of \texorpdfstring{$\mathcal{TL}_{0}(\Bbbk)$}{TL0(k)}}
	
	In this section, we collect some of the general monoidal categorical properties of $\mathcal{TL}_{0}(\Bbbk)$. 
	
	\textbf{Rigidity. } Recall that an object $X$ of a monoidal category $\mathcal{C}$ is said to have a right dual if there is an object $X^{\vee}$ and morphisms $\eta: \mathbb{1}_{\mathcal{C}} \rightarrow X^{\vee} \otimes X$ and $\varepsilon: X \otimes X^{\vee} \rightarrow \mathbb{1}_{\mathcal{C}}$ which satisfy triangle identities analogous to those for adjoint functors. 
    Similarly one defines left duals. An object is said to be {\it rigid} if it admits both a left and a right dual, and $\mathcal{C}$ is said to be rigid if all of its objects are rigid.
	
	\begin{proposition}\label{prop:notrigid}
		The category $\mathcal{TL}_{0}(\Bbbk)$ is not rigid.
	\end{proposition}

	\begin{proof}
		We show that the object $\obj{1} \in \mathcal{TL}_{0}(\Bbbk)$ is not rigid. Assume that there is an object $\obj{k}$ which is right dual to $\obj{1}$. 
        The duality would then yield isomorphisms $\on{Hom}_{\mathcal{TL}_{0}(\Bbbk)}(\obj{m+1},\obj{n}) \simeq \on{Hom}_{\mathcal{TL}_{0}(\Bbbk)}(\obj{m},\obj{n+k})$. But $\on{dim}\on{Hom}_{\mathcal{TL}_{0}(\Bbbk)}(\obj{m+1},\obj{n}) = C_{(m+1+n)/2}$ and similarly $\on{dim}\on{Hom}_{\mathcal{TL}_{0}(\Bbbk)}(\obj{m},\obj{n+k}) = C_{(m+n+k)/2}$, where $C_{j}$ denotes the $j^\text{th}$ Catalan number. 
        Thus, the only possibility is $k=1$. 
        The spaces $\on{Hom}_{\mathcal{TL}_{0}\Bbbk)}(\obj{0}, \obj{1} \otimes \obj{1}) \simeq \Bbbk\setj{\on{cup}}$ and $\on{Hom}_{\mathcal{TL}_{0}\Bbbk)}(\obj{1 \otimes 1},0) \simeq \Bbbk\setj{\on{cap}}$ are both one-dimensional, and so the only candidates for the unit and counit for the duality are scalar multiples of $\on{cup}$ and $\on{cap}$ respectively. 
        Thus, we just need to check that the triangle identities do not hold for such scalar multiples. And indeed, for any $\alpha,\beta \in \Bbbk$ we have
		\[
		(1 \otimes \alpha \cdot \on{cap}) \circ (\beta \cdot \on{cup} \otimes 1) = \alpha\cdot \beta \cdot (1 \otimes \on{cap}) \circ (\on{cup} \otimes 1) = \alpha \cdot \beta \cdot 0 = 0 \neq \on{id}_{1}.
		\]
	\end{proof}
	
	\textbf{Semisimplicity.} We will now use Jones--Wenzl projectors to find semisimple bases for endomorphism algebras of $\mathcal{TL}_0(\mathbb k)$, thus showing that the category is semisimple.
	
	The following claim is similar to \autoref{linindepcupcap}.
	\begin{lemma}\label{jnuj}
		Let $\euler{u}$ be a cup diagram in $\on{Hom}_{\mathcal{TL}_{0}(\Bbbk)}(\obj{a},\obj{b})$ and let $\euler{v}$ be a cap diagram in $\on{Hom}_{\mathcal{TL}_{0}(\Bbbk)}(\obj{b},\obj{c})$. Then 
		\[
		\euler{j}_{c} \circ \euler{v} \circ \euler{u} \circ \euler{j}_{a} = 
		\begin{cases}
			\euler{j}_{a} \quad&\text{ if } \euler{v} = \overline{\euler{u}} \\
			0 &\text{ else.}
		\end{cases}
		\]
	\end{lemma}
	
	\begin{proof}
		If $\euler{v} = \overline{\euler{u}}$ then $\euler{j}_{c} \circ (\euler{v} \circ \euler{u}) \circ \euler{j}_{a} = \euler{j}_{c} \circ \on{id}_{\obj{\thru(\euler{u})}} \circ \euler{j}_{a} = \euler{j}_{a} \circ \euler{j}_{a} = \euler{j}_{a}$, using the fact that $\euler{v} = \overline{\euler{u}}$ implies $a = c$.
		
		If $\euler{v} \neq \overline{\euler{u}}$, then, similarly to the proof of \autoref{linindepcupcap}, there is $i \in \setj{1,\ldots,m}$ such that one of the following three statements holds:
		\begin{enumerate}
			\item there is $i \in \setj{1,\ldots,b}$ such that the strand in $\euler{u}$ outgoing in position $i$ is a through-strand while the strand incoming in position $i$ in $\euler{v}$ is not;
			\item the strand in $\euler{v}$ incoming in position $i$ is a through-strand while that in $\euler{u}$ outgoing in position $i$ is not;
			\item the strand in $\euler{u}$ outgoing in position $i$ is cupped to position $d$, the strand incoming in position $i$ in $\euler{v}$ is capped to position $d'$ and $d\neq d'$.
		\end{enumerate} 
		In the first two cases, the strand in position cupped (or capped) to $i$ joins to a through-strand, forming a cup (or a cap) which is then followed by $\euler{j}_{c}$ (or preceded by $\euler{j}_{a}$), evaluating to $0$, by \autoref{annihilation}, or it joins to a cap (cup), again evaluating to $0$, by the zig-zag relation. 
        In the last case, we have a cap and a cup join to form a zig-zag, and thus again evaluate to zero.
	\end{proof}

	\begin{definition}\label{maxsummand} 
		Given a Temperley--Lieb diagram $\euler{x} = \euler{u} \circ \euler{v}$ in $\on{Hom}_{\mathcal{TL}_{0}(\Bbbk)}(\obj{m},
        \obj{n})$, with $\thru(\euler{x}) = k$, we define its associated {\it max-summand morphism} as $\widehat{\euler{x}}: = \euler{u} \circ \euler{j}_{k} \circ \euler{v}$, where $\euler{j}_{k}$ is the $k^{\text{th}}$ Jones--Wenzl projector, $\euler{j}_{k} \in \on{End}_{\mathcal{TL}_{0}(\Bbbk)}(\obj{k})$.
	\end{definition}
	
	\begin{proposition}\label{JWBasis}
		The set $\setj{\widehat{\euler{x}} \; | \; \euler{x} \text{ is a TL diagram in }\on{Hom}_{\mathcal{TL}_{0}(\Bbbk)}(\obj{m},\obj{n})}$ is a basis for $\on{Hom}_{\mathcal{TL}_{0}(\Bbbk)}(\obj{m},\obj{n})$.
	\end{proposition}
	
	\begin{proof}
		Similarly to the proof of \autoref{TLBasis}, if we write $\euler{u} \circ \euler{j}_{m} \circ \euler{v} = \sum_{k=1}^{n} \lambda_{k} \euler{u}_{k} \circ \euler{j}_{m_{k}} \circ \euler{v}_{k}$, we have 
		\[
		m = \thru(\euler{u} \circ \euler{j}_{m} \circ \euler{v}) = \max_{k}\setj{\thru(\euler{u}_{k} \circ \euler{j}_{m_{k}} \circ \euler{v}_{k})} = \max_{k}m_{k}
		\]
		and
		\[
		m = \thru(\overline{\euler{u}} \circ \euler{u} \circ \euler{j}_{m} \circ \euler{v}\circ \overline{\euler{v}}) = \max\limits_{k} \thru(\overline{\euler{u}} \circ \euler{u}_{k} \circ \euler{j}_{m} \circ \euler{v}_{k} \circ \overline{\euler{v}})
		\]
		and the $k$ making the latter maximum attained must also attain the former. By \autoref{linindepcupcap}, it follows that $\overline{\euler{u}} = \overline{\euler{u}_{k}}$ and hence that $\euler{u} = \euler{u}_{k}$, and similarly one shows that $\euler{v} = \euler{v}_{k}$. 
        This establishes the linear independence of the set $\setj{\widehat{\euler{x}} \; | \; \euler{x} \text{ is a TL diagram in }\on{Hom}_{\mathcal{TL}_{0}(\Bbbk)}(\obj{m},\obj{n})}$; it spans $\on{Hom}_{\mathcal{TL}_{0}(\Bbbk)}(\obj{m},\obj{n})$, since its cardinality equals the dimension of $\on{Hom}_{\mathcal{TL}_{0}(\Bbbk)}(\obj{m},\obj{n})$, as calculated in \autoref{TLBasis}.
	\end{proof}

    \begin{example}
        For $\on{Hom}_{\mathcal{TL}_0(\Bbbk)}(\obj{2},\obj{4})$ the basis described in \ref{JWBasis} is the following set of maps:
        \[\begin{tikzpicture}[scale=0.65]
	\begin{pgfonlayer}{nodelayer}
		\node [style=none] (51) at (-2, 0.5) {};
		\node [style=none] (52) at (-1.5, 0.5) {};
		\node [style=none] (53) at (-1, 0.5) {};
		\node [style=none] (54) at (-0.5, 0.5) {};
		\node [style=none] (55) at (1, 0.5) {};
		\node [style=none] (56) at (1.5, 0.5) {};
		\node [style=none] (57) at (0.5, 0.5) {};
		\node [style=none] (58) at (2, 0.5) {};
		\node [style=none] (59) at (3, 0.5) {};
		\node [style=none] (60) at (3.5, 0.5) {};
		\node [style=none] (61) at (4, 0.5) {};
		\node [style=none] (62) at (4.5, 0.5) {};
		\node [style=none] (63) at (3.5, -0.5) {};
		\node [style=none] (64) at (4, -0.5) {};
		\node [style=none] (65) at (6, 0.5) {};
		\node [style=none] (66) at (6.5, 0.5) {};
		\node [style=none] (67) at (5.5, 0.5) {};
		\node [style=none] (68) at (7, 0.5) {};
		\node [style=none] (69) at (6, -0.5) {};
		\node [style=none] (70) at (6.5, -0.5) {};
		\node [style=none] (71) at (9, 0.5) {};
		\node [style=none] (72) at (9.5, 0.5) {};
		\node [style=none] (73) at (8, 0.5) {};
		\node [style=none] (74) at (8.5, 0.5) {};
		\node [style=none] (75) at (8.5, -0.5) {};
		\node [style=none] (76) at (9, -0.5) {};
		\node [style=none] (85) at (0, -0.5) {,};
		\node [style=none] (86) at (2.5, -0.5) {,};
		\node [style=none] (87) at (5, -0.5) {,};
		\node [style=none] (88) at (7.5, 0) {,};
		\node [style=none] (90) at (-1.5, -1.5) {};
		\node [style=none] (91) at (-1, -1.5) {};
		\node [style=none] (92) at (1, -1.5) {};
		\node [style=none] (93) at (1.5, -1.5) {};
		\node [style=none] (94) at (3.5, -1.5) {};
		\node [style=none] (95) at (4, -1.5) {};
		\node [style=none] (96) at (6, -1.5) {};
		\node [style=none] (97) at (6.5, -1.5) {};
		\node [style=none] (98) at (8.5, -1.5) {};
		\node [style=none] (99) at (9, -1.5) {};
		\node [style=jonesrectangle] (100) at (3.75, -0.5) {$\euler{j}_2$};
		\node [style=jonesrectangle] (101) at (6.25, -0.5) {$\euler{j}_2$};
		\node [style=jonesrectangle] (102) at (8.75, -0.5) {$\euler{j}_2$};
		\node [style=none] (105) at (-2.75, 0.75) {};
		\node [style=none] (106) at (-2.5, -0.25) {};
		\node [style=none] (107) at (-2.75, -0.5) {};
		\node [style=none] (108) at (-2.5, -0.75) {};
		\node [style=none] (109) at (-2.75, -1.75) {};
		\node [style=none] (111) at (-2.5, -2) {};
		\node [style=none] (112) at (-2.5, 1) {};
		\node [style=none] (113) at (10.25, 0.75) {};
		\node [style=none] (114) at (10, -0.25) {};
		\node [style=none] (115) at (10.25, -0.5) {};
		\node [style=none] (116) at (10, -0.75) {};
		\node [style=none] (117) at (10.25, -1.75) {};
		\node [style=none] (118) at (10, -2) {};
		\node [style=none] (119) at (10, 1) {};
		\node [style=none] (120) at (-5, -3.25) {};
		\node [style=none] (121) at (-4.5, -3.25) {};
		\node [style=none] (122) at (-4, -3.25) {};
		\node [style=none] (123) at (-3.5, -3.25) {};
		\node [style=none] (124) at (-2, -3.25) {};
		\node [style=none] (125) at (-1.5, -3.25) {};
		\node [style=none] (126) at (-2.5, -3.25) {};
		\node [style=none] (127) at (-1, -3.25) {};
		\node [style=none] (128) at (0, -3.25) {};
		\node [style=none] (129) at (0.5, -3.25) {};
		\node [style=none] (130) at (1, -3.25) {};
		\node [style=none] (131) at (1.5, -3.25) {};
		\node [style=none] (134) at (5.5, -3.25) {};
		\node [style=none] (135) at (6, -3.25) {};
		\node [style=none] (136) at (5, -3.25) {};
		\node [style=none] (137) at (6.5, -3.25) {};
		\node [style=none] (140) at (11, -3.25) {};
		\node [style=none] (141) at (11.5, -3.25) {};
		\node [style=none] (142) at (10, -3.25) {};
		\node [style=none] (143) at (10.5, -3.25) {};
		\node [style=none] (146) at (-3, -4.25) {,};
		\node [style=none] (147) at (-0.5, -4.25) {,};
		\node [style=none] (148) at (4.5, -4.25) {,};
		\node [style=none] (149) at (9.5, -4.25) {,};
		\node [style=none] (150) at (-4.5, -5.25) {};
		\node [style=none] (151) at (-4, -5.25) {};
		\node [style=none] (152) at (-2, -5.25) {};
		\node [style=none] (153) at (-1.5, -5.25) {};
		\node [style=none] (154) at (0.5, -5.25) {};
		\node [style=none] (155) at (1, -5.25) {};
		\node [style=none] (156) at (5.5, -5.25) {};
		\node [style=none] (157) at (6, -5.25) {};
		\node [style=none] (158) at (10.5, -5.25) {};
		\node [style=none] (159) at (11, -5.25) {};
		\node [style=none] (163) at (-5.75, -3) {};
		\node [style=none] (164) at (-5.5, -4) {};
		\node [style=none] (165) at (-5.75, -4.25) {};
		\node [style=none] (166) at (-5.5, -4.5) {};
		\node [style=none] (167) at (-5.75, -5.5) {};
		\node [style=none] (168) at (-5.5, -5.75) {};
		\node [style=none] (169) at (-5.5, -2.75) {};
		\node [style=none] (170) at (14.75, -3) {};
		\node [style=none] (171) at (14.5, -4) {};
		\node [style=none] (172) at (14.75, -4.25) {};
		\node [style=none] (173) at (14.5, -4.5) {};
		\node [style=none] (174) at (14.75, -5.5) {};
		\node [style=none] (175) at (14.5, -5.75) {};
		\node [style=none] (176) at (14.5, -2.75) {};
		\node [style=none] (177) at (2.5, -3.25) {};
		\node [style=none] (178) at (3, -3.25) {};
		\node [style=none] (179) at (3.5, -3.25) {};
		\node [style=none] (180) at (4, -3.25) {};
		\node [style=none] (181) at (3, -5.25) {};
		\node [style=none] (182) at (3.5, -5.25) {};
		\node [style=none] (183) at (2, -4.25) {$-$};
		\node [style=none] (184) at (8, -3.25) {};
		\node [style=none] (185) at (8.5, -3.25) {};
		\node [style=none] (186) at (7.5, -3.25) {};
		\node [style=none] (187) at (9, -3.25) {};
		\node [style=none] (188) at (8, -5.25) {};
		\node [style=none] (189) at (8.5, -5.25) {};
		\node [style=none] (190) at (7, -4.25) {$-$};
		\node [style=none] (191) at (12.5, -3.25) {};
		\node [style=none] (192) at (13, -3.25) {};
		\node [style=none] (193) at (13.5, -3.25) {};
		\node [style=none] (194) at (14, -3.25) {};
		\node [style=none] (195) at (13, -5.25) {};
		\node [style=none] (196) at (13.5, -5.25) {};
		\node [style=none] (197) at (12, -4.25) {$-$};
		\node [style=none] (198) at (-6.5, -4.25) {$=$};
	\end{pgfonlayer}
	\begin{pgfonlayer}{edgelayer}
		\draw [style=thickstrand] [bend right=90, looseness=1.75] (51.center) to (52.center);
		\draw [style=thickstrand] [bend right=90, looseness=1.75] (53.center) to (54.center);
		\draw [style=thickstrand] [bend right=90, looseness=1.75] (55.center) to (56.center);
		\draw [style=thickstrand] [in=-90, out=-90, looseness=1.25] (57.center) to (58.center);
		\draw [style=thickstrand] [bend right=90, looseness=1.75] (59.center) to (60.center);
		\draw [style=thickstrand] [bend right=90, looseness=1.75] (65.center) to (66.center);
		\draw [style=thickstrand] [bend right=90, looseness=1.75] (71.center) to (72.center);
		\draw [style=thickstrand] [in=90, out=-90] (61.center) to (63.center);
		\draw [style=thickstrand] [in=90, out=-90] (62.center) to (64.center);
		\draw [style=thickstrand] [in=90, out=-90] (67.center) to (69.center);
		\draw [style=thickstrand] [in=90, out=-90] (68.center) to (70.center);
		\draw [style=thickstrand] [in=90, out=-90] (73.center) to (75.center);
		\draw [style=thickstrand] [in=-90, out=90] (76.center) to (74.center);
		\draw [style=thickstrand] [bend left=90, looseness=1.75] (90.center) to (91.center);
		\draw [style=thickstrand] [bend left=90, looseness=1.75] (92.center) to (93.center);
		\draw [style=thickstrand] (63.center) to (94.center);
		\draw [style=thickstrand] (64.center) to (95.center);
		\draw [style=thickstrand] (69.center) to (96.center);
		\draw [style=thickstrand] (70.center) to (97.center);
		\draw [style=thickstrand] (75.center) to (98.center);
		\draw [style=thickstrand] (76.center) to (99.center);
		\draw [style=thickstrand] [in=90, out=-90] (105.center) to (106.center);
		\draw [style=thickstrand] [in=180, out=-90, looseness=1.25] (106.center) to (107.center);
		\draw [style=thickstrand] [in=90, out=-180] (107.center) to (108.center);
		\draw [style=thickstrand] [in=90, out=-90] (108.center) to (109.center);
		\draw [style=thickstrand] [in=180, out=90] (105.center) to (112.center);
		\draw [style=thickstrand] [in=180, out=-90] (109.center) to (111.center);
		\draw [style=thickstrand] [in=90, out=-90] (113.center) to (114.center);
		\draw [style=thickstrand] [in=0, out=-90, looseness=1.25] (114.center) to (115.center);
		\draw [style=thickstrand] [in=90, out=0] (115.center) to (116.center);
		\draw [style=thickstrand] [in=90, out=-90] (116.center) to (117.center);
		\draw [style=thickstrand] [in=0, out=90] (113.center) to (119.center);
		\draw [style=thickstrand] [in=0, out=-90] (117.center) to (118.center);
		\draw [style=thickstrand] [bend right=90, looseness=1.75] (120.center) to (121.center);
		\draw [style=thickstrand] [bend right=90, looseness=1.75] (122.center) to (123.center);
		\draw [style=thickstrand] [bend right=90, looseness=1.75] (124.center) to (125.center);
		\draw [style=thickstrand] [in=-90, out=-90, looseness=1.25] (126.center) to (127.center);
		\draw [style=thickstrand] [bend right=90, looseness=1.75] (128.center) to (129.center);
		\draw [style=thickstrand] [bend right=90, looseness=1.75] (134.center) to (135.center);
		\draw [style=thickstrand] [bend right=90, looseness=1.75] (140.center) to (141.center);
		\draw [style=thickstrand] [bend left=90, looseness=1.75] (150.center) to (151.center);
		\draw [style=thickstrand] [bend left=90, looseness=1.75] (152.center) to (153.center);
		\draw [style=thickstrand] [in=90, out=-90] (163.center) to (164.center);
		\draw [style=thickstrand] [in=180, out=-90, looseness=1.25] (164.center) to (165.center);
		\draw [style=thickstrand] [in=90, out=-180] (165.center) to (166.center);
		\draw [style=thickstrand] [in=90, out=-90] (166.center) to (167.center);
		\draw [style=thickstrand] [in=180, out=90] (163.center) to (169.center);
		\draw [style=thickstrand] [in=180, out=-90] (167.center) to (168.center);
		\draw [style=thickstrand] [in=90, out=-90] (170.center) to (171.center);
		\draw [style=thickstrand] [in=0, out=-90, looseness=1.25] (171.center) to (172.center);
		\draw [style=thickstrand] [in=90, out=0] (172.center) to (173.center);
		\draw [style=thickstrand] [in=90, out=-90] (173.center) to (174.center);
		\draw [style=thickstrand] [in=0, out=90] (170.center) to (176.center);
		\draw [style=thickstrand] [in=0, out=-90] (174.center) to (175.center);
		\draw [style=thickstrand] [bend right=90, looseness=1.75] (177.center) to (178.center);
		\draw [style=thickstrand] [bend right=90, looseness=1.75] (179.center) to (180.center);
		\draw [style=thickstrand] [bend left=90, looseness=1.75] (181.center) to (182.center);
		\draw [style=thickstrand] [bend right=90, looseness=1.75] (184.center) to (185.center);
		\draw [style=thickstrand] [in=-90, out=-90, looseness=1.25] (186.center) to (187.center);
		\draw [style=thickstrand] [bend left=90, looseness=1.75] (188.center) to (189.center);
		\draw [style=thickstrand] [bend right=90, looseness=1.75] (191.center) to (192.center);
		\draw [style=thickstrand] [bend right=90, looseness=1.75] (193.center) to (194.center);
		\draw [style=thickstrand] [bend left=90, looseness=1.75] (195.center) to (196.center);
		\draw [style=thickstrand] [in=90, out=-90, looseness=1.25] (130.center) to (154.center);
		\draw [style=thickstrand] [in=90, out=-90, looseness=1.25] (131.center) to (155.center);
		\draw [style=thickstrand] [in=90, out=-90, looseness=1.25] (136.center) to (156.center);
		\draw [style=thickstrand] [in=-90, out=90, looseness=1.25] (157.center) to (137.center);
		\draw [style=thickstrand] [in=90, out=-90, looseness=1.25] (142.center) to (158.center);
		\draw [style=thickstrand] [in=-90, out=90, looseness=1.25] (159.center) to (143.center);
	\end{pgfonlayer}
\end{tikzpicture}.\]
        
    \end{example}
	
	\begin{lemma}\label{multiplicationrule}
		Given Temperley--Lieb diagrams $\euler{x} = \euler{u} \circ \euler{v}$ and $\euler{x}' = \euler{u}' \circ \euler{v}'$, we have
		\[
		\widehat{\euler{x}'} \circ \widehat{\euler{x}} =
		\begin{cases}
			\widehat{\euler{u}' \circ \euler{v}} = \euler{u}' \circ \euler{j}_{\thru(\euler{u})} \circ \euler{v} \quad\quad &\text{ if } \euler{v}' = \overline{\euler{u}} \\
			0 &\text{ otherwise.}
		\end{cases}
		\]
	\end{lemma}
	
	\begin{proof}
		By \autoref{jnuj}, we have
		\[
		\widehat{\euler{x}'} \circ \widehat{\euler{x}} = \euler{u}' \circ \euler{j}_{\thru(\euler{u}')} \circ \euler{v}' \circ \euler{u} \circ \euler{j}_{\thru(\euler{u})} \circ \euler{v} = 
		\begin{cases}
            \euler{u}' \circ \euler{j}_{\thru(\euler{u})}^{2} \circ \euler{v} = \widehat{\euler{u}' \circ \euler{v}} \quad\quad &\text{ if } \euler{v}' = \overline{\euler{u}}\\
			\euler{u}' \circ 0 \circ \euler{v} = 0  &\text{ otherwise.}
		\end{cases}
		\]
	\end{proof}
	
	\begin{corollary}\label{orthogonalisotypic}
		Given Temperley--Lieb diagrams $\euler{x}, \euler{x}': \obj{a} \rightarrow \obj{a}$ such that $\thru(\euler{x}) \neq \thru(\euler{x}')$, we have $\widehat{\euler{x}'} \circ \widehat{\euler{x}} = 0$.
	\end{corollary}
	
	\begin{proof}
		Writing $\euler{x} = \euler{u} \circ \euler{v}$ and $\euler{x}' = \euler{u}' \circ \euler{v}'$, the equality $\euler{v}' = \overline{\euler{u}}$ implies $\thru(\euler{x}') = \thru(\euler{v}') = \thru(\overline{\euler{u}}) = \thru(\euler{u}) = \thru(\euler{x})$.
	\end{proof}
	
	We say that an object $X$ of a $\Bbbk$-linear category $\mathcal{A}$ is {\it simple} if any non-zero endomoprhism of $X$ is an isomorphism. 
    The category $\mathcal{A}$ is said to be semisimple if it is additive and any object of $\mathcal{A}$ is a direct sum of finitely many simple objects.
	
	\begin{lemma}[{\cite[Lemma~2.1]{MS}}]\label{MorrisonSnyder}
		Let $\mathcal{A}$ be a $\Bbbk$-linear, idempotent split category. Then $\mathcal{A}$ is semisimple if and only if for every $X \in \mathcal{A}$, the algebra $\on{End}_{\mathcal{A}}(X)$ is semisimple.
	\end{lemma}
	
	\begin{definition}\label{matrixlabels}
		For $k,m \in \mathbb{N}$ with $k\leq m$, let $$\on{End}_{\mathcal{TL}_{0}(\Bbbk)}(\obj{m})_{\thru = k} = \on{Span}\setj{\widehat{\euler{x}} \; | \; \euler{x}: \obj{m}\rightarrow \obj{m} \text{ is a TL-diagram with } \thru(\euler{x}) = k}.$$ 
		
		Let $\setj{\euler{u}_{1},\ldots,\euler{u}_{r}}$ be an enumeration of all cup diagrams with domain $\obj{k}$ and codomain $\obj{m}$.
        For $a,b \in \setj{1,\ldots,r}$, let $\widehat{\euler{x}}_{ab} = \euler{u}_{a} \circ \euler{j}_{k} \circ \overline{\euler{u}}_{b}$.
	\end{definition}
	
	\begin{lemma}\label{algebraiso}
		For any $k,m \in \mathbb{N}$ with $k\leq m$, the subspace $\on{End}_{\mathcal{TL}_{0}(\Bbbk)}(\obj{m})_{\thru = k}$ is a $\Bbbk$-subalgebra of $\on{End}_{\mathcal{TL}_{0}(\Bbbk)}(\obj{m})$. 
		Moreover, for an enumeration $\setj{\euler{u}_{1},\ldots,\euler{u}_{r}}$ as in \autoref{matrixlabels}, the map
		\[
		\begin{aligned}
			\on{End}_{\mathcal{TL}_{0}(\Bbbk)}(\obj{m})_{\thru = k} &\to \on{Mat}_r(\Bbbk)\\
			\widehat{\euler{x}}_{ab} &\mapsto E_{ab}
		\end{aligned}
		\]
		where $E_{ab}$ is the $(a,b)^\text{th}$ elementary matrix, is an algebra isomorphism.
	\end{lemma}

	\begin{proof}
		By \autoref{multiplicationrule}, we have 
		$\widehat{\euler{x}}_{cd}\circ \widehat{\euler{x}}_{ab} = \delta_{d,a} \widehat{\euler{x}}_{cb}$, where $\delta$ denotes the Kronecker delta. 
		This shows the first statement, and also the second since we obtain structure constants $\gamma_{(cd),(ab)}^{(ef)} = \delta_{(ef),(cb)}\cdot \delta_{d,a}$, which agrees with the structure constants for the basis of elementary matrices in $\on{Mat}_{r}(\Bbbk)$.
	\end{proof}
	
	\begin{corollary}\label{EndAlgebras}
		Fix $m \in \mathbb{N}$. We have 
		\[
		\on{End}_{\mathcal{TL}_{0}(\Bbbk)}(\obj{m}) \cong \prod_{k\leq m} \on{End}_{\mathcal{TL}_{0}(\Bbbk)}(\obj{m})_{\thru = k} \cong \prod_{k\leq m} \on{Mat}_{r_{k}}(\Bbbk),
		\]
		where $r_{k} = \left|\setj{\text{cup diagrams in } \on{Hom}_{\mathcal{TL}_{0}(\Bbbk)}(\obj{k},\obj{m})}\right| 
        $.
	\end{corollary}
	
	\begin{proof}
		This follows by combining \autoref{orthogonalisotypic} and \autoref{algebraiso}.
	\end{proof}

    Let $\mathscr{D}_{\mobijd{n}}$ be the set of all possible cap diagrams in $\on{Hom}_{\mathcal{TL}_0(\Bbbk)}(\obj{n},\obj{m})$ for some $m\leq n$.
    The following easy corollary will be useful in the next section:

        \begin{corollary}\label{DnIdempotents}
		$\{\overline{\euler{x}}\circ \euler{j}_{\thru(\euler{x})}\circ \euler{x}\}_{\euler{x}\in \mathscr{D}_{\mobijd{n}}}$ is a set of orthogonal idempotents of $\on{End}_{\mathcal{TL}_0(\Bbbk)}(\obj{n})$ whose sum is $\on{id}_{\obj{n}}$.
	\end{corollary}

\begin{proof}
    This follows directly from \autoref{EndAlgebras} by examining the preimages of the idempotents $E_{ii}\in\on{Mat}_r(\Bbbk)$ under the map of \autoref{algebraiso}.
\end{proof}

        \begin{example}
        For $\obj{m}=\obj{4}$, we have $r_0=2$, $r_2=3$, and $r_4=1$ corresponding to cup diagrams:
        
        
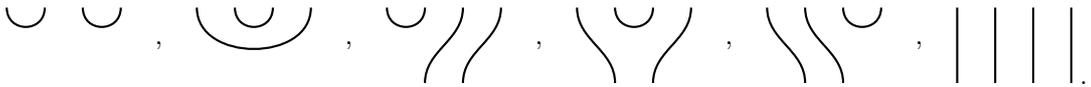
\[\begin{tikzpicture}
	\begin{pgfonlayer}{nodelayer}
		\node [style=none] (51) at (-5, 1) {};
		\node [style=none] (52) at (-4.5, 1) {};
		\node [style=none] (53) at (-4, 1) {};
		\node [style=none] (54) at (-3.5, 1) {};
		\node [style=none] (55) at (-2, 1) {};
		\node [style=none] (56) at (-1.5, 1) {};
		\node [style=none] (57) at (-2.5, 1) {};
		\node [style=none] (58) at (-1, 1) {};
		\node [style=none] (59) at (0, 1) {};
		\node [style=none] (60) at (0.5, 1) {};
		\node [style=none] (61) at (1, 1) {};
		\node [style=none] (62) at (1.5, 1) {};
		\node [style=none] (63) at (0.5, 0) {};
		\node [style=none] (64) at (1, 0) {};
		\node [style=none] (65) at (3, 1) {};
		\node [style=none] (66) at (3.5, 1) {};
		\node [style=none] (67) at (2.5, 1) {};
		\node [style=none] (68) at (4, 1) {};
		\node [style=none] (69) at (3, 0) {};
		\node [style=none] (70) at (3.5, 0) {};
		\node [style=none] (71) at (6, 1) {};
		\node [style=none] (72) at (6.5, 1) {};
		\node [style=none] (73) at (5, 1) {};
		\node [style=none] (74) at (5.5, 1) {};
		\node [style=none] (75) at (5.5, 0) {};
		\node [style=none] (76) at (6, 0) {};
		\node [style=none] (77) at (7.5, 1) {};
		\node [style=none] (78) at (8, 1) {};
		\node [style=none] (79) at (8.5, 1) {};
		\node [style=none] (80) at (9, 1) {};
		\node [style=none] (81) at (7.5, 0) {};
		\node [style=none] (82) at (8, 0) {};
		\node [style=none] (83) at (8.5, 0) {};
		\node [style=none] (84) at (9, 0) {};
		\node [style=none] (85) at (-3, 0.5) {,};
		\node [style=none] (86) at (-0.5, 0.5) {,};
		\node [style=none] (87) at (2, 0.5) {,};
		\node [style=none] (88) at (4.5, 0.5) {,};
		\node [style=none] (89) at (7, 0.5) {,};
	\end{pgfonlayer}
	\begin{pgfonlayer}{edgelayer}
		\draw [style=thickstrand] [bend right=90, looseness=1.75] (51.center) to (52.center);
		\draw [style=thickstrand] [bend right=90, looseness=1.75] (53.center) to (54.center);
		\draw [style=thickstrand] [bend right=90, looseness=1.75] (55.center) to (56.center);
		\draw [style=thickstrand] [in=-90, out=-90, looseness=1.25] (57.center) to (58.center);
		\draw [style=thickstrand] [bend right=90, looseness=1.75] (59.center) to (60.center);
		\draw [style=thickstrand] [bend right=90, looseness=1.75] (65.center) to (66.center);
		\draw [style=thickstrand] [bend right=90, looseness=1.75] (71.center) to (72.center);
		\draw [style=thickstrand] [in=90, out=-90] (61.center) to (63.center);
		\draw [style=thickstrand] [in=90, out=-90] (62.center) to (64.center);
		\draw [style=thickstrand] [in=90, out=-90] (67.center) to (69.center);
		\draw [style=thickstrand] [in=90, out=-90] (68.center) to (70.center);
		\draw [style=thickstrand] [in=90, out=-90] (73.center) to (75.center);
		\draw [style=thickstrand] [in=-90, out=90] (76.center) to (74.center);
		\draw [style=thickstrand] (77.center) to (81.center);
		\draw [style=thickstrand] (78.center) to (82.center);
		\draw [style=thickstrand] (79.center) to (83.center);
		\draw [style=thickstrand] (80.center) to (84.center);
	\end{pgfonlayer}
\end{tikzpicture}
\;.\]
        
        Thus, we have a semisimple decomposition 
        $\on{End}_{\mathcal{TL}_{0}(\Bbbk)}(\obj{m})\cong \on{Mat}_2(\Bbbk)\times \on{Mat}_3(\Bbbk) \times \Bbbk$.
    \end{example}

    Combining \autoref{EndAlgebras} and \autoref{MorrisonSnyder}, we find the following:
	\begin{corollary}\label{cor:semisimple}
		The category $\mathbf{CrysTL}$ is semisimple.
	\end{corollary}

 The semisimple decomposition of $\on{End}_{\mathcal{TL}_{0}(\Bbbk)}(\obj{m})$ also easily gives us a complete list of simple module of these algebras, which happens to be the same as the list of simple modules for $\on{End}_{\mathcal{TL}_{q}(\Bbbk)}(\obj{m})$ when $q\neq 0$. See e.g. \cite{RSA}, \cite{Wes} for more on the representation theory in the case $q \neq 0$.
    \begin{definition} \label{simplemods}
        Let $L_{m,k}$ be the $\Bbbk$-span of all cup diagrams from $\obj{k}$ to $\obj{m}$ equipped with the left action of $\on{End}_{\mathcal{TL}_{0}(\Bbbk)}(\obj{m})$ defined on cup diagrams (and extended linearly) as follows: given $\euler{x}\in \on{End}_{\mathcal{TL}_{0}(\Bbbk)}(\obj{m})$ and a cup diagram $\euler{u}:\obj{k}\to\obj{m}$, let $\euler{x}\cdot \euler{u} = \begin{cases}
            \euler{x}\circ \euler{u}  \quad &\text{if $\euler{x}\circ \euler{u}$ is a cup diagram}\\
            0 &\text{otherwise}\\
        \end{cases}$.
        Similarly, define the right modules $R_{m,k}$ using cap (instead of cup) diagrams and pre-composition (instead of post-composition).
    \end{definition}

    \begin{corollary}\label{irredundantlist}
        The set $\{L_{m,k}\}$ (resp. $\{R_{m,k}\}$) for $1\leq k\leq m$ and $k\equiv m \mod 2$ gives a complete and irredundant list of simple left (resp. right) modules over $\on{End}_{\mathcal{TL}_{0}(\Bbbk)}(\obj{m})$
    \end{corollary}
    \begin{proof}
        It is easy to verify that the definition above defines left/right actions and that it is transitive on cup and cap diagrams: indeed, for any cup diagrams $\euler{u},\euler{u}'$, we have $(\euler{u'}\circ \overline{\euler{u}})\cdot \euler{u} = \euler{u}'$, and similarly for cap diagrams; hence, these modules are simple.
        By comparing dimensions with the decomposition given in \autoref{EndAlgebras}, we see that the list is complete and irredundant.
    \end{proof}

    	\subsection{Pointed Sets, Monoid Algebras, M{\" o}bius Inversion and Cells}\label{sec:mobius}

    In \cite{Sm}, a $\mathbf{Set}_{\ast}$-enriched category of $\mathfrak{g}$-crystals, is defined. In the case $\mathfrak{g} = \mathfrak{sl}_{2}$, we define the category $\mathbf{Crys}_{\ast}$ as follows:
    \begin{itemize}
        \item $\on{Ob}(\mathbf{Crys}_{\ast}) = \mathbb{N}$;
        \item $\on{Hom}_{\mathbf{Crys}_{\ast}}(\obj{m},\obj{n}$ is the disjoint union of the set of Temperley--Lieb diagrams from $\obj{m}$ to $\obj{n}$ and $\setj{\ast}$. 
    \end{itemize}
    We compose Temperley--Lieb diagrams as in $\mathcal{TL}_{0}(\Bbbk)$, except that the diagrams which would compose to zero in $\mathcal{TL}_{0}(\Bbbk)$ instead compose to $\ast$ in $\mathbf{Crys}_{\ast}$. 
    The monoidal structure is still given by horizontal concatenation of Temperley--Lieb diagrams, and $\ast \otimes f = \ast$ for any morphism $f$ of $\on{Ob}(\mathbf{Crys}_{\ast})$.

    \begin{definition}
       For a commutative ring $R$, there is a monoidal functor $R \boxtimes_{\!\ast} -: \mathbf{Set}_{\ast} \rightarrow R\!\on{-Mod}$, sending $S \in \mathbf{Set}_{\ast}$ to $R\setj{S}/\on{Span}\setj{\ast}$, where $R\setj{S}$ is the free $R$-module on $S$. 
       Abusing notation, we also denote by $R \boxtimes_{\!\ast} -$ the induced monoidal $2$-functor from $\mathbf{Set}_{\ast}\text{-}\mathbf{Cat} \rightarrow \mathbf{Cat}_{R}$.
    \end{definition}

    Clearly, $\mathcal{TL}_{0}(\Bbbk) \simeq \Bbbk\boxtimes_{\!\ast} \mathbf{Crys}_{\ast}$. Further, since the $2$-functor $\Bbbk \boxtimes_{\!\ast} -$ is monoidal, it sends pseudomonoids to pseudomonoids, endowing $\Bbbk\boxtimes_{\!\ast} \mathbf{Crys}_{\ast}$ with a monoidal structure. 
    It is easy to see that the equivalence with $\mathcal{TL}_{0}(\Bbbk)$ is then monoidal. 
    
    Next, we note that for a monoid object $M$ in $\mathbf{Set}_{\ast}$, i.e. a monoid with a zero element, the monoid $R \boxtimes_{\!\ast} M$ is the
    {\it contracted monoid algebra } (or {\it reduced monoid algebra})
    for $M$ over $R$. So, $\on{End}_{\Bbbk\boxtimes_{\!\ast} \mathbf{Crys}_{\ast}}(\obj{n}) \simeq \on{End}_{\mathcal{TL}_{0}(\Bbbk)}(\obj{n}) \sqcup \setj{\ast}$. 
    We denote this monoid by $T_{n}$.

    In particular, $\on{End}_{\mathcal{TL}_{0}(\Bbbk)}(\obj{n})$ is the contracted monoid algebra of $T_{n}$, obtained as the quotient of $\Bbbk T_{n}$ by the ideal spanned by the zero element $\ast$. 
    Decomposing the regular left module $\Bbbk T_{n}$, this ideal spans the unique trivial direct summand of the regular module, where by trivial we mean a module in which every monoid element acts by identity. 
    
   Thus, restriction along the surjective $\Bbbk$-algebra map $\Bbbk T_{n} \rightarrow \on{End}_{\mathcal{TL}_{0}(\Bbbk)}(\obj{n})$ realizes the category $\on{End}_{\mathcal{TL}_{0}(\Bbbk)}(\obj{n})\!\on{-mod}$ as the full subcategory of $T_{n}\!\on{-mod}$ consisting of modules which are not trivial - since in such a module $\ast$ necessarily acts by zero.

    Recall that for a monoid $T$, the left, right and two-sided preorders $\mathcal{L,R}$ and $\mathcal{J}$ on $T$ are defined by
    \[
    x \leq_{\mathcal{L}} y \text{ if } Tx \subseteq Ty, \; 
    x \leq_{\mathcal{R}} y \text{ if } xT \subseteq yT, \text{ and }
    x \leq_{\mathcal{J}} y \text{ if } TxT \subseteq TyT.
    \]
	Their associated equivalence relations, known as {\it Green's relations}, partition $T$ into so-called cells. 
    The following is a brief description of Green's cells for the monoid $T_{n}$: $x \leq_{\mathcal{J}} y$ if and only if $\thru(x) \leq \thru(y)$, and for $\euler{x},\euler{x}'$ in the same $\mathcal{J}$-cell, writing $\euler{x} = \euler{u} \circ \euler{v}$ and $\euler{x}' = \euler{u}' \circ \euler{v}'$ yields $\euler{x} \sim_{\mathcal{L}} \euler{x}'$ if and only if $\euler{v} = \euler{v}'$, otherwise the elements are not $\mathcal{L}$-comparable; similarly for $\mathcal{R}$ and $\euler{u},\euler{u}'$. 
    This description of the cell structure is valid also in the case $q=0$, and we record a short proof for the $\mathcal{J}$-order. 
    In the formulation of the next statement, we may set $\thru(\ast) = -\infty$.
    
    \begin{lemma}
    For $\euler{x}, \euler{x}' \in T_{n}$ we have $\euler{x} \leq_{\mathcal{J}} \euler{x}'$ if and only if $\thru(\euler{x}) \leq \thru(\euler{x}')$.
    \end{lemma}

    \begin{proof}
     The statement is obvious when either $\euler{x}$ or $\euler{x}'$ equals $\ast$, so we omit that case.
     If $\euler{x} \leq_{\mathcal{J}} \euler{x}'$ then $\euler{x} = \euler{y}' \circ \euler{x}' \circ \euler{y}$ for some $\euler{y}, \euler{y}' \in T_{n}$. 
     It follows that $\thru(\euler{x}) \leq \thru(\euler{x}')$. 

     Assume now that $\thru(\euler{x}) \leq \thru(\euler{x}')$, and write $\euler{x} = \euler{u} \circ \euler{v}$ and $\euler{x}' = \euler{u}' \circ \euler{v}'$. 
     Let $\euler{z} \in \on{Hom}_{\mathcal{TL}_{0}(\Bbbk)}(\obj{0},\obj{\thru(\euler{x}') - \thru(\euler{x})})$ be such that $\thru(\euler{z}) = 0$. 
     Then 
     \[
     \euler{x} = \euler{u} \circ (\on{id}_{\obj{\thru(\euler{x})}} \otimes \overline{\euler{z}}) \circ \euler{u}' \circ \euler{x}' \circ \euler{v}' \circ (\on{id}_{\obj{\thru(\euler{x})}} \otimes \euler{z}) \circ \euler{v}.
     \]
    \end{proof}

    To conclude our discussion about cells in $T_{n}$, we remark that $\on{End}_{\mathcal{TL}_{0}(\Bbbk)}(\obj{n})$ is cellular in the sense of \cite{GL}, with cellular data similar to $\mathcal{TL}_{q}(\Bbbk)$ for $q\neq 0$, given by $(\Lambda, M, C,\ast)$, where $\Lambda = \setj{k \in \mathbb{N}_{0}, k \leq n \text{ and } n - k \in 2\mathbb{Z}}$, $M(t)$ is the set of cup diagrams with $t$ through-strands, and $C_{\euler{u}\euler{u}'}^{t} = \overline{\euler{u}'}\circ \euler{u}$. 
    This cellular structure can be used as an alternative way to obtain the description of simple modules in \autoref{irredundantlist} and the proof of semisimplicity of \autoref{cor:semisimple}.

    For a monoid $T$, we denote the set of idempotents in $T$ by $E(T)$. 
    For $e \in E(T)$, the {\it maximal subgroup of $T$ at $e$} is the group of units in the monoid $eTe$. 
    The non-zero idempotents of $T_{n}$ are in bijection with cap diagrams in $\on{End}_{\mathcal{TL}_0(\Bbbk)}(\obj{n})$, mapping a cap diagram $\euler{v}$ to the idempotent $\overline{\euler{v}} \circ \euler{v}$. 
    
	Recall that a monoid $S$ is said to be {\it inverse} if for every $s \in S$ there is a unique $t$ such that $s = sts$ and $t = tst$. The element $t$ is said to be the {\it inverse} of $s$, and we will denote it by $s^{(-1)}$.
	\begin{lemma}
		The monoid $T_{n}$ is an inverse monoid.   
	\end{lemma}
	
	\begin{proof}
		First, clearly we have $\ast^{(-1)} = \ast$. 
        Let $\euler{x} = \euler{u} \circ \euler{v}$ be a Temperley--Lieb diagram, where $\euler{u}$ is a cup diagram and $\euler{v}$ is a cap diagram. 
        We claim that $(\euler{u} \circ \euler{v})^{(-1)} = \overline{\euler{u} \circ \euler{v}} = \overline{\euler{v}} \circ \overline{\euler{u}}$. The inverse relation follows directly from $\overline{\euler{u}} \circ \euler{u} = \on{id}_{\obj{\thru(\euler{u})}}$ and $\euler{v} \circ \overline{\euler{v}} = \on{id}_{\obj{\thru(\euler{v})}}$. 
        To see that it is unique, observe that 
		\[
		\thru\left((\euler{u} \circ \euler{v}) \circ (\euler{u}' \circ \euler{v}') \circ (\euler{u} \circ \euler{v})\right) = \thru(\euler{u}\circ \euler{v})
		\]
		requires $\thru(\euler{v} \circ \euler{u}') = \thru(\euler{v})$ and $\thru(\euler{v}' \circ \euler{u}) = \thru(\euler{u})$, which by \autoref{linindepcupcap} requires $\euler{u}' = \overline{\euler{v}}$ and $\euler{v}' = \overline{\euler{u}}$.
	\end{proof}

    We now show that the Jones--Wenzl projectors, as well as the basis obtained from them in \autoref{JWBasis}, can be obtained via M{\" o}bius inversion for the inverse semigroup $T_{n}$, see \cite[Section~9.2]{St}. 
    Recall that a finite inverse monoid $T$ becomes a partial order under the relation $\preccurlyeq$ defined by $x \preccurlyeq y$ if $x = ye$ for $e \in E(T)$. 
    Choosing an idempotent $e_{J}$ for every $\mathcal{J}$-cell $J$ of $T$, there is an isomorphism
    \[
    T \xrightarrow[\sim]{\Psi} \prod_{J} \mathbf{Mat}_{|E(J)|}(\Bbbk G_{e_{J}})
    \]
    of $\Bbbk$-algebras, where $G_{e_{J}}$ is the group of units of the monoid $eTe$, and $E(J)$ is the set of idempotents in $J$. 
    As we have observed previously, in the case of $T_{n}$, these groups are all trivial, which gives us an alternative way of establishing semisimplicity of $\on{End}_{\mathcal{TL}_{0}(\Bbbk)}(\obj{n})$. 
    Choosing again the enumeration of \autoref{matrixlabels}, we denote by $[\euler{x}]$ the element $\Psi^{-1}(E_{ab})$, where $\euler{x} = \euler{u}_{a} \circ \overline{\euler{u}_{b}}$. 
    The set $\setj{[\euler{x}] \; | \; \euler{x} \text{ is a TL-diagram with } \thru(\euler{x}) = n}$ is a basis for $\on{End}_{\mathcal{TL}_{0}(\Bbbk)}(\obj{n})$. 
    By \cite[Theorem~9.3]{St}:
    \begin{equation}\label{bracketsdefined}
        \euler{x} = \sum_{y \leq x} [y] = \sum_{\exists \euler{e}: \euler{y} = \euler{xe}} [\euler{y}].
    \end{equation}
    
    \begin{lemma}\label{inversemonoidorder}
     Let $\euler{x} \in T_{n}\setminus\setj{\ast}$ and write $\euler{x} = \euler{u} \circ \euler{v}$. 
     Then $\euler{y} \in T_{n}\setminus\setj{\ast}$ is of the form $\euler{x} \euler{e}$ for $\euler{e} \in E(T_{n})\setminus\setj{\ast}$ if and only if $\euler{y} = \euler{u} \circ \euler{e}' \circ \euler{v}$, where $\euler{e}' \in E(T_{\thru(\euler{x})}\setminus\setj{\ast})$. 
     
     Moreover, $\euler{e}'$ is unique, so $\setj{\euler{y} \in T_{n}: \exists \euler{e} \text{s.t. }\euler{y} = \euler{xe}} = \setj{\euler{u}\circ \euler{e}' \circ \euler{v} \; | \; \euler{e}' \in E(T_{\thru(\euler{x})}\setminus\setj{\ast})}$.
    \end{lemma}

    \begin{proof}
     Since $\euler{e}$ is idempotent, we can write $\euler{e} = \overline{\euler{v}'} \circ \euler{v}'$, and $\euler{x} \euler{e} \neq \ast$ if and only if two strands being capped in $\euler{v}$ implies that they are also capped in $\euler{v}'$, i.e. $K(\euler{v})\subseteq K(\euler{v}')$. 
     In that case we have $\euler{v} \circ \euler{e} = \euler{e}_{\thru(n)}\circ \euler{v}$, where $\euler{e}_{\thru(\euler{v})}$ is $\euler{v} \circ \overline{\euler{v}}$, which clearly is idempotent. 
     The result follows.
    \end{proof}

    \begin{proposition}\label{changeofbasistohats}
       For $\euler{x} = \euler{u} \circ \euler{v} \in \on{End}_{\mathcal{TL}_{0}(\Bbbk)}(\obj{n})$, we have $\euler{x} = \sum_{y \leq x} \widehat{\euler{y}}$. 
    \end{proposition}

    \begin{proof}
        We have
        \[
        \begin{aligned}
        \euler{u} \circ \euler{v} &= \euler{u} \circ \on{id}_{\obj{\thru(\euler{x})}} \circ \euler{v} 
        = \euler{u} \circ \left(\sum_{\euler{v}' \in\mathscr{D}_{\thru(\euler{x})}} \overline{\euler{v}'} \circ \euler{j}_{\thru(\euler{v}')} \circ \euler{v}'\right) \circ \euler{v}
        =\sum_{\euler{v}' \in\mathscr{D}_{\thru(\euler{x})}} \euler{u} \circ \overline{\euler{v}'} \circ \euler{j}_{\thru(\euler{v}')} \circ \euler{v}' \circ \euler{v}\\
        &=\sum_{\euler{v}' \in\mathscr{D}_{\thru(\euler{x})}} \widehat{\euler{u} \circ (\overline{\euler{v}'} \circ \euler{v}') \circ \euler{v}}
        =\sum_{\euler{e}' \in E(T_{\thru(\euler{x})}\setminus\setj{\ast})} \widehat{\euler{u} \circ \euler{e}' \circ \euler{v}}
        =\sum_{\euler{y} \leq \euler{x}} \widehat{\euler{y}}.
        \end{aligned}
        \]
        The second equality follows from \autoref{algebraiso}, and the last equality follows from \autoref{inversemonoidorder}.
    \end{proof}
 
    \begin{corollary}
        For $\euler{x} \in \on{End}_{\mathcal{TL}_{0}(\Bbbk)}(\obj{n})$, we have $\widehat{\euler{x}} = [x]$.
    \end{corollary}
    \begin{proof}
        Comparing \autoref{changeofbasistohats} with \autoref{bracketsdefined}, we find that the change of basis matrix from the basis given by Temperley--Lieb diagrams to $\setj{\widehat{\euler{x}} \; | \; \euler{x} = \euler{u} \circ \euler{v}}$ coincides with that from Temperley--Lieb diagrams to $\setj{[\euler{x}] \; | \; \euler{x} = \euler{u} \circ \euler{v}}$, so the latter two bases coincide.
    \end{proof}
    
	\section{Equivalence of $\mathfrak{sl}_2\mathbf{-Crys}$ and $\mathbf{CrysTL}$}\label{sec:equiv}
	
	To fix notation, let $B_\lambda = \{b_0, b_1, \dots, b_\lambda\}$ for each indecomposable $B_\lambda$ in $\mathfrak{sl}_2\mathbf{-Crys}$ of highest weight $\lambda\in \Lambda_+ = \mathbb N$, where $b_0$ is of weight $\lambda$ weight and $f(b_i)=b_{i+1}$ for $i=0,1,\dots,\lambda-1$ and $f(b_\lambda) = 0$. 
    For simplicity, let us also write $B:=B_1$, the crystal basis of the defining 2-dimensional representation.
	
	\subsection{Equivalence as Monoidal Categories}
	
	Define the following two morphisms in $\mathfrak{sl}_2\mathbf{-Crys}$: $\alpha$ is the embedding of $B_0$ into $B\otimes B \cong B_0\oplus B_2$, and $\beta$ is the projection of $B\otimes B$ onto $B_0$.
	On bases, $\alpha$ and $\beta$ are define explicitly by
	$$\alpha(b_0)=b_0\otimes b_1, \quad \beta(b_0\otimes b_1)=b_0, \quad \beta(b_0\otimes b_0)=\beta(b_1\otimes b_0)=\beta(b_1\otimes b_1)=0.$$
	
	Now, we can define a monoidal functor $F: \mathbf{CrysTL} \to \mathfrak{sl}_2\mathbf{-Crys}$ on generators by
    \begin{align*}
        \obj{1}&\mapsto B\\
        \on{cup}&\mapsto \alpha\\
        \on{cap}&\mapsto \beta,
    \end{align*}
	and extend it monoidally and linearly and to direct sums and summands on all of $\mathbf{CrysTL}$, so that $\obj{n}\mapsto B^{\otimes n}$; in particular, $\obj{0}\mapsto B_0$.
	We will show that $F$ is an equivalence of monoidal categories.

    First we shall need the following lemmas:

    
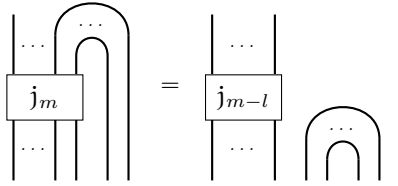
\begin{lemma}\label{JWcaps}
        \[\begin{tikzpicture}[scale=0.55]
	\begin{pgfonlayer}{nodelayer}
		\node [style=jonesrectangle] (0) at (-10, 0.5) {$\euler{j}_m$};
		\node [style=none] (1) at (-9.25, -1.5) {};
		\node [style=none] (3) at (-10.75, -1.5) {};
		\node [style=none] (4) at (-10.75, 2.5) {};
		\node [style=none] (6) at (-9.25, 1.5) {};
		\node [style=none] (10) at (-8.5, 1.5) {};
		\node [style=none] (11) at (-8, 2) {};
		\node [style=none] (12) at (-8.5, -1.5) {};
		\node [style=none] (13) at (-8, -1.5) {};
		\node [style=none] (14) at (-7, 0.75) {$=$};
		\node [style=none] (15) at (-9.75, -1.5) {};
		\node [style=none] (16) at (-9.75, 2) {};
		\node [style=none] (17) at (-10.25, -0.75) {\tiny\dots};
		\node [style=none] (21) at (-10.25, 1.75) {\tiny\dots};
		\node [style=none] (25) at (-8.875, 2.25) {\tiny\dots};
		\node [style=jonesrectangle] (26) at (-5.25, 0.5) {$\euler{j}_{m-l}$};
		\node [style=none] (27) at (-3.25, -1.5) {};
		\node [style=none] (28) at (-6, -1.5) {};
		\node [style=none] (29) at (-6, 2.5) {};
		\node [style=none] (30) at (-3.25, -1) {};
		\node [style=none] (31) at (-2.5, -1) {};
		\node [style=none] (32) at (-2, -0.5) {};
		\node [style=none] (33) at (-2.5, -1.5) {};
		\node [style=none] (34) at (-2, -1.5) {};
		\node [style=none] (35) at (-3.75, -1.5) {};
		\node [style=none] (36) at (-3.75, -0.5) {};
		\node [style=none] (37) at (-5.25, -0.75) {\tiny\dots};
		\node [style=none] (38) at (-5.25, 1.75) {\tiny\dots};
		\node [style=none] (39) at (-2.875, -0.25) {\tiny\dots};
		\node [style=none] (40) at (-4.5, -1.5) {};
		\node [style=none] (41) at (-4.5, 2.5) {};
		\node [style=none] (42) at (-1.5, -1.5) {.};
	\end{pgfonlayer}
	\begin{pgfonlayer}{edgelayer}
		\draw [style=thickstrand] (4.center) to (3.center);
		\draw [style=thickstrand] (1.center) to (6.center);
		\draw [style=thickstrand] (16.center) to (15.center);
		\draw [style=thickstrand] (12.center) to (10.center);
		\draw [style=thickstrand] [bend left=90, looseness=2.00] (6.center) to (10.center);
		\draw [style=thickstrand] (11.center) to (13.center);
		\draw [style=thickstrand] [bend left=90, looseness=1.50] (16.center) to (11.center);
		\draw [style=thickstrand] (29.center) to (28.center);
		\draw [style=thickstrand] (27.center) to (30.center);
		\draw [style=thickstrand] (36.center) to (35.center);
		\draw [style=thickstrand] (33.center) to (31.center);
		\draw [style=thickstrand] [bend left=90, looseness=2.00] (30.center) to (31.center);
		\draw [style=thickstrand] (32.center) to (34.center);
		\draw [style=thickstrand] [bend left=90, looseness=1.50] (36.center) to (32.center);
		\draw [style=thickstrand] (41.center) to (40.center);
	\end{pgfonlayer}
\end{tikzpicture}
\]
    \end{lemma}
    \begin{proof}
        The statement follows directly from applying \autoref{def:JW} and observing that any term in the expression for $\euler{j}_m$ involving cups in the strands $m-l+1,\dots m$ is killed by the zig-zag relation upon $l$-hooking so that the remaining terms are exactly of the form $\euler{j}_{m-l} \otimes \on{id}_l$.
    \end{proof}
    
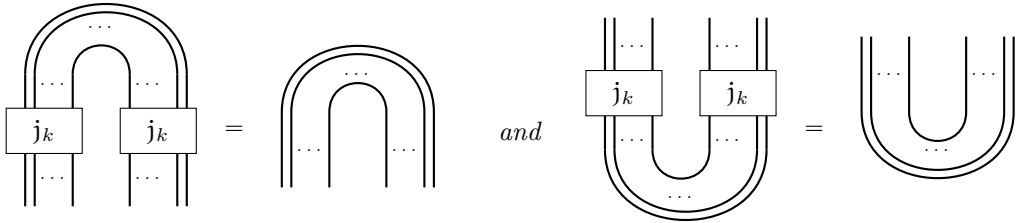
    \begin{lemma}\label{2JWwithCaps}
        \[\begin{tikzpicture}[scale=0.5]
	\begin{pgfonlayer}{nodelayer}
		\node [style=none] (14) at (1.25, 0.5) {$=$};
		\node [style=jonesrectangle] (43) at (-3.75, 0.5) {$\euler{j}_k$};
		\node [style=none] (45) at (-4, -1.5) {};
		\node [style=none] (46) at (-4, 2) {};
		\node [style=none] (54) at (-3.5, -0.75) {\tiny\dots};
		\node [style=none] (55) at (-3.5, 1.75) {\tiny\dots};
		\node [style=jonesrectangle] (56) at (-0.75, 0.5) {$\euler{j}_k$};
		\node [style=none] (57) at (-1.5, -1.5) {};
		\node [style=none] (58) at (-1.5, 2) {};
		\node [style=none] (59) at (0, -1.5) {};
		\node [style=none] (60) at (0, 2) {};
		\node [style=none] (61) at (-1, -0.75) {\tiny\dots};
		\node [style=none] (62) at (-1, 1.75) {\tiny\dots};
		\node [style=none] (63) at (-2.25, 3.25) {\tiny\dots};
		\node [style=none] (79) at (8.75, 0.5) {and};
		\node [style=none] (122) at (-4.25, -1.5) {};
		\node [style=none] (123) at (-0.25, -1.5) {};
		\node [style=none] (124) at (-0.25, 2) {};
		\node [style=none] (125) at (-4.25, 2) {};
		\node [style=none] (127) at (-3, -1.5) {};
		\node [style=none] (128) at (-3, 2) {};
		\node [style=none] (130) at (2.5, -1) {};
		\node [style=none] (131) at (2.5, 1) {};
		\node [style=none] (133) at (3.25, 0) {\tiny\dots};
		\node [style=none] (135) at (5.25, -1) {};
		\node [style=none] (136) at (5.25, 1) {};
		\node [style=none] (137) at (6.5, -1) {};
		\node [style=none] (138) at (6.5, 1) {};
		\node [style=none] (140) at (5.75, 0) {\tiny\dots};
		\node [style=none] (141) at (4.5, 2) {\tiny\dots};
		\node [style=none] (142) at (2.75, -1) {};
		\node [style=none] (143) at (6.25, -1) {};
		\node [style=none] (144) at (6.25, 1) {};
		\node [style=none] (145) at (2.75, 1) {};
		\node [style=none] (146) at (3.75, -1) {};
		\node [style=none] (147) at (3.75, 1) {};
		\node [style=none] (148) at (16.5, 0.5) {$=$};
		\node [style=jonesrectangle] (149) at (11.5, 1.5) {$\euler{j}_k$};
		\node [style=none] (150) at (11.25, 3.5) {};
		\node [style=none] (151) at (11.25, 0) {};
		\node [style=none] (152) at (11.75, 2.75) {\tiny\dots};
		\node [style=none] (153) at (11.75, 0.25) {\tiny\dots};
		\node [style=jonesrectangle] (154) at (14.5, 1.5) {$\euler{j}_k$};
		\node [style=none] (155) at (13.75, 3.5) {};
		\node [style=none] (156) at (13.75, 0) {};
		\node [style=none] (157) at (15.25, 3.5) {};
		\node [style=none] (158) at (15.25, 0) {};
		\node [style=none] (159) at (14.25, 2.75) {\tiny\dots};
		\node [style=none] (160) at (14.25, 0.25) {\tiny\dots};
		\node [style=none] (161) at (13, -1.25) {\tiny\dots};
		\node [style=none] (162) at (11, 3.5) {};
		\node [style=none] (163) at (15, 3.5) {};
		\node [style=none] (164) at (15, 0) {};
		\node [style=none] (165) at (11, 0) {};
		\node [style=none] (166) at (12.25, 3.5) {};
		\node [style=none] (167) at (12.25, 0) {};
		\node [style=none] (168) at (17.75, 3) {};
		\node [style=none] (169) at (17.75, 1) {};
		\node [style=none] (170) at (18.5, 2) {\tiny\dots};
		\node [style=none] (171) at (20.5, 3) {};
		\node [style=none] (172) at (20.5, 1) {};
		\node [style=none] (173) at (21.75, 3) {};
		\node [style=none] (174) at (21.75, 1) {};
		\node [style=none] (175) at (21, 2) {\tiny\dots};
		\node [style=none] (176) at (19.75, 0) {\tiny\dots};
		\node [style=none] (177) at (18, 3) {};
		\node [style=none] (178) at (21.5, 3) {};
		\node [style=none] (179) at (21.5, 1) {};
		\node [style=none] (180) at (18, 1) {};
		\node [style=none] (181) at (19, 3) {};
		\node [style=none] (182) at (19, 1) {};
	\end{pgfonlayer}
	\begin{pgfonlayer}{edgelayer}
		\draw [style=thickstrand] (46.center) to (45.center);
		\draw [style=thickstrand] (58.center) to (57.center);
		\draw [style=thickstrand] (60.center) to (59.center);
		\draw [style=thickstrand] (122.center) to (125.center);
		\draw [style=thickstrand] (123.center) to (124.center);
		\draw [style=thickstrand] (127.center) to (128.center);
		\draw [style=thickstrand] [bend left=90, looseness=1.75] (128.center) to (58.center);
		\draw [style=thickstrand] (131.center) to (130.center);
		\draw [style=thickstrand] (136.center) to (135.center);
		\draw [style=thickstrand] (138.center) to (137.center);
		\draw [style=thickstrand] [bend left=90, looseness=1.50] (131.center) to (138.center);
		\draw [style=thickstrand] (142.center) to (145.center);
		\draw [style=thickstrand] (143.center) to (144.center);
		\draw [style=thickstrand] [bend left=90, looseness=1.50] (145.center) to (144.center);
		\draw [style=thickstrand] (146.center) to (147.center);
		\draw [style=thickstrand] [bend left=90, looseness=1.75] (147.center) to (136.center);
		\draw [style=thickstrand] [bend left=90, looseness=1.50] (46.center) to (124.center);
		\draw [style=thickstrand] [bend left=90, looseness=1.50] (125.center) to (60.center);
		\draw [style=thickstrand] (151.center) to (150.center);
		\draw [style=thickstrand] (156.center) to (155.center);
		\draw [style=thickstrand] (158.center) to (157.center);
		\draw [style=thickstrand] (162.center) to (165.center);
		\draw [style=thickstrand] (163.center) to (164.center);
		\draw [style=thickstrand] (166.center) to (167.center);
		\draw [style=thickstrand] [bend right=90, looseness=1.75] (167.center) to (156.center);
		\draw [style=thickstrand] (169.center) to (168.center);
		\draw [style=thickstrand] (172.center) to (171.center);
		\draw [style=thickstrand] (174.center) to (173.center);
		\draw [style=thickstrand] [bend right=90, looseness=1.50] (169.center) to (174.center);
		\draw [style=thickstrand] (177.center) to (180.center);
		\draw [style=thickstrand] (178.center) to (179.center);
		\draw [style=thickstrand] [bend right=90, looseness=1.50] (180.center) to (179.center);
		\draw [style=thickstrand] (181.center) to (182.center);
		\draw [style=thickstrand] [bend right=90, looseness=1.75] (182.center) to (172.center);
		\draw [style=thickstrand] [bend right=90, looseness=1.50] (151.center) to (164.center);
		\draw [style=thickstrand] [bend right=90, looseness=1.50] (165.center) to (158.center);
	\end{pgfonlayer}
\end{tikzpicture}
\]
    \end{lemma}

    
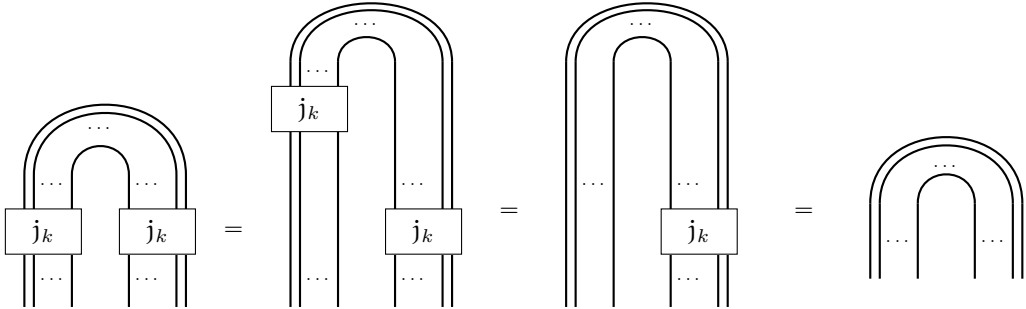
\begin{proof}
        \[\begin{tikzpicture}[scale=0.5]
	\begin{pgfonlayer}{nodelayer}
		\node [style=none] (0) at (1.25, 0.5) {$=$};
		\node [style=jonesrectangle] (1) at (-3.75, 0.5) {$\euler{j}_k$};
		\node [style=none] (2) at (-4, -1.5) {};
		\node [style=none] (3) at (-4, 2) {};
		\node [style=none] (4) at (-3.5, -0.75) {\tiny\dots};
		\node [style=none] (5) at (-3.5, 1.75) {\tiny\dots};
		\node [style=jonesrectangle] (6) at (-0.75, 0.5) {$\euler{j}_k$};
		\node [style=none] (7) at (-1.5, -1.5) {};
		\node [style=none] (8) at (-1.5, 2) {};
		\node [style=none] (9) at (0, -1.5) {};
		\node [style=none] (10) at (0, 2) {};
		\node [style=none] (11) at (-1, -0.75) {\tiny\dots};
		\node [style=none] (12) at (-1, 1.75) {\tiny\dots};
		\node [style=none] (13) at (-2.25, 3.25) {\tiny\dots};
		\node [style=none] (14) at (-4.25, -1.5) {};
		\node [style=none] (15) at (-0.25, -1.5) {};
		\node [style=none] (16) at (-0.25, 2) {};
		\node [style=none] (17) at (-4.25, 2) {};
		\node [style=none] (18) at (-3, -1.5) {};
		\node [style=none] (19) at (-3, 2) {};
		\node [style=none] (20) at (18, -0.75) {};
		\node [style=none] (21) at (18, 1.25) {};
		\node [style=none] (22) at (18.75, 0.25) {\tiny\dots};
		\node [style=none] (23) at (20.75, -0.75) {};
		\node [style=none] (24) at (20.75, 1.25) {};
		\node [style=none] (25) at (22, -0.75) {};
		\node [style=none] (26) at (22, 1.25) {};
		\node [style=none] (27) at (21.25, 0.25) {\tiny\dots};
		\node [style=none] (28) at (20, 2.25) {\tiny\dots};
		\node [style=none] (29) at (18.25, -0.75) {};
		\node [style=none] (30) at (21.75, -0.75) {};
		\node [style=none] (31) at (21.75, 1.25) {};
		\node [style=none] (32) at (18.25, 1.25) {};
		\node [style=none] (33) at (19.25, -0.75) {};
		\node [style=none] (34) at (19.25, 1.25) {};
		\node [style=none] (35) at (8.5, 1) {$=$};
		\node [style=jonesrectangle] (36) at (3.25, 3.75) {$\euler{j}_k$};
		\node [style=none] (37) at (3, -1.5) {};
		\node [style=none] (38) at (3, 5) {};
		\node [style=none] (39) at (3.5, -0.75) {\tiny\dots};
		\node [style=none] (40) at (3.5, 4.75) {\tiny\dots};
		\node [style=jonesrectangle] (41) at (6.25, 0.5) {$\euler{j}_k$};
		\node [style=none] (42) at (5.5, -1.5) {};
		\node [style=none] (43) at (5.5, 5) {};
		\node [style=none] (44) at (7, -1.5) {};
		\node [style=none] (45) at (7, 5) {};
		\node [style=none] (46) at (6, -0.75) {\tiny\dots};
		\node [style=none] (47) at (6, 1.75) {\tiny\dots};
		\node [style=none] (48) at (4.75, 6) {\tiny\dots};
		\node [style=none] (49) at (2.75, -1.5) {};
		\node [style=none] (50) at (6.75, -1.5) {};
		\node [style=none] (51) at (6.75, 5) {};
		\node [style=none] (52) at (2.75, 5) {};
		\node [style=none] (53) at (4, -1.5) {};
		\node [style=none] (54) at (4, 5) {};
		\node [style=none] (89) at (16.25, 1) {$=$};
		\node [style=none] (91) at (10.25, -1.5) {};
		\node [style=none] (92) at (10.25, 5) {};
		\node [style=none] (94) at (10.75, 1.75) {\tiny\dots};
		\node [style=jonesrectangle] (95) at (13.5, 0.5) {$\euler{j}_k$};
		\node [style=none] (96) at (12.75, -1.5) {};
		\node [style=none] (97) at (12.75, 5) {};
		\node [style=none] (98) at (14.25, -1.5) {};
		\node [style=none] (99) at (14.25, 5) {};
		\node [style=none] (100) at (13.25, -0.75) {\tiny\dots};
		\node [style=none] (101) at (13.25, 1.75) {\tiny\dots};
		\node [style=none] (102) at (12, 6) {\tiny\dots};
		\node [style=none] (103) at (10, -1.5) {};
		\node [style=none] (104) at (14, -1.5) {};
		\node [style=none] (105) at (14, 5) {};
		\node [style=none] (106) at (10, 5) {};
		\node [style=none] (107) at (11.25, -1.5) {};
		\node [style=none] (108) at (11.25, 5) {};
	\end{pgfonlayer}
	\begin{pgfonlayer}{edgelayer}
		\draw [style=thickstrand] (3.center) to (2.center);
		\draw [style=thickstrand] (8.center) to (7.center);
		\draw [style=thickstrand] (10.center) to (9.center);
		\draw [style=thickstrand] (14.center) to (17.center);
		\draw [style=thickstrand] (15.center) to (16.center);
		\draw [style=thickstrand] (18.center) to (19.center);
		\draw [style=thickstrand] [bend left=90, looseness=1.75] (19.center) to (8.center);
		\draw [style=thickstrand] (21.center) to (20.center);
		\draw [style=thickstrand] (24.center) to (23.center);
		\draw [style=thickstrand] (26.center) to (25.center);
		\draw [style=thickstrand] [bend left=90, looseness=1.50] (21.center) to (26.center);
		\draw [style=thickstrand] (29.center) to (32.center);
		\draw [style=thickstrand] (30.center) to (31.center);
		\draw [style=thickstrand] [bend left=90, looseness=1.50] (32.center) to (31.center);
		\draw [style=thickstrand] (33.center) to (34.center);
		\draw [style=thickstrand] [bend left=90, looseness=1.75] (34.center) to (24.center);
		\draw [style=thickstrand] [bend left=90, looseness=1.50] (3.center) to (16.center);
		\draw [style=thickstrand] [bend left=90, looseness=1.50] (17.center) to (10.center);
		\draw [style=thickstrand] (38.center) to (37.center);
		\draw [style=thickstrand] (43.center) to (42.center);
		\draw [style=thickstrand] (45.center) to (44.center);
		\draw [style=thickstrand] (49.center) to (52.center);
		\draw [style=thickstrand] (50.center) to (51.center);
		\draw [style=thickstrand] (53.center) to (54.center);
		\draw [style=thickstrand] [bend left=90, looseness=1.50] (54.center) to (43.center);
		\draw [style=thickstrand] [bend left=90, looseness=1.25] (38.center) to (51.center);
		\draw [style=thickstrand] [bend left=270, looseness=1.25] (45.center) to (52.center);
		\draw [style=thickstrand] (92.center) to (91.center);
		\draw [style=thickstrand] (97.center) to (96.center);
		\draw [style=thickstrand] (99.center) to (98.center);
		\draw [style=thickstrand] (103.center) to (106.center);
		\draw [style=thickstrand] (104.center) to (105.center);
		\draw [style=thickstrand] (107.center) to (108.center);
		\draw [style=thickstrand] [bend left=90, looseness=1.50] (108.center) to (97.center);
		\draw [style=thickstrand] [bend left=90, looseness=1.25] (92.center) to (105.center);
		\draw [style=thickstrand] [bend left=270, looseness=1.25] (99.center) to (106.center);
	\end{pgfonlayer}
\end{tikzpicture}
\]
        where the second and third equalities follow by \autoref{JWcaps}. The second statement is exactly analogous.
    \end{proof}    
    Recall that a \emph{tensor ideal} in a $\Bbbk$-linear category $\mathcal C$ consists of a collection of $\Bbbk$-submodules $\mathcal{I}(X,Y) \subseteq \on{Hom}_{\mathcal{C}}(X,Y)$, which is stable under precomposition, postcomposition, and tensoring with arbitrary morphisms in $\mathcal{C}$. 

    Given an ideal $\mathcal{I}$ of $\mathcal{C}$, the quotient category $\mathcal C/\mathcal I$ is by definition the category with the same objects as $\mathcal C$ and whose morphisms are given by $$\on{Hom}_{\mathcal {C/I}}(X,Y) = \on{Hom}_{\mathcal C}(X,Y)/\mathcal{I}(X,Y).$$
    Then, the functor $\mathcal C\to \mathcal {C/I}$ is a monoidal functor with kernel $\mathcal I$, and indeed the kernel of any monoidal functor is a tensor ideal. 
    To show that a monoidal functor is faithful, it suffices to show that its kernel is the zero ideal.
    
    \begin{proposition}
\label{tensorIdeals}
        $\mathbf{CrysTL}$ has no non-trivial tensor ideals.
    \end{proposition}

    \begin{proof}
        Suppose $\mathcal I$ is a nonzero tensor ideal of $\mathbf{CrysTL}$, and let $f$ be a nonzero morphism of $\mathcal I$.
        Writing $f = \sum_{i}c_i\widehat{\euler{x}_i}$, where $\{\widehat{\euler{x}_i}\}$ is the semisimple basis of \autoref{JWBasis}, where $\euler{x_i} = \euler{u}_i \circ \euler{v}_i$ and $c_i\in \Bbbk$ are nonzero scalars.
        Assume without loss of generality that $\mathtt{th}(\euler{x}_i)\leq \mathtt{th}(\euler{x}_j)$ whenever $i<j$. 
        Note that $\overline{\euler{u_1}}\circ f \circ \overline{\euler{v_1}} = c_1 \overline{\euler{u_1}}\circ \widehat{\euler{x}_1} \circ \overline{\euler{v_1}} = c_1 \euler{j}_{\mathtt{th}(\euler{x}_1)}$ since all other terms are annihilated by the zig-zag relation or by attaching cups or caps to a Jones--Wenzl projector.
        Indeed, if $\overline{\euler{u_1}}\circ x_i \circ \overline{\euler{v_1}} \neq 0$, we must have $K(\euler{v}_i)\subseteq K(\euler{v_1})$ and   $\overline{K}(\euler{u}_i)\subseteq \overline{K}(\euler{u_1})$; on the other hand, by assumption, $\mathtt{th}(x_1)\leq\mathtt{th}(x_i)$ yielding equalities of the sets of cups and caps so that $x_i=x_1$.
        
        It follows that $\euler{j}_k\in \mathcal I$ for some $k=\mathtt{th}(\euler{x}_1)\geq 0$ (where $\euler{j}_0 = \on{id}_{\obj{0}}$ by convention), hence so is $\euler{j}_k\otimes \euler{j}_k$. 
        By \autoref{2JWwithCaps} pre-composing $\euler{j}_k\otimes \euler{j}_k$ with $k$ nested cups and post-composing with $k$ nested caps shows that $\on{id}_{\obj{0}}\in\mathcal I$. 
        Therefore, $\mathcal I$ is all of $\mathbf{CrysTL}$.
 \end{proof}

In \autoref{tensorFunctorsAreFaithful}, we assume the the codomain of the monoidal functor is non-zero, i.e. not the terminal category, as it is the only monoidal category in which the unit object is the zero object.
    \begin{corollary}\label{tensorFunctorsAreFaithful}
        Any $\Bbbk$-linear monoidal functor from $\mathbf{CrysTL}$ to a non-zero $\Bbbk$-linear monoidal category $\mathcal{C}$ is faithful.
    \end{corollary}

    \begin{proof}
        The kernel of such a functor is a \emph{proper} tensor ideal in $\mathbf{CrysTL}$ since it preserves the unit object, which by assumption is non-zero in $\mathcal{C}$; the statement immediately follows by \autoref{tensorIdeals}.
    \end{proof}

    \begin{corollary}
	   The functor $F: \mathbf{CrysTL} \rightarrow \mathfrak{sl}_2\mathbf{-Crys}$ is faithful.
    \end{corollary}
	
	\begin{proposition}
	   The functor $F$ is full.
	\end{proposition}
		
	\begin{proof}
		Since $F$ is faithful, we have for every $m,n\in\mathbb N$ an injective linear map $$\on{Hom}_{\mathbf{CrysTL}}(\obj{m},\obj{n}) \to \on{Hom}_{\mathfrak{sl}_2\mathbf{-Crys}}(B^{\otimes m},B^{\otimes n}),$$ so it suffices to show that 
		$$\dim \on{Hom}_{\mathbf{CrysTL}}(\obj{m},\obj{n}) = \dim \on{Hom}_{\mathfrak{sl}_2\mathbf{-Crys}}(B^{\otimes m},B^{\otimes n}).$$
		But this is indeed true since the left-hand side is counting crossingless matchings from $m$ points to $n$ points and the right handside is equal to $\dim \on{Hom}_{\textbf{Rep}(\mathfrak{sl}_2)}(V^{\otimes m},V^{\otimes n})$, where $V$ is the defining representation, and both are known to be equal to the Catalan number $C_{\frac{m+n}{2}}$.
	\end{proof}
	
	Finally, note that $F$ is clearly essentially surjective on objects since any indecomposable $B_\lambda$ in $\mathfrak{sl}_2\mathbf{-Crys}$ is a summand in $F([\lambda]) = B^{\otimes\lambda}$, and hence in the image of $F$.
	Furthermore, $F$ is monoidal by construction with the obvious structure maps.
	Thus, we have shown:
	
	\begin{theorem}\label{thm:equivalence}
		The functor $F: \mathbf{CrysTL} \to \mathfrak{sl}_2\mathbf{-Crys}$ is an equivalence of monoidal categories.
	\end{theorem}
    \begin{remark}\label{rem:alternativeequivalence}
        If $\Bbbk$ is algebraically closed, we have an alternative way to prove fullness: first, observe that by semisimplicity and Schur's Lemma any morphism in $\mathfrak{sl}_2-\mathbf{Crys}$ can be written as a linear combination of projections onto followed by embeddings into irreducible summands; then, note that such projections and embeddings are in the image of $F$, which is demonstrated in \autoref{ProjEmb} below. 
        In other words, \autoref{ProjEmb} shows that $F$ is a functor of semisimple categories which induces an isomorphism of Grothendieck rings, and hence it is an equivalence.
    \end{remark}

  \begin{corollary}\label{notbraided}
      The category $\mathbf{CrysTL}$ does not admit a braiding.
  \end{corollary}

  \begin{proof}
      From \cite[Section 2.3]{HKam}, we know that the category $\mathfrak{sl}_2\mathbf{-Crys}$ does not admit a braiding. 
      The claim now follows directly from \autoref{thm:equivalence}.
  \end{proof}

    \subsection{A Commutor for the Crystal Temperley--Lieb Category}\label{sec:commutor}
	
	We will now define a commutor for $\mathcal{TL}_0(\Bbbk)$ that endows it with the structure of a coboundary category. 
	
    Recall that $\mathscr{D}_{\mobijd{n}}$ is the set of all possible cap diagrams in $\on{Hom}_{\mathcal{TL}_0(\Bbbk)}(\obj{n},\obj{m})$ for some $m\leq n$.
	Given $\euler{x}\in\mathscr{D}_{\mobijd{m+n}}$, let $\euler{x}_{\leq m}$ (resp. $\euler{x}_{>m}$) be the diagram obtained by deleting strands $m+1,\dots,n$ (resp. $1,\dots,m)$) in the diagram of $\euler{x}$ and turning any remaining half-caps into through strands.
	We call $\euler{x}$ \emph{$l$-hooked at $m$} if $\thru(\euler{x})+2l = \thru(\euler{x}_{\leq m})+\thru(\euler{x}_{>m})$.
	We also use the notation $\mathtt{hk}_m\euler{x}:=l$ whenever $\euler{x}$ is $l$-hooked at $m$.
	
	Conversely, starting with two diagrams $\euler{y}\in\mathscr{D}_{\mobijd{m}}$ and $\euler{y}'\in \mathscr{D}_{\mobijd{n}}$, for each $l\leq\min\{\thru(\euler{y}),\thru(\euler{y}')\}$, there is a unique $l$-hooked at $m$ diagram $\euler{x}\in \mathscr{D}_{\mobijd{m+n}}$ such that $\euler{x}_{\leq m} = \euler{y}$ and $\euler{x}_{>m} = \euler{y}'$, namely the diagram acquired by concatenating $\euler{y}$ and $\euler{y}'$ and inductively joining the rightmost through strand of $\euler{y}$ with the left-most through strand of $\euler{y}'$ to form a cap $l$ times; we call this operation \emph{$l$-hooking at $m$} and denote it by $\euler{y}\odot_l\euler{y}'$).
	
	For each $n,m\in\mathbb N$, we define a map $\kappa_{\mobij{m},\mobij{n}}:\mathscr{D}_{\mobijd{m+n}}\to \mathscr{D}_{\mobijd{m+n}}$ by
	$\kappa_{\mobij{m},\mobij{n}}(\euler{x})=
	\euler{x}_{> m}\odot_{\mathtt{hk}_m(\euler{x})} \euler{x}_{\leq m}$.
	For example,
	\[\begin{tikzpicture}
	\begin{pgfonlayer}{nodelayer}
		\node [style=none] (0) at (-5.25, -0.75) {};
		\node [style=none] (1) at (-5, -0.75) {};
		\node [style=none] (2) at (-2.5, -0.75) {};
		\node [style=none] (3) at (-1, -0.75) {};
		\node [style=none] (7) at (-4.75, -0.75) {};
		\node [style=none] (8) at (-4.5, -0.75) {};
		\node [style=none] (9) at (-4.25, -0.75) {};
		\node [style=none] (10) at (-4, -0.75) {};
		\node [style=none] (11) at (-3.75, -0.75) {};
		\node [style=none] (12) at (-5.25, 0.75) {};
		\node [style=none] (13) at (-4.625, -1) {};
		\node [style=none] (14) at (-4.625, 1) {};
		\node [style=none] (16) at (-3.25, 0) {$=$};
		\node [style=none] (17) at (-2.75, -0.75) {};
		\node [style=none] (18) at (-2.25, -0.75) {};
		\node [style=none] (19) at (-2.75, 0.75) {};
		\node [style=none] (20) at (-2.5, 0.75) {};
		\node [style=none] (21) at (-2.25, 0.75) {};
		\node [style=none] (22) at (-1.75, 0) {$\odot_2$};
		\node [style=none] (23) at (-1.25, -0.75) {};
		\node [style=none] (24) at (-0.75, -0.75) {};
		\node [style=none] (25) at (-0.5, -0.75) {};
		\node [style=none] (26) at (-1.25, 0.75) {};
		\node [style=none] (27) at (-1, 0.75) {};
		\node [style=none] (35) at (2, 0) {$\odot_2$};
		\node [style=none] (36) at (2.75, -0.75) {};
		\node [style=none] (37) at (2.5, -0.75) {};
		\node [style=none] (38) at (3, -0.75) {};
		\node [style=none] (39) at (2.5, 0.75) {};
		\node [style=none] (40) at (2.75, 0.75) {};
		\node [style=none] (41) at (3, 0.75) {};
		\node [style=none] (42) at (3.5, 0) {$=$};
		\node [style=none] (43) at (1, -0.75) {};
		\node [style=none] (44) at (0.75, -0.75) {};
		\node [style=none] (45) at (1.25, -0.75) {};
		\node [style=none] (46) at (1.5, -0.75) {};
		\node [style=none] (47) at (0.75, 0.75) {};
		\node [style=none] (48) at (1, 0.75) {};
		\node [style=none] (49) at (4, -0.75) {};
		\node [style=none] (50) at (4.25, -0.75) {};
		\node [style=none] (51) at (4.5, -0.75) {};
		\node [style=none] (52) at (4.75, -0.75) {};
		\node [style=none] (53) at (5, -0.75) {};
		\node [style=none] (54) at (5.25, -0.75) {};
		\node [style=none] (55) at (5.5, -0.75) {};
		\node [style=none] (57) at (4.875, -1) {};
		\node [style=none] (58) at (4.875, 1) {};
		\node [style=none] (59) at (5.5, 0.75) {};
		\node [style=none] (61) at (0, 0) {$\overset{\kappa_{3,4}}{\mapsto}$};
	\end{pgfonlayer}
	\begin{pgfonlayer}{edgelayer}
		\draw [style=thickstrand] [bend left=90, looseness=1.75] (7.center) to (8.center);
		\draw [style=thickstrand] (0.center) to (12.center);
		\draw [style=background] (14.center) to (13.center);
		\draw [style=thickstrand] (17.center) to (19.center);
		\draw [style=thickstrand] (2.center) to (20.center);
		\draw [style=thickstrand] (18.center) to (21.center);
		\draw [style=thickstrand] (23.center) to (26.center);
		\draw [style=thickstrand] (37.center) to (39.center);
		\draw [style=thickstrand] (36.center) to (40.center);
		\draw [style=thickstrand] (38.center) to (41.center);
		\draw [style=thickstrand, bend left=90, looseness=1.50] (1.center) to (9.center);
		\draw [style=thickstrand, bend left=90, looseness=1.75] (10.center) to (11.center);
		\draw [style=thickstrand] (3.center) to (27.center);
		\draw [style=thickstrand, bend left=90, looseness=1.75] (24.center) to (25.center);
		\draw [style=thickstrand] (44.center) to (47.center);
		\draw [style=thickstrand] (43.center) to (48.center);
		\draw [style=thickstrand, bend left=90, looseness=1.75] (45.center) to (46.center);
		\draw [style=thickstrand] [bend left=90, looseness=1.75] (51.center) to (52.center);
		\draw [style=background] (58.center) to (57.center);
		\draw [style=thickstrand, bend left=90, looseness=1.25] (49.center) to (54.center);
		\draw [style=thickstrand, bend left=90, looseness=1.25] (50.center) to (53.center);
		\draw [style=thickstrand] (55.center) to (59.center);
	\end{pgfonlayer}
\end{tikzpicture}
\]
	
	It is easy to see that $\kappa_{\mobij{n},\mobij{m}}\circ\kappa_{\mobij{m},\mobij{n}} = \on{id}_{\mathscr D_{\mobijd{m+n}}}$ and that $\thru(\kappa_{\mobij{m},\mobij{n}}(\euler{x}))=\thru(\euler{x})$. 
	
	Next define a function $\tau_{\mobijt{m},\mobijt{n}}:\mathscr{D}_{\mobijd{m},\mobijd{n}}\to \on{Hom}_{\mathcal{TL}_0(\Bbbk)}(\obj{m}\otimes \obj{n},\obj{n}\otimes \obj{m})$ given by $\tau_{\mobijt{m},\mobijt{n}}(\euler{x}) = \overline{\kappa_{\mobij{m},\mobij{n}}(\euler{x})}\circ \euler{j}_{\mathtt{th(\euler{x})}}\circ \euler{x}$. 
    In particular, it is the max-summand morphism of $\overline{\kappa_{\mobij{m},\mobij{n}}(\euler{x})}\circ \euler{x}$, using the terminology of \autoref{maxsummand}.
    For example, the computation of $\kappa_{3,4}$ on the diagram above shows that 
    
    \[\begin{tikzpicture}
	\begin{pgfonlayer}{nodelayer}
		\node [style=none] (0) at (-5.5, -0.25) {};
		\node [style=none] (1) at (-5.25, -0.25) {};
		\node [style=none] (7) at (-5, -0.25) {};
		\node [style=none] (8) at (-4.75, -0.25) {};
		\node [style=none] (9) at (-4.5, -0.25) {};
		\node [style=none] (10) at (-4.25, -0.25) {};
		\node [style=none] (11) at (-4, -0.25) {};
		\node [style=none] (12) at (-5.5, 0.25) {};
		\node [style=none] (16) at (-3, 0) {$=$};
		\node [style=none] (49) at (-2.25, 0.75) {};
		\node [style=none] (50) at (-2, 0.75) {};
		\node [style=none] (51) at (-1.75, 0.75) {};
		\node [style=none] (52) at (-1.5, 0.75) {};
		\node [style=none] (53) at (-1.25, 0.75) {};
		\node [style=none] (54) at (-1, 0.75) {};
		\node [style=none] (55) at (-0.75, 0.75) {};
		\node [style=none] (62) at (-5.75, 0.5) {};
		\node [style=none] (63) at (-3.75, 0.5) {};
		\node [style=none] (64) at (-6.75, 0) {$\tau_{3,4}$};
		\node [style=none] (65) at (-5.75, -0.5) {};
		\node [style=none] (66) at (-3.75, -0.5) {};
		\node [style=none] (67) at (-2.25, -0.75) {};
		\node [style=none] (68) at (-2, -0.75) {};
		\node [style=none] (69) at (-1.75, -0.75) {};
		\node [style=none] (70) at (-1.5, -0.75) {};
		\node [style=none] (71) at (-1.25, -0.75) {};
		\node [style=none] (72) at (-1, -0.75) {};
		\node [style=none] (73) at (-0.75, -0.75) {};
	\end{pgfonlayer}
	\begin{pgfonlayer}{edgelayer}
		\draw [style=thickstrand] [bend left=90, looseness=1.75] (7.center) to (8.center);
		\draw [style=thickstrand] (0.center) to (12.center);
		\draw [style=thickstrand, bend left=90, looseness=1.50] (1.center) to (9.center);
		\draw [style=thickstrand, bend left=90, looseness=1.75] (10.center) to (11.center);
		\draw [style=thickstrand] [bend right=90, looseness=1.75] (51.center) to (52.center);
		\draw [style=thickstrand, bend right=90, looseness=1.25] (49.center) to (54.center);
		\draw [style=thickstrand, bend right=90, looseness=1.25] (50.center) to (53.center);
		\draw [style=thickstrand] [bend right, looseness=0.75] (62.center) to (65.center);
		\draw [style=thickstrand] [bend left, looseness=0.75] (63.center) to (66.center);
		\draw [style=thickstrand] [bend left=90, looseness=1.75] (69.center) to (70.center);
		\draw [style=thickstrand, bend left=90, looseness=1.50] (68.center) to (71.center);
		\draw [style=thickstrand, bend left=90, looseness=1.75] (72.center) to (73.center);
		\draw [style=thickstrand] [in=-90, out=90] (67.center) to (55.center);
	\end{pgfonlayer}
\end{tikzpicture}
\]
    
	We are finally ready to define the commutor.
	
	\begin{theorem}\label{TLCob}
		The maps $\sigma_{\mobijs{m},\mobijs{n}}\in\on{Hom}_{\mathcal{TL}_0(\Bbbk)}(\obj{m}\otimes \obj{n},\obj{n}\otimes \obj{m})$ given by 
		$$\sigma_{\mobijs{m},\mobijs{n}} = \sum_{\euler{x}\in \mathscr{D}_{\mobijd{m+n}}} \tau_{\mobijt{m},\mobijt{n}}(\euler{x}) = \sum_{\euler{x}\in \mathscr{D}_{\mobijd{m+n}}} \overline{\kappa_{\mobij{m},\mobij{n}}(\euler{x})}\circ \euler{j}_{\mathtt{th(\euler{x})}}\circ \euler{x}$$
		define a natural (in $\obj{m}$ and $\obj{n}$) isomorphism, and satisfy the coboundary commutor axioms, turning $\mathcal{TL}_0(\Bbbk)$ into a coboundary category.
	\end{theorem}
	
	Together with the formulas for $\euler{j}_n$ given in the previous section, one can explicitly compute the commutor in terms of basis diagrams; for example,
	\\
	\begin{center}
		\begin{tikzpicture}[scale=0.5]
	\begin{pgfonlayer}{nodelayer}
		\node [style=none] (1) at (-1, 0) {$\sigma_{\obj{1},\obj{3}} = $};
		\node [style=none] (2) at (1, 2) {};
		\node [style=none] (3) at (1.5, 2) {};
		\node [style=none] (4) at (2, 2) {};
		\node [style=none] (5) at (2.5, 2) {};
		\node [style=none] (6) at (1, -2) {};
		\node [style=none] (7) at (1.5, -2) {};
		\node [style=none] (8) at (2, -2) {};
		\node [style=none] (9) at (2.5, -2) {};
		\node [style=jonesrectangle] (10) at (1.75, 0) {$\euler{j}_4$};
		\node [style=none] (11) at (4.5, 2) {};
		\node [style=none] (12) at (5, 2) {};
		\node [style=none] (13) at (5.5, 2) {};
		\node [style=none] (14) at (6, 2) {};
		\node [style=none] (15) at (4.5, -2) {};
		\node [style=none] (16) at (5, -2) {};
		\node [style=none] (17) at (5.5, -2) {};
		\node [style=none] (18) at (6, -2) {};
		\node [style=jonesrectangle] (19) at (5.25, 0) {$\euler{j}_2$};
		\node [style=none] (56) at (3.5, 0) {$+$};
		\node [style=none] (57) at (8, 2) {};
		\node [style=none] (58) at (8.5, 2) {};
		\node [style=none] (59) at (9, 2) {};
		\node [style=none] (60) at (9.5, 2) {};
		\node [style=none] (61) at (8, -2) {};
		\node [style=none] (62) at (8.5, -2) {};
		\node [style=none] (63) at (9, -2) {};
		\node [style=none] (64) at (9.5, -2) {};
		\node [style=jonesrectangle] (65) at (8.75, 0) {$\euler{j}_2$};
		\node [style=none] (66) at (7, 0) {$+$};
		\node [style=none] (67) at (11.5, 2) {};
		\node [style=none] (68) at (12, 2) {};
		\node [style=none] (69) at (12.5, 2) {};
		\node [style=none] (70) at (13, 2) {};
		\node [style=none] (71) at (11.5, -2) {};
		\node [style=none] (72) at (12, -2) {};
		\node [style=none] (73) at (12.5, -2) {};
		\node [style=none] (74) at (13, -2) {};
		\node [style=jonesrectangle] (75) at (12.25, 0) {$\euler{j}_2$};
		\node [style=none] (76) at (10.5, 0) {$+$};
		\node [style=none] (77) at (15, 2) {};
		\node [style=none] (78) at (15.5, 2) {};
		\node [style=none] (79) at (16, 2) {};
		\node [style=none] (80) at (16.5, 2) {};
		\node [style=none] (81) at (15, -2) {};
		\node [style=none] (82) at (15.5, -2) {};
		\node [style=none] (83) at (16, -2) {};
		\node [style=none] (84) at (16.5, -2) {};
		\node [style=none] (86) at (14, 0) {$+$};
		\node [style=none] (87) at (18.5, 2) {};
		\node [style=none] (88) at (19, 2) {};
		\node [style=none] (89) at (19.5, 2) {};
		\node [style=none] (90) at (20, 2) {};
		\node [style=none] (91) at (18.5, -2) {};
		\node [style=none] (92) at (19, -2) {};
		\node [style=none] (93) at (19.5, -2) {};
		\node [style=none] (94) at (20, -2) {};
		\node [style=none] (96) at (17.5, 0) {$+$};
		\node [style=none] (97) at (1, -4) {};
		\node [style=none] (98) at (1.5, -4) {};
		\node [style=none] (99) at (2, -4) {};
		\node [style=none] (100) at (2.5, -4) {};
		\node [style=none] (101) at (1, -8) {};
		\node [style=none] (102) at (1.5, -8) {};
		\node [style=none] (103) at (2, -8) {};
		\node [style=none] (104) at (2.5, -8) {};
		\node [style=none] (106) at (15, -4) {};
		\node [style=none] (107) at (15.5, -4) {};
		\node [style=none] (108) at (16, -4) {};
		\node [style=none] (109) at (16.5, -4) {};
		\node [style=none] (110) at (15, -8) {};
		\node [style=none] (111) at (15.5, -8) {};
		\node [style=none] (112) at (16, -8) {};
		\node [style=none] (113) at (16.5, -8) {};
		\node [style=none] (115) at (14, -6) {$+$};
		\node [style=none] (116) at (18.5, -4) {};
		\node [style=none] (117) at (19, -4) {};
		\node [style=none] (118) at (19.5, -4) {};
		\node [style=none] (119) at (20, -4) {};
		\node [style=none] (120) at (18.5, -8) {};
		\node [style=none] (121) at (19, -8) {};
		\node [style=none] (122) at (19.5, -8) {};
		\node [style=none] (123) at (20, -8) {};
		\node [style=none] (125) at (17.5, -6) {$+$};
		\node [style=none] (126) at (22, -4) {};
		\node [style=none] (127) at (22.5, -4) {};
		\node [style=none] (128) at (23, -4) {};
		\node [style=none] (129) at (23.5, -4) {};
		\node [style=none] (130) at (22, -8) {};
		\node [style=none] (131) at (22.5, -8) {};
		\node [style=none] (132) at (23, -8) {};
		\node [style=none] (133) at (23.5, -8) {};
		\node [style=none] (135) at (21, -6) {$+$};
		\node [style=none] (154) at (-0.5, -6) {$=$};
		\node [style=none] (155) at (4.5, -4) {};
		\node [style=none] (156) at (5, -4) {};
		\node [style=none] (157) at (5.5, -4) {};
		\node [style=none] (158) at (6, -4) {};
		\node [style=none] (159) at (4.5, -8) {};
		\node [style=none] (160) at (5, -8) {};
		\node [style=none] (161) at (5.5, -8) {};
		\node [style=none] (162) at (6, -8) {};
		\node [style=none] (163) at (3.5, -6) {$-$};
		\node [style=none] (164) at (8, -4) {};
		\node [style=none] (165) at (8.5, -4) {};
		\node [style=none] (166) at (9, -4) {};
		\node [style=none] (167) at (9.5, -4) {};
		\node [style=none] (168) at (8, -8) {};
		\node [style=none] (169) at (8.5, -8) {};
		\node [style=none] (170) at (9, -8) {};
		\node [style=none] (171) at (9.5, -8) {};
		\node [style=none] (172) at (7, -6) {$-$};
		\node [style=none] (173) at (11.5, -4) {};
		\node [style=none] (174) at (12, -4) {};
		\node [style=none] (175) at (12.5, -4) {};
		\node [style=none] (176) at (13, -4) {};
		\node [style=none] (177) at (11.5, -8) {};
		\node [style=none] (178) at (12, -8) {};
		\node [style=none] (179) at (12.5, -8) {};
		\node [style=none] (180) at (13, -8) {};
		\node [style=none] (181) at (10.5, -6) {$-$};
	\end{pgfonlayer}
	\begin{pgfonlayer}{edgelayer}
		\draw [style=thickstrand] (2.center) to (6.center);
		\draw [style=thickstrand] (3.center) to (7.center);
		\draw [style=thickstrand] (4.center) to (8.center);
		\draw [style=thickstrand] (5.center) to (9.center);
		\draw [style=thickstrand] [bend left=90, looseness=2.50] (15.center) to (16.center);
		\draw [style=thickstrand] [bend right=90, looseness=2.50] (13.center) to (14.center);
		\draw [style=thickstrand] [in=90, out=-90] (11.center) to (17.center);
		\draw [style=thickstrand] [in=-90, out=90] (18.center) to (12.center);
		\draw [style=thickstrand] [bend left=90, looseness=2.50] (62.center) to (63.center);
		\draw [style=thickstrand] [bend left=90, looseness=2.50] (73.center) to (74.center);
		\draw [style=thickstrand] [bend left=90, looseness=2.50] (81.center) to (82.center);
		\draw [style=thickstrand] [bend left=90, looseness=2.50] (83.center) to (84.center);
		\draw [style=thickstrand] [bend left=90, looseness=2.50] (92.center) to (93.center);
		\draw [style=thickstrand] [bend left=90, looseness=2.25] (91.center) to (94.center);
		\draw [style=thickstrand] [bend right=90, looseness=2.50] (78.center) to (79.center);
		\draw [style=thickstrand] [bend right=90, looseness=2.25] (77.center) to (80.center);
		\draw [style=thickstrand] [bend right=90, looseness=2.50] (87.center) to (88.center);
		\draw [style=thickstrand] [bend right=90, looseness=2.50] (89.center) to (90.center);
		\draw [style=thickstrand] [bend right=90, looseness=2.50] (68.center) to (69.center);
		\draw [style=thickstrand] [bend right=90, looseness=2.50] (57.center) to (58.center);
		\draw [style=thickstrand] [in=90, out=-90] (59.center) to (61.center);
		\draw [style=thickstrand] (64.center) to (60.center);
		\draw [style=thickstrand] (67.center) to (71.center);
		\draw [style=thickstrand] [in=-90, out=90] (72.center) to (70.center);
		\draw [style=thickstrand] (97.center) to (101.center);
		\draw [style=thickstrand] (98.center) to (102.center);
		\draw [style=thickstrand] (99.center) to (103.center);
		\draw [style=thickstrand] (100.center) to (104.center);
		\draw [style=thickstrand] [bend left=90, looseness=2.50] (110.center) to (111.center);
		\draw [style=thickstrand] [bend right=90, looseness=2.50] (108.center) to (109.center);
		\draw [style=thickstrand] [in=90, out=-90] (106.center) to (112.center);
		\draw [style=thickstrand] [in=-90, out=90] (113.center) to (107.center);
		\draw [style=thickstrand] [bend left=90, looseness=2.50] (121.center) to (122.center);
		\draw [style=thickstrand] [bend left=90, looseness=2.50] (132.center) to (133.center);
		\draw [style=thickstrand] [bend right=90, looseness=2.50] (127.center) to (128.center);
		\draw [style=thickstrand] [bend right=90, looseness=2.50] (116.center) to (117.center);
		\draw [style=thickstrand] [in=90, out=-90] (118.center) to (120.center);
		\draw [style=thickstrand] (123.center) to (119.center);
		\draw [style=thickstrand] (126.center) to (130.center);
		\draw [style=thickstrand] [in=-90, out=90] (131.center) to (129.center);
		\draw [style=thickstrand] [bend right=90, looseness=2.50] (155.center) to (156.center);
		\draw [style=thickstrand] [bend left=90, looseness=2.25] (159.center) to (160.center);
		\draw [style=thickstrand] (157.center) to (161.center);
		\draw [style=thickstrand] (162.center) to (158.center);
		\draw [style=thickstrand] [bend left=90, looseness=2.50] (169.center) to (170.center);
		\draw [style=thickstrand] [bend right=90, looseness=2.50] (165.center) to (166.center);
		\draw [style=thickstrand] (164.center) to (168.center);
		\draw [style=thickstrand] (171.center) to (167.center);
		\draw [style=thickstrand] [bend left=90, looseness=2.25] (179.center) to (180.center);
		\draw [style=thickstrand] [bend right=90, looseness=2.50] (175.center) to (176.center);
		\draw [style=thickstrand] (173.center) to (177.center);
		\draw [style=thickstrand] (178.center) to (174.center);
	\end{pgfonlayer}
\end{tikzpicture}
	\end{center}

	\begin{proof}
		\textit{Naturality:} We show naturality in $\obj{m}$ and a completely analogous argument shows naturality in $\obj{n}$.
		Given any morphism $\euler{f}:\obj{m}\to \obj{m'}$, we wish to show that $\sigma_{\mobijs{m'},\mobijs{n}}\circ (\euler{f}\otimes \on{id}_{\obj{n}}) = (\on{id}_{\obj{n}}\otimes \euler{f})\circ \sigma_{\mobijs{m},\mobijs{n}}$.
		Since Temperley--Lieb diagrams form a basis of all morphisms, it suffices to show this equality for a Temperley--Lieb diagram $\euler{f}$.
		Note that for any $\euler{x}\in \mathscr{D}_{\mobijd{m'+n}}$, we have $\tau_{\mobijt{m'},\mobijt{n}}(\euler{x}) \circ (\euler{f}\otimes \on{id}_{\obj{n}}) = 0$ unless $\overline{K}(\euler{f})\subseteq K (\euler{x})$; 
		indeed, if we glue a cup at $(i,j)$ in $\euler{f}$ with a cap at $(i,k)$ or $(k,j)$ ($\neq (i,j)$) in $\euler{x}$ we get zero by the zig-zag relation, and if we glue a cup in $\euler{f}$ to two through strands of $\euler{x}$, we get zero upon composition with the Jones--Wenzl projector by \autoref{annihilation}.
		It follows that if $\euler{f}$ has a cup $(i,j)$, then $\tau_{\mobijt{m'},\mobijt{n}}(\euler{x})\circ (\euler{f}\otimes \on{id}_{\obj{n}})\neq 0$ if and only if $\tau_{\mobijt{m'},\mobijt{n}}(\euler{x})$ has a cap at $(i,j)$ if and only if $\tau_{\mobijt{m'},\mobijt{n}}(\euler{x})$ has a cup at $(i+n,j+n)$.
		Analogously, $\euler{f}$ has a cap $(i,j)$, then $(\on{id}_{\obj{n}}\otimes \euler{f})\circ\tau_{\mobijt{m},\mobijt{n}}(\euler{x})\neq 0$ if and only if $\tau_{\mobijt{m},\mobijt{n}}(\euler{x})$ has a cup at $(i,j)$ if and only if $\tau_{\mobijt{m},\mobijt{n}}(\euler{x})$ has a cap at $(i-n,j-n)$.
		Thus, we wish to show that 
		$$\sum_{\substack{\euler{x}\in \mathscr{D}_{\mobijd{m'+n}} \\ \overline{K}(\euler{f})\subseteq K (\tau_{\mobijt{m'},\mobijt{n}}(\euler{x}))}}\tau_{\mobijt{m'},\mobijt{n}}(\euler{x})\circ (\euler{f}\otimes \on{id}_{\obj{n}}) = \sum_{\substack{\euler{x}\in \mathscr{D}_{\mobijd{m+n}}\\ K (\euler{f})\subseteq\overline{K}(\tau_{\mobijt{m},\mobijt{n}}(\euler{x}))}}(\on{id}_{\obj{n}}\otimes \euler{f})\circ\tau_{\mobijt{m},\mobijt{n}}(\euler{x}).$$
		
		Next, consider the map $\euler{x}\mapsto \euler{x}\circ (\euler{f}\otimes\on{id})$ from the indexing set of the left-handside summation to the indexing set of the right-handside summation. This is well-defined since $\euler{x}\circ (\euler{f}\otimes \on{id}_{\obj{n}})$ are cap diagrams by the argument just given, and $ K (\euler{f})\subseteq K (\euler{x}\circ (\euler{f}\otimes \on{id}_{\obj{n}})$ as $\euler{x}$ is a cap diagram.
		This map is a bijection with the inverse $\euler{x}\mapsto\euler{x}\circ(\overline{\euler{f}}\circ \on{id}_{\obj{n}})$ (the inverse is well-defined for similar reasons);
		indeed, $\euler{x}\circ (\euler{f}\otimes\on{id}_{\obj{n}})\circ(\overline{\euler{f}}\circ \on{id}_{\obj{n}}) = \euler{x}$ since $\overline{K}(\euler{f})\subseteq K (\euler{x})$ and similarly for composition in the other direction.
		Hence, it suffices to prove that the equality holds term by term under this bijection, i.e. for every $\euler{x}\in \mathscr{D}_{\mobijd{m'+n}}$ such that $\overline{K}(\euler{f})\subseteq K (\tau_{\mobijt{m'},\mobijt{n}}(\euler{x}))$, we have
		$$\overline{\kappa_{m',n}(\euler{x})}\circ \euler{j}_{\thru(\euler{x})}\circ \euler{x}\circ (\euler{f}\otimes \on{id}_{\obj{n}}) = (\on{id}_{\obj{n}}\otimes \euler{f})\circ \overline{\kappa_{m,n}(\euler{x}\circ (\euler{f}\otimes \on{id}_{\obj{n}}})) \circ \euler{j}_{\thru(\euler{x})}\circ (\euler{x}\circ (\euler{f}\otimes \on{id}_{\obj{n}})).$$
		This is consequence of the equality
		$$\overline{\kappa_{m',n}(\euler{x})} = (\on{id}_{\obj{n}}\otimes \euler{f})\circ \overline{\kappa_{m,n}(\euler{x}\circ (\euler{f}\otimes \on{id}_{\obj{n}}})),$$
		which is proved by applying $\overline{(-)}$ to the following computation:
		\begin{equation*}
			\begin{split}
				\kappa_{m,n}(\euler{x}\circ(\euler{f}\otimes\on{id}_{\obj{n}}))\circ\overline{(\on{id}_{\obj{n}}\otimes \euler{f})} 
				&= \kappa_{m,n}(\euler{x}\circ(\euler{f}\otimes\on{id}_{\obj{n}}))\circ(\on{id}_{\obj{n}}\otimes \overline{\euler{f}}) \\
				&= [(\euler{x}\circ(\euler{f}\otimes\on{id}_{\obj{n}}))_{>m} \odot_l (\euler{x}\circ(\euler{f}\otimes\on{id}_{\obj{n}}))_{\leq m}]\circ(\on{id}_{\obj{n}}\otimes \overline{\euler{f}})\\
				&=[\euler{x}_{>m'} \odot_l (\euler{x}_{\leq m'}\circ \euler{f})]\circ(\on{id}_{\obj{n}}\otimes \overline{\euler{f}}) \\
				&=\euler{x}_{>m'} \odot_l (\euler{x}_{\leq m'}\circ \euler{f}\circ \overline{\euler{f}})\\
				&= \euler{x}_{>m'} \odot_l \euler{x}_{\leq m'} = \kappa_{m',n}(\euler{x}),
			\end{split}
		\end{equation*}
		where $l=\mathtt{hk}_{m'}(\euler{x})=\mathtt{hk}_m(\euler{x}\circ(\euler{f}\otimes \on{id}_{\obj{n}}))$.
		This completes the proof for naturality of $\sigma_{\mobijs{m},\mobijs{n}}$.
		
		Next we prove that $\sigma$ satisfies the properties of a commutor for coboundary category.
		The unit axiom is obvious as the unit morphisms in $\mathcal{TL}_0(\Bbbk)$ are trivial; and the fact that $\sigma_{\mobijs{m},\mobijs{n}}$ is an isomorphism follows from the symmetry axiom, so it remains to show the latter and the cactus axiom.
		
		\textit{Symmetry Axiom:}
		Since $\kappa_{m,n}$ is a bijection on $\mathscr{D}_{\mobijd{m+n}}$ whose inverse is $\kappa_{n,m}$, and since $\thru(\kappa_{m,n}(\euler{x}))=\thru(\euler{x})$, we have
		\begin{equation*}
			\begin{split}
				\sigma_{\mobijs{n},\mobijs{m}}\circ\sigma_{\mobijs{m},\mobijs{n}}
				&= \left(\sum_{\euler{x}\in\mathscr{D}_{\mobijd{m+n}}} \overline{\kappa_{\mobij{n},\mobij{m}}(\euler{y})}\circ \euler{j}_{\mobij{\thru(\euler{y})}}\circ \euler{y} \right) \circ \left(\sum_{\euler{x}\in\mathscr D_{\mobijd{m+n}}} \overline{\kappa_{\mobij{m},\mobij{n}}(\euler{x})}\circ \euler{j}_{\mobij{\thru(\euler{x})}}\circ \euler{x} \right) \\
				&= \left(\sum_{\euler{x}\in\mathscr{D}_{\mobijd{m+n}}} \overline{\kappa_{n,m}(\euler{y})}\circ \euler{j}_{\mathtt{th(\euler{y})}}\circ \euler{y} \right) \circ \left(\sum_{\euler{x}\in\mathscr{D}_{\mobijd{m+n}}} \overline{\euler{x}}\circ \euler{j}_{\mathtt{th(\euler{x})}}\circ \kappa_{n,m}(\euler{x}) \right) \\
				&= \sum_{(\euler{x},\euler{y})\in\mathscr{D}_{\mobijd{m+n}}^2} \overline{\kappa_{n,m}(\euler{y})}\circ \euler{j}_{\mathtt{th(\euler{y})}}\circ \euler{y} \circ \overline{\euler{x}}\circ \euler{j}_{\mathtt{th(\euler{x})}}\circ \kappa_{n,m}(\euler{x}) \\
				&= \sum_{\euler{x}\in\mathscr{D}_{\mobijd{m+n}}} \overline{\kappa_{n,m}(\euler{x})}\circ \euler{j}_{\mathtt{th(\euler{x})}}\circ \kappa_{n,m}(\euler{x}) \\
				&= \sum_{\euler{x}\in\mathscr D_{\mobijd{m+n}}} \overline{\euler{x}}\circ \euler{j}_{\mathtt{th(\euler{x})}}\circ \euler{x} = \on{id}_{\obj{m}\otimes \obj{n}},
			\end{split}
		\end{equation*}
		where we've used the orthogonality and sum statements of \autoref{DnIdempotents} for the fourth and the last equalities, respectively.
		
		\textit{Cactus Axiom:} Since associativity maps are obvious in $\mathcal{TL}_0(\Bbbk)$, we ignore them here. 
		We wish to show that the following diagram commutes:
		\[\begin{tikzcd}
			{\obj{r}\otimes \obj{s}\otimes \obj{t}} & {\obj{r}\otimes \obj{t}\otimes \obj{s}} \\
			{\obj{s}\otimes \obj{r}\otimes \obj{t}} & {\obj{t}\otimes \obj{s}\otimes \obj{r}}
			\arrow["{\on{id}_{\obj{t}}\otimes\sigma_{\obj{s},\obj{t}}}", from=1-1, to=1-2]
			\arrow["{\sigma_{\mobijs{r},\mobijs{s}}\otimes\on{id}_{\obj{t}}}"', from=1-1, to=2-1]
			\arrow["{\sigma_{\mobijs{r},\mobijs{t+s}}}", from=1-2, to=2-2]
			\arrow["{\sigma_{\mobijs{s+r},\mobijs{t}}}"', from=2-1, to=2-2]
		\end{tikzcd}\]
    Let us examine which terms in the summation
    $$\sigma_{\mobijs{r},\mobijs{s+t}}\circ (\on{id}_{\obj{r}}\otimes \sigma_{\mobijs{s},\mobijs{t}}) = \sum_{(\euler{y},\euler{x})\in \mathscr{D}_{\mobijd{r+s+t}} \times \mathscr{D}_{\mobijd{s+t}}} \tau_{\mobijt{r},\mobijt{s+t}}(\euler{y})\circ (\on{id}_{\obj{r}} \otimes \tau_{\mobijt{s},\mobijt{t}}(\euler{x}))$$
    are nonzero.
    Notice that for any given $\euler{y}\in\mathscr{D}_{\mobijd{r+s+t}}$, if $\kappa_{s,t}(\euler{x})\neq \euler{y}_{>r}\in \mathscr{D}_{\mobijd{s+t}}$, we get
    $$\euler{j}_{\thru(\euler{y})}\circ \euler{y} \circ (\on{id}_{\obj{r}}\otimes (\overline{\kappa_{s,t}(\euler{x})} \circ \euler{j}_{\thru(\euler{x})})) = 0$$
    since we get a cup or cap annihilating a Jones--Wenzl projector, or a zig-zag. 
    Thus, $\tau_{\mobijt{r},\mobijt{s+t}}(\euler{y})\circ (\on{id}_{\obj{r}} \otimes \tau_{\mobijt{s},\mobijt{t}}(\euler{x}))\neq 0$ only if $\euler{x} = \kappa_{t,s}(\euler{y}_{>r})$ and 
    $$\sigma_{\mobijs{r},\mobijs{s+t}}\circ (\on{id}_{\obj{r}}\otimes \sigma_{\mobijs{s}, \mobijs{t}}) = \sum_{\euler{y}\in \mathscr{D}_{\mobijd{r+s+t}}} \tau_{\mobijt{r},\mobijt{s+t}}(\euler{y})\circ (\on{id}_{\obj{r}} \otimes \tau_{\mobijt{s},\mobijt{t}}(\kappa_{t,s}(\euler{y}_{>r}))).$$
    Similarly, 
    $$\sigma_{\mobijs{r+s},\mobijs{t}}\circ (\sigma_{\mobijs{r},\mobijs{s}} \otimes \on{id}_{\obj{r}}) = \sum_{\euler{y}\in \mathscr{D}_{\mobijd{r+s+t}}} \tau_{\mobijt{r+s},\mobijt{t}}(\euler{y})\circ (\tau_{\mobijt{r}, \mobijt{s}}(\kappa_{s,r}(\euler{y}_{
    \leq r+s})) \otimes \on{id}_{\obj{t}}).$$
    To prove that the Cactus Axiom diagram commutes, we shall match the terms of the sum one by one using the bijection $\kappa_{t,r+s}\circ\kappa_{r,s+t}:\mathscr{D}_{\mobijd{r+s+t}}\to \mathscr{D}_{\mobijd{r+s+t}}$.
    Letting $\euler{y} = \euler{y}_r \odot_k \euler{y}_t \odot_l \euler{y}_s$ for some $\euler{y}_r\in\mathscr{D}_{\mobijd{r}}$, $\euler{y}_t\in\mathscr{D}_{\mobijd{t}}$ and  $\euler{y}_s\in\mathscr{D}_{\mobijd{s}}$, we get $\euler{y}':=\kappa_{t,r+s}\circ\kappa_{r,s+t}(\euler{y}) = \euler{y}_s \odot_k \euler{y}_r \odot_l \euler{y}_t$. 
    The following computation (illustrated pictorially in the diagram that follows) shows that the $\euler{y}$-term of $\sigma_{\mobijs{r},\mobijs{s+t}}\circ (\on{id}_{\obj{r}}\otimes \sigma_{\mobijs{s}, \mobijs{t}})$ equals the $\euler{y}'$-term of $\sigma_{\mobijs{r+s},\mobijs{t}}\circ (\sigma_{\mobijs{r},\mobijs{s}} \otimes \on{id}_{\obj{r}})$, and hence summing over all $\euler{y}\in\mathscr{D}_{\mobijd{r+s+t}}$ we see that the Cactus Axiom diagram commutes:
    \begin{equation*}
        \begin{split}
            \tau_{\mobijt{r+s},\mobijt{t}}(\euler{y}')\circ (\tau_{\mobijt{r}, \mobijt{s}}(\kappa_{s,r}(\euler{y}'_{
    \leq r+s})) \otimes \on{id}_{\obj{r}}) 
    &= \overline{\kappa_{r+s,t}(\euler{y}')}\circ \euler{j}_{\thru(\euler{y}')}\circ \euler{y}' \circ ((\overline{\euler{y}'_{\leq r+s}} \circ \euler{j}_{\thru(\euler{y}'_{\leq r+s})} \circ \kappa_{s,r}(\euler{y}'_{\leq r+s}))\otimes \on{id}_{\obj{t}}) \\
    &= \overline{\kappa_{r+s,t}(\euler{y}')}\circ \euler{j}_{\thru(\euler{y}')}\circ  ((\on{id}_{\thru(\euler{y}'_{\leq r+s})}\circ \euler{j}_{\thru(\euler{y}'_{\leq r+s})}\circ \kappa_{s,r}(\euler{y}'_{\leq r+s}))\odot_l\euler{y}'_{>r+s}) \\
    &= \overline{\kappa_{r+s,t}(\euler{y}')}\circ \euler{j}_{\thru(\euler{y}')}\circ  (\kappa_{s,r}(\euler{y}'_{\leq r+s})\odot_l\euler{y}'_{>r+s}) \\
    &= \overline{\kappa_{r+s,t}(\kappa_{t,r+s}\circ\kappa_{r,s+t}(\euler{y}))}\circ \euler{j}_{\thru(\euler{y})}\circ  (\kappa_{s,r}(\euler{y}_s \odot_k \euler{y}_r)\odot_l\euler{y}_t) \\
    &= \overline{\kappa_{r,s+t}(\euler{y})}\circ \euler{j}_{\thru(\euler{y})}\circ  (\euler{y}_r \odot_k \euler{y}_s\odot_l\euler{y}_t)\\
    &= \overline{\kappa_{r,s+t}(\euler{y})}\circ \euler{j}_{\thru(\euler{y})}\circ  (\euler{y}_{\leq r} \odot_k \kappa_{t,s}(\euler{y}_{>r})) \\
    &= \overline{\kappa_{r,s+t}(\euler{y})}\circ \euler{j}_{\thru(\euler{y})}\circ  (\euler{y}_{\leq r} \odot_k (\on{id}_{\thru(\euler{y}_{> r})}\circ \euler{j}_{\thru(\euler{y}_{> r})}\circ \kappa_{t,s}(\euler{y}_{> r}))) \\
    &= \overline{\kappa_{r,s+t}(\euler{y})}\circ \euler{j}_{\thru(\euler{y})}\circ \euler{y} \circ (\on{id}_{\obj{r}}\otimes (\overline{\euler{y}_{> r}} \circ \euler{j}_{\thru(\euler{y}_{>r})} \circ \kappa_{t,s}(\euler{y}_{>r}))) \\
    &= \tau_{\mobijt{r},\mobijt{s+t}}(\euler{y})\circ (\on{id}_{\obj{r}} \otimes \tau_{\mobijt{s},\mobijt{t}}(\kappa_{t,s}(\euler{y}_{>r}))),
        \end{split}
    \end{equation*}
    where the third and eighth equalities implicitly use \autoref{JWcaps} below. 
    \[\begin{tikzpicture}[scale=0.65]
	\begin{pgfonlayer}{nodelayer}
		\node [style=none] (0) at (-10.5, 3.5) {$\overline{\euler{y}_r}$};
		\node [style=none] (1) at (-11.5, 3.5) {$\overline{\euler{y}_s}$};
		\node [style=none] (2) at (-12.5, 3.5) {$\overline{\euler{y}_t}$};
		\node [style=none] (3) at (-10, 4) {};
		\node [style=none] (4) at (-11, 4) {};
		\node [style=none] (5) at (-12, 4) {};
		\node [style=none] (6) at (-13, 4) {};
		\node [style=none] (7) at (-13, 3) {};
		\node [style=none] (8) at (-10, 3) {};
		\node [style=none] (9) at (-11, 3) {};
		\node [style=none] (10) at (-12, 3) {};
		\node [style=none] (11) at (-11, 2) {};
		\node [style=none] (12) at (-12, 2) {};
		\node [style=none] (13) at (-13, 2) {};
		\node [style=none] (14) at (-10, 2) {};
		\node [style=none] (15) at (-11, 1) {};
		\node [style=none] (16) at (-12, 1) {};
		\node [style=none] (17) at (-13, 1) {};
		\node [style=none] (18) at (-13, 0) {};
		\node [style=none] (19) at (-12, 0) {};
		\node [style=none] (20) at (-11, 0) {};
		\node [style=none] (21) at (-11.5, 2.5) {JW};
		\node [style=none] (22) at (-10, 1) {};
		\node [style=none] (23) at (-10.5, 1.5) {$\euler{y}_t$};
		\node [style=none] (24) at (-11.5, 1.5) {$\euler{y}_r$};
		\node [style=none] (25) at (-12.5, 1.5) {$\euler{y}_s$};
		\node [style=none] (26) at (-11.5, 0.5) {$\overline{\euler{y}_r}$};
		\node [style=none] (27) at (-12.5, 0.5) {$\overline{\euler{y}_s}$};
		\node [style=none] (28) at (-12, -0.5) {JW};
		\node [style=none] (29) at (-11, -1) {};
		\node [style=none] (30) at (-12, -1) {};
		\node [style=none] (31) at (-13, -1) {};
		\node [style=none] (32) at (-13, -2) {};
		\node [style=none] (33) at (-12, -2) {};
		\node [style=none] (34) at (-11, -2) {};
		\node [style=none] (35) at (-11.5, -1.5) {$\euler{y}_s$};
		\node [style=none] (36) at (-12.5, -1.5) {$\euler{y}_r$};
		\node [style=none] (37) at (-10, -2) {};
		\node [style=none] (38) at (7.5, 3.5) {$\overline{\euler{y}_t}$};
		\node [style=none] (39) at (8.5, 3.5) {$\overline{\euler{y}_s}$};
		\node [style=none] (40) at (9.5, 3.5) {$\overline{\euler{y}_r}$};
		\node [style=none] (41) at (7, 4) {};
		\node [style=none] (42) at (8, 4) {};
		\node [style=none] (43) at (9, 4) {};
		\node [style=none] (44) at (10, 4) {};
		\node [style=none] (45) at (10, 3) {};
		\node [style=none] (46) at (7, 3) {};
		\node [style=none] (47) at (8, 3) {};
		\node [style=none] (48) at (9, 3) {};
		\node [style=none] (49) at (8, 2) {};
		\node [style=none] (50) at (9, 2) {};
		\node [style=none] (51) at (10, 2) {};
		\node [style=none] (52) at (7, 2) {};
		\node [style=none] (54) at (9, 1) {};
		\node [style=none] (55) at (10, 1) {};
		\node [style=none] (56) at (10, 0) {};
		\node [style=none] (57) at (9, 0) {};
		\node [style=none] (58) at (8, 0) {};
		\node [style=none] (59) at (8.5, 2.5) {JW};
		\node [style=none] (60) at (7, 1) {};
		\node [style=none] (61) at (7.5, 1.5) {$\euler{y}_r$};
		\node [style=none] (62) at (8.5, 1.5) {$\euler{y}_t$};
		\node [style=none] (63) at (9.5, 1.5) {$\euler{y}_s$};
		\node [style=none] (64) at (8.5, 0.5) {$\overline{\euler{y}_t}$};
		\node [style=none] (65) at (9.5, 0.5) {$\overline{\euler{y}_s}$};
		\node [style=none] (66) at (9, -0.5) {JW};
		\node [style=none] (67) at (8, -1) {};
		\node [style=none] (68) at (9, -1) {};
		\node [style=none] (69) at (10, -1) {};
		\node [style=none] (70) at (10, -2) {};
		\node [style=none] (71) at (9, -2) {};
		\node [style=none] (73) at (8.5, -1.5) {$\euler{y}_s$};
		\node [style=none] (74) at (9.5, -1.5) {$\euler{y}_t$};
		\node [style=none] (75) at (7, -2) {};
		\node [style=none] (78) at (-0.5, 3.5) {$\overline{\euler{y}_r}$};
		\node [style=none] (79) at (-1.5, 3.5) {$\overline{\euler{y}_s}$};
		\node [style=none] (80) at (-2.5, 3.5) {$\overline{\euler{y}_t}$};
		\node [style=none] (81) at (0, 4) {};
		\node [style=none] (82) at (-1, 4) {};
		\node [style=none] (83) at (-2, 4) {};
		\node [style=none] (84) at (-3, 4) {};
		\node [style=none] (85) at (-3, 3) {};
		\node [style=none] (86) at (0, 3) {};
		\node [style=none] (87) at (-1, 3) {};
		\node [style=none] (88) at (-2, 3) {};
		\node [style=none] (89) at (-1, 2) {};
		\node [style=none] (90) at (-2, 2) {};
		\node [style=none] (91) at (-3, 2) {};
		\node [style=none] (92) at (0, 2) {};
		\node [style=none] (93) at (-1, -2) {};
		\node [style=none] (99) at (-1.5, 2.5) {JW};
		\node [style=none] (101) at (-0.5, 0) {$\euler{y}_t$};
		\node [style=none] (108) at (-2, 2) {};
		\node [style=none] (110) at (-3, -2) {};
		\node [style=none] (111) at (-2, -2) {};
		\node [style=none] (112) at (-1, -2) {};
		\node [style=none] (113) at (-1.5, 0) {$\euler{y}_s$};
		\node [style=none] (114) at (-2.5, 0) {$\euler{y}_r$};
		\node [style=none] (115) at (0, -2) {};
		\node [style=none] (116) at (-4, 1) {$=$};
		\node [style=none] (117) at (1, 1) {$=$};
		\node [style=none] (118) at (-11.5, 4.5) {$\euler{y}'$-term};
		\node [style=none] (119) at (8.5, 4.5) {$\euler{y}$-term};
		\node [style=none] (120) at (11.5, -0.25) {};
		\node [style=none] (121) at (11.5, -1.25) {};
		\node [style=none] (122) at (11.5, 1.5) {};
		\node [style=none] (123) at (11.5, 0.5) {};
		\node [style=none] (124) at (11.5, 3.25) {};
		\node [style=none] (125) at (11.5, 2.25) {};
		\node [style=none] (126) at (12, -0.75) {$\otimes$};
		\node [style=none] (127) at (12, 1) {$\odot_k$};
		\node [style=none] (128) at (12, 2.75) {$\odot_l$};
		\node [style=none] (129) at (8, 1) {};
		\node [style=none] (130) at (8, -2) {};
		\node [style=none] (131) at (-5.5, 3.5) {$\overline{\euler{y}_r}$};
		\node [style=none] (132) at (-6.5, 3.5) {$\overline{\euler{y}_s}$};
		\node [style=none] (133) at (-7.5, 3.5) {$\overline{\euler{y}_t}$};
		\node [style=none] (134) at (-5, 4) {};
		\node [style=none] (135) at (-6, 4) {};
		\node [style=none] (136) at (-7, 4) {};
		\node [style=none] (137) at (-8, 4) {};
		\node [style=none] (138) at (-8, 3) {};
		\node [style=none] (139) at (-5, 3) {};
		\node [style=none] (140) at (-6, 3) {};
		\node [style=none] (141) at (-7, 3) {};
		\node [style=none] (142) at (-6, 2) {};
		\node [style=none] (143) at (-7, 2) {};
		\node [style=none] (144) at (-8, 2) {};
		\node [style=none] (145) at (-5, 2) {};
		\node [style=none] (146) at (-6, -2) {};
		\node [style=none] (147) at (-7, 1) {$\on{id}$};
		\node [style=none] (148) at (-8, 1) {};
		\node [style=none] (149) at (-8, 0) {};
		\node [style=none] (150) at (-7, 0) {};
		\node [style=none] (151) at (-6, 0) {};
		\node [style=none] (152) at (-6.5, 2.5) {JW};
		\node [style=none] (153) at (-5, 1) {};
		\node [style=none] (154) at (-5.5, 0) {$\euler{y}_t$};
		\node [style=none] (159) at (-7, -0.5) {JW};
		\node [style=none] (160) at (-6, -1) {};
		\node [style=none] (161) at (-7, -1) {};
		\node [style=none] (162) at (-8, -1) {};
		\node [style=none] (163) at (-8, -2) {};
		\node [style=none] (164) at (-7, -2) {};
		\node [style=none] (166) at (-6.5, -1.5) {$\euler{y}_s$};
		\node [style=none] (167) at (-7.5, -1.5) {$\euler{y}_r$};
		\node [style=none] (168) at (-5, -2) {};
		\node [style=none] (169) at (2.5, 3.5) {$\overline{\euler{y}_t}$};
		\node [style=none] (170) at (3.5, 3.5) {$\overline{\euler{y}_s}$};
		\node [style=none] (171) at (4.5, 3.5) {$\overline{\euler{y}_r}$};
		\node [style=none] (172) at (2, 4) {};
		\node [style=none] (173) at (3, 4) {};
		\node [style=none] (174) at (4, 4) {};
		\node [style=none] (175) at (5, 4) {};
		\node [style=none] (176) at (5, 3) {};
		\node [style=none] (177) at (2, 3) {};
		\node [style=none] (178) at (3, 3) {};
		\node [style=none] (179) at (4, 3) {};
		\node [style=none] (180) at (3, 2) {};
		\node [style=none] (181) at (4, 2) {};
		\node [style=none] (182) at (5, 2) {};
		\node [style=none] (183) at (2, 2) {};
		\node [style=none] (184) at (4, 1) {$\on{id}$};
		\node [style=none] (185) at (5, 1) {};
		\node [style=none] (186) at (5, 0) {};
		\node [style=none] (187) at (4, 0) {};
		\node [style=none] (188) at (3, 0) {};
		\node [style=none] (189) at (3.5, 2.5) {JW};
		\node [style=none] (190) at (2, 1) {};
		\node [style=none] (191) at (2.5, 0) {$\euler{y}_r$};
		\node [style=none] (196) at (4, -0.5) {JW};
		\node [style=none] (197) at (3, -1) {};
		\node [style=none] (198) at (4, -1) {};
		\node [style=none] (199) at (5, -1) {};
		\node [style=none] (200) at (5, -2) {};
		\node [style=none] (201) at (4, -2) {};
		\node [style=none] (202) at (3.5, -1.5) {$\euler{y}_s$};
		\node [style=none] (203) at (4.5, -1.5) {$\euler{y}_t$};
		\node [style=none] (204) at (2, -2) {};
		\node [style=none] (206) at (3, -2) {};
		\node [style=none] (207) at (-10.5, -0.5) {$\on{id}_t$};
		\node [style=none] (208) at (7.5, -0.5) {$\on{id}_r$};
		\node [style=none] (209) at (-9, 1) {$=$};
		\node [style=none] (210) at (6, 1) {$=$};
	\end{pgfonlayer}
	\begin{pgfonlayer}{edgelayer}
		\draw [style=thickstrand] (3.center) to (6.center);
		\draw [style=thickstrand] (6.center) to (32.center);
		\draw [style=thickstrand] (32.center) to (37.center);
		\draw [style=thickstrand] (37.center) to (3.center);
		\draw [style=thickstrand] (8.center) to (7.center);
		\draw [style=thickstrand] (14.center) to (13.center);
		\draw [style=thickstrand] (22.center) to (17.center);
		\draw [style=thickstrand] (29.center) to (31.center);
		\draw [style=thickstrand] (20.center) to (18.center);
		\draw [style=thickstrand] (41.center) to (44.center);
		\draw [style=thickstrand] (44.center) to (70.center);
		\draw [style=thickstrand] (70.center) to (75.center);
		\draw [style=thickstrand] (75.center) to (41.center);
		\draw [style=thickstrand] (46.center) to (45.center);
		\draw [style=thickstrand] (52.center) to (51.center);
		\draw [style=thickstrand] (60.center) to (55.center);
		\draw [style=thickstrand] (67.center) to (69.center);
		\draw [style=thickstrand] (58.center) to (56.center);
		\draw [style=thickstrand] (34.center) to (15.center);
		\draw [style=redline] (54.center) to (50.center);
		\draw [style=blueline] (48.center) to (43.center);
		\draw [style=redline] (47.center) to (42.center);
		\draw [style=redline] (57.center) to (54.center);
		\draw [style=redline] (71.center) to (68.center);
		\draw [style=blueline] (16.center) to (12.center);
		\draw [style=blueline] (16.center) to (19.center);
		\draw [style=blueline] (30.center) to (33.center);
		\draw [style=redline] (15.center) to (11.center);
		\draw [style=redline] (10.center) to (5.center);
		\draw [style=blueline] (9.center) to (4.center);
		\draw [style=thickstrand] (81.center) to (84.center);
		\draw [style=thickstrand] (84.center) to (110.center);
		\draw [style=thickstrand] (110.center) to (115.center);
		\draw [style=thickstrand] (115.center) to (81.center);
		\draw [style=thickstrand] (86.center) to (85.center);
		\draw [style=thickstrand] (92.center) to (91.center);
		\draw [style=blueline] (108.center) to (111.center);
		\draw [style=redline] (93.center) to (89.center);
		\draw [style=redline] (88.center) to (83.center);
		\draw [style=blueline] (87.center) to (82.center);
		\draw [style=thickstrand] (120.center) to (121.center);
		\draw [style=blueline] (122.center) to (123.center);
		\draw [style=redline] (124.center) to (125.center);
		\draw [style=thickstrand] (130.center) to (129.center);
		\draw [style=blueline] (129.center) to (49.center);
		\draw [style=thickstrand] (134.center) to (137.center);
		\draw [style=thickstrand] (137.center) to (163.center);
		\draw [style=thickstrand] (163.center) to (168.center);
		\draw [style=thickstrand] (168.center) to (134.center);
		\draw [style=thickstrand] (139.center) to (138.center);
		\draw [style=thickstrand] (145.center) to (144.center);
		\draw [style=thickstrand] (160.center) to (162.center);
		\draw [style=thickstrand] (151.center) to (149.center);
		\draw [style=blueline] (161.center) to (164.center);
		\draw [style=redline] (146.center) to (142.center);
		\draw [style=redline] (141.center) to (136.center);
		\draw [style=blueline] (140.center) to (135.center);
		\draw [style=thickstrand] (172.center) to (175.center);
		\draw [style=thickstrand] (175.center) to (200.center);
		\draw [style=thickstrand] (200.center) to (204.center);
		\draw [style=thickstrand] (204.center) to (172.center);
		\draw [style=thickstrand] (177.center) to (176.center);
		\draw [style=thickstrand] (183.center) to (182.center);
		\draw [style=thickstrand] (197.center) to (199.center);
		\draw [style=thickstrand] (188.center) to (186.center);
		\draw [style=blueline] (179.center) to (174.center);
		\draw [style=redline] (178.center) to (173.center);
		\draw [style=redline] (201.center) to (198.center);
		\draw [style=blueline] (206.center) to (180.center);
	\end{pgfonlayer}
\end{tikzpicture}
\]
	\end{proof}

    \begin{definition}
        Define the \emph{diagram reversal} operations on cap diagrams $\vartheta:\mathscr D_n\to \mathscr D_n$ by letting $\vartheta(x)$ be the diagram acquired by vertically reflecting $\euler{x}$ about its center.
    \end{definition}

    Using diagram reversal operations, we can see how the cactus group acts on tensor products in $\mathcal{TL}_0(\Bbbk)$.
    \begin{corollary}\label{cor:intervalreversal}
        The interval reversal morphisms of $\mathcal{TL}_0(\Bbbk)$ (with the commutor defined in \autoref{TLCob}) are given by 
        $$s_{p,q} = \on{id}_{\obj{p-1}}\otimes \left(\sum_{\euler{x}\in\mathscr D_n} \overline{\vartheta(\euler{x})}\circ \euler{j}_{\mathtt{th({\euler{x})}}}\circ \euler{x}\right)\otimes \on{id}_{\obj{n-q}}$$.
    \end{corollary}
    \begin{example}
        The map $s_{1,4}\in \on{End}_{\mathcal{TL}_0(\Bbbk)}(\obj{5})$ is given by
        \[\begin{tikzpicture}[scale=0.55]
	\begin{pgfonlayer}{nodelayer}
		\node [style=none] (51) at (9, 0) {};
		\node [style=none] (52) at (9.5, 0) {};
		\node [style=none] (53) at (10, 0) {};
		\node [style=none] (54) at (10.5, 0) {};
		\node [style=none] (55) at (5.5, 0) {};
		\node [style=none] (56) at (6, 0) {};
		\node [style=none] (57) at (5, 0) {};
		\node [style=none] (58) at (6.5, 0) {};
		\node [style=none] (59) at (1, 0) {};
		\node [style=none] (60) at (1.5, 0) {};
		\node [style=none] (61) at (2, 0) {};
		\node [style=none] (62) at (2.5, 0) {};
		\node [style=none] (71) at (-6, 0) {};
		\node [style=none] (72) at (-5.5, 0) {};
		\node [style=none] (73) at (-7, 0) {};
		\node [style=none] (74) at (-6.5, 0) {};
		\node [style=none] (75) at (-6, 3) {};
		\node [style=none] (76) at (-5.5, 3) {};
		\node [style=none] (77) at (-9, 0) {};
		\node [style=none] (78) at (-9.5, 0) {};
		\node [style=none] (79) at (-10, 0) {};
		\node [style=none] (80) at (-10.5, 0) {};
		\node [style=none] (81) at (-9, 3) {};
		\node [style=none] (82) at (-9.5, 3) {};
		\node [style=none] (83) at (-10, 3) {};
		\node [style=none] (84) at (-10.5, 3) {};
		\node [style=none] (85) at (8, 1.5) {$+$};
		\node [style=none] (86) at (4, 1.5) {$+$};
		\node [style=none] (87) at (0, 1.5) {$+$};
		\node [style=none] (88) at (-4, 1.5) {$+$};
		\node [style=none] (89) at (-8, 1.5) {$+$};
		\node [style=none] (90) at (5.5, 3) {};
		\node [style=none] (91) at (6, 3) {};
		\node [style=none] (92) at (5, 3) {};
		\node [style=none] (93) at (6.5, 3) {};
		\node [style=none] (94) at (9, 3) {};
		\node [style=none] (95) at (9.5, 3) {};
		\node [style=none] (96) at (10, 3) {};
		\node [style=none] (97) at (10.5, 3) {};
		\node [style=none] (98) at (2, 3) {};
		\node [style=none] (99) at (2.5, 3) {};
		\node [style=none] (100) at (1, 3) {};
		\node [style=none] (101) at (1.5, 3) {};
		\node [style=none] (102) at (-2.5, 3) {};
		\node [style=none] (103) at (-2, 3) {};
		\node [style=none] (104) at (-3, 3) {};
		\node [style=none] (105) at (-1.5, 3) {};
		\node [style=none] (106) at (-2.5, 0) {};
		\node [style=none] (107) at (-2, 0) {};
		\node [style=none] (108) at (-3, 0) {};
		\node [style=none] (109) at (-1.5, 0) {};
		\node [style=none] (110) at (-7, 3) {};
		\node [style=none] (111) at (-6.5, 3) {};
		\node [style=jonesrectangle] (112) at (-10.25, 1.5) {$\euler{j}_4$};
		\node [style=jonesrectangle] (113) at (-6.25, 1.5) {$\euler{j}_2$};
		\node [style=jonesrectangle] (114) at (-2.25, 1.5) {$\euler{j}_2$};
		\node [style=jonesrectangle] (115) at (1.75, 1.5) {$\euler{j}_2$};
		\node [style=none] (116) at (11, 3) {};
		\node [style=none] (117) at (11, 0) {};
		\node [style=none] (118) at (7, 3) {};
		\node [style=none] (119) at (7, 0) {};
		\node [style=none] (120) at (3, 3) {};
		\node [style=none] (121) at (3, 0) {};
		\node [style=none] (122) at (-1, 3) {};
		\node [style=none] (123) at (-1, 0) {};
		\node [style=none] (124) at (-5, 3) {};
		\node [style=none] (125) at (-5, 0) {};
		\node [style=none] (126) at (-11, 3) {};
		\node [style=none] (127) at (-11, 0) {};
		\node [style=none] (128) at (-12, 1.5) {$=$};
		\node [style=none] (129) at (-13, 1.5) {$s_{1,4}$};
		\node [style=none] (130) at (-11, -4) {};
		\node [style=none] (131) at (-10.5, -4) {};
		\node [style=none] (132) at (-10, -4) {};
		\node [style=none] (133) at (-9.5, -4) {};
		\node [style=none] (134) at (-11, -1) {};
		\node [style=none] (135) at (-10.5, -1) {};
		\node [style=none] (136) at (-10, -1) {};
		\node [style=none] (137) at (-9.5, -1) {};
		\node [style=none] (138) at (-12, -2.5) {$=$};
		\node [style=none] (140) at (-9, -1) {};
		\node [style=none] (141) at (-9, -4) {};
		\node [style=none] (142) at (-7, -4) {};
		\node [style=none] (143) at (-6.5, -4) {};
		\node [style=none] (144) at (-6, -4) {};
		\node [style=none] (145) at (-5.5, -4) {};
		\node [style=none] (146) at (-7, -1) {};
		\node [style=none] (147) at (-6.5, -1) {};
		\node [style=none] (148) at (-6, -1) {};
		\node [style=none] (149) at (-5.5, -1) {};
		\node [style=none] (150) at (-5, -1) {};
		\node [style=none] (151) at (-5, -4) {};
		\node [style=none] (152) at (-8, -2.5) {$-$};
		\node [style=none] (153) at (-3, -4) {};
		\node [style=none] (154) at (-2.5, -4) {};
		\node [style=none] (155) at (-2, -4) {};
		\node [style=none] (156) at (-1.5, -4) {};
		\node [style=none] (157) at (-3, -1) {};
		\node [style=none] (158) at (-2.5, -1) {};
		\node [style=none] (159) at (-2, -1) {};
		\node [style=none] (160) at (-1.5, -1) {};
		\node [style=none] (161) at (-1, -1) {};
		\node [style=none] (162) at (-1, -4) {};
		\node [style=none] (163) at (-4, -2.5) {$-$};
		\node [style=none] (164) at (1, -4) {};
		\node [style=none] (165) at (1.5, -4) {};
		\node [style=none] (166) at (2, -4) {};
		\node [style=none] (167) at (2.5, -4) {};
		\node [style=none] (168) at (0, -2.5) {$+$};
		\node [style=none] (169) at (2, -1) {};
		\node [style=none] (170) at (2.5, -1) {};
		\node [style=none] (171) at (1, -1) {};
		\node [style=none] (172) at (1.5, -1) {};
		\node [style=none] (174) at (3, -1) {};
		\node [style=none] (175) at (3, -4) {};
		\node [style=none] (176) at (6, -4) {};
		\node [style=none] (177) at (6.5, -4) {};
		\node [style=none] (178) at (5, -4) {};
		\node [style=none] (179) at (5.5, -4) {};
		\node [style=none] (180) at (6, -1) {};
		\node [style=none] (181) at (6.5, -1) {};
		\node [style=none] (182) at (4, -2.5) {$+$};
		\node [style=none] (183) at (5, -1) {};
		\node [style=none] (184) at (5.5, -1) {};
		\node [style=none] (186) at (7, -1) {};
		\node [style=none] (187) at (7, -4) {};
	\end{pgfonlayer}
	\begin{pgfonlayer}{edgelayer}
		\draw [style=thickstrand] [bend left=90, looseness=1.75] (51.center) to (52.center);
		\draw [style=thickstrand] [bend left=90, looseness=1.75] (53.center) to (54.center);
		\draw [style=thickstrand] [bend left=90, looseness=1.75] (55.center) to (56.center);
		\draw [style=thickstrand] [in=90, out=90, looseness=1.25] (57.center) to (58.center);
		\draw [style=thickstrand] [bend left=90, looseness=1.75] (59.center) to (60.center);
		\draw [style=thickstrand] [bend left=90, looseness=1.75] (71.center) to (72.center);
		\draw [style=thickstrand] [in=-90, out=90] (73.center) to (75.center);
		\draw [style=thickstrand] [in=90, out=-90] (76.center) to (74.center);
		\draw [style=thickstrand] (77.center) to (81.center);
		\draw [style=thickstrand] (78.center) to (82.center);
		\draw [style=thickstrand] (79.center) to (83.center);
		\draw [style=thickstrand] (80.center) to (84.center);
		\draw [style=thickstrand] [bend right=90, looseness=1.75] (90.center) to (91.center);
		\draw [style=thickstrand] [bend right=90, looseness=1.25] (92.center) to (93.center);
		\draw [style=thickstrand] [bend right=90, looseness=1.75] (94.center) to (95.center);
		\draw [style=thickstrand] [bend right=90, looseness=1.75] (96.center) to (97.center);
		\draw [style=thickstrand] [bend right=90, looseness=1.75] (98.center) to (99.center);
		\draw [style=thickstrand] [in=90, out=-90] (101.center) to (62.center);
		\draw [style=thickstrand] [in=90, out=-90] (100.center) to (61.center);
		\draw [style=thickstrand] [bend right=90, looseness=1.75] (102.center) to (103.center);
		\draw [style=thickstrand] [bend left=90, looseness=1.75] (106.center) to (107.center);
		\draw [style=thickstrand] (109.center) to (105.center);
		\draw [style=thickstrand] [in=-90, out=90] (108.center) to (104.center);
		\draw [style=thickstrand] [bend right=90, looseness=1.75] (110.center) to (111.center);
		\draw [style=thickstrand] (116.center) to (117.center);
		\draw [style=thickstrand] (118.center) to (119.center);
		\draw [style=thickstrand] (120.center) to (121.center);
		\draw [style=thickstrand] (122.center) to (123.center);
		\draw [style=thickstrand] (124.center) to (125.center);
		\draw [style=thickstrand] (126.center) to (127.center);
		\draw [style=thickstrand] (130.center) to (134.center);
		\draw [style=thickstrand] (131.center) to (135.center);
		\draw [style=thickstrand] (132.center) to (136.center);
		\draw [style=thickstrand] (133.center) to (137.center);
		\draw [style=thickstrand] (140.center) to (141.center);
		\draw [style=thickstrand] [bend left=90, looseness=1.75] (142.center) to (143.center);
		\draw [style=thickstrand] [bend right=90, looseness=1.75] (146.center) to (147.center);
		\draw [style=thickstrand] (150.center) to (151.center);
		\draw [style=thickstrand] [bend left=90, looseness=1.75] (155.center) to (156.center);
		\draw [style=thickstrand] [bend right=90, looseness=1.75] (159.center) to (160.center);
		\draw [style=thickstrand] (161.center) to (162.center);
		\draw [style=thickstrand] [bend left=90, looseness=1.75] (164.center) to (165.center);
		\draw [style=thickstrand] [bend right=90, looseness=1.75] (169.center) to (170.center);
		\draw [style=thickstrand] [in=90, out=-90] (172.center) to (167.center);
		\draw [style=thickstrand] [in=90, out=-90] (171.center) to (166.center);
		\draw [style=thickstrand] (174.center) to (175.center);
		\draw [style=thickstrand] [bend left=90, looseness=1.75] (176.center) to (177.center);
		\draw [style=thickstrand] [in=-90, out=90] (178.center) to (180.center);
		\draw [style=thickstrand] [in=90, out=-90] (181.center) to (179.center);
		\draw [style=thickstrand] [bend right=90, looseness=1.75] (183.center) to (184.center);
		\draw [style=thickstrand] (186.center) to (187.center);
		\draw [style=thickstrand] (148.center) to (144.center);
		\draw [style=thickstrand] (145.center) to (149.center);
		\draw [style=thickstrand] (157.center) to (153.center);
		\draw [style=thickstrand] (154.center) to (158.center);
	\end{pgfonlayer}
\end{tikzpicture}
\]
    \end{example}

    The proof of this corollary involves lengthy, but straight-forward computations, so we omit the details.
    
    \textit{Proof Sketch. } It suffices to prove this formula for $s_{1,n}$ since other interval reversal morphisms are acquired from such via tensoring with identity maps.
    By the recursive formula given for interval reversal maps in \autoref{CobCats}, we have 
    $$s_{1,n} = \sigma_{1,1,n}\circ \sigma_{2,2,n} \circ\dots\circ \sigma_{n-1,n-1,n},$$
    where $\sigma_{i,i,n}$ may be alternatively written as $\on{id}_{\obj{i-1}}\otimes \sigma_{\obj{1},\obj{n-i-1}}$.
    Using the formula of \autoref{TLCob} we get 
    $$s_{1,n}=\left(\sum_{\euler{x}\in \mathscr D_n} \tau_{1,{n-1}}(\euler{x})\right)\circ
    \left(\on{id}_{\obj{1}}\otimes \sum_{\euler{x}\in \mathscr D_{n-1}} \tau_{1,{n-2}}(\euler{x})\right) \circ 
    \dots \circ 
    \left(\on{id}_{\obj{n-2}}\sum_{\euler{x}\in \mathscr D_2} \tau_{1,1}(\euler{x})\right).$$
    Expanding and eliminating zero terms (due to zig-zags or cups/caps annihilating Jones Wenzl projectors) we end up with only one term involving each cap diagram $\euler{x}\in\mathscr D_n$; more explicitly, 
    $$s_{1,n} = \sum_{\euler{x}\in \mathscr D_n} \overline{(\kappa_{n-1, 1}\circ\kappa_{n-2,2}\circ \dots\circ \kappa_{1,n-1})(\euler{x})}
    \circ \euler{j}_{\mathtt{th}(\euler{x})}\circ \euler{x}.$$
    Finally, one can show that $\vartheta=\kappa_{n-1, 1}\circ\kappa_{n-2,2}\circ \dots\circ \kappa_{1,n-1}$.\hfill $\qed$
    
	\subsection{Equivalence as Coboundary Categories}\label{sec:coboundaryequiv}

    Due to naturality of the morphisms constituting a coboundary structure, the one defined in the last section on $\mathcal{TL}_0(\Bbbk)$ extends uniquely to a coboundary structure on its Cauchy completion $\mathbf{CrysTL}$ (see \eqref{eq:MonCauchy} and the explanations there).
    Proving that the equivalence $F$ defined above is an equivalence of coboundary categories relies on defining a bijection between cap diagrams on $n$ strands on the one hand and direct summands of $B^{\otimes n}$ on the other.

    \begin{notation}
        Let $C$ be a component of highest weight $\lambda$ and highest weight element $c_0$ of the crystal $B=B^{\otimes n}$.
		So, the elements of $C$ are $c_0, c_1:=f(c_0), \dots, c_\lambda:= f^\lambda(c_0)$.
		Let $C_+$ (resp. $C_-$) be the component of highest weights $\lambda+1$ (resp. $\lambda-1$) of $B\otimes C$ generated by the highest weight elements $b_0\otimes c_0$ (resp. $b_0\otimes c_1$). 
        Note that $C_-$ is not defined if $C$ is of highest weight $0$ in which case we only get one component, $C_+$ in $B^{\otimes n+1}$.
		Similarly, let $C^+$ (resp. $C^-$) be the component of highest weights $\lambda+1$ (resp. $\lambda-1$) of $C\otimes B$ generated by the highest weight elements $c_0\otimes b_0$ (resp. $c_0\otimes b_1$).
    \end{notation}
    
    Consider the map $\Phi_n:\mathscr{D}_{\mobijd{n}} \to \mathscr{C}_{\mobijd{n}}$ between the set $\mathscr{D}_{\mobijd{n}}$ of cap diagrams on $n$ strands and the set $\mathscr{C}_{\mobijd{n}}$ of connected components of $B^{\otimes n}$ defined recursively as follows:
		define $\Phi_0(\on{id}_{\obj{0}})=B_0=\mathbb 1$; and for every $\euler{x}\in \mathscr{D}_{\mobijd{n}}$, define $\Phi_{n+1}( \euler{x} \odot_0 \on{id}_{\obj{1}}) = \Phi_n(\euler{x})^+$ and $\Phi_{n+1}(\euler{x} \odot_1 \on{id}_{\obj{1}}) = \Phi_n(\euler{x})^-$, whenever defined. 
		It is easy to verify that this defines a bijection. 
		
        The following \emph{branching graph}  depicts this recursively-defined bijection up to $n=4$. 
            
		
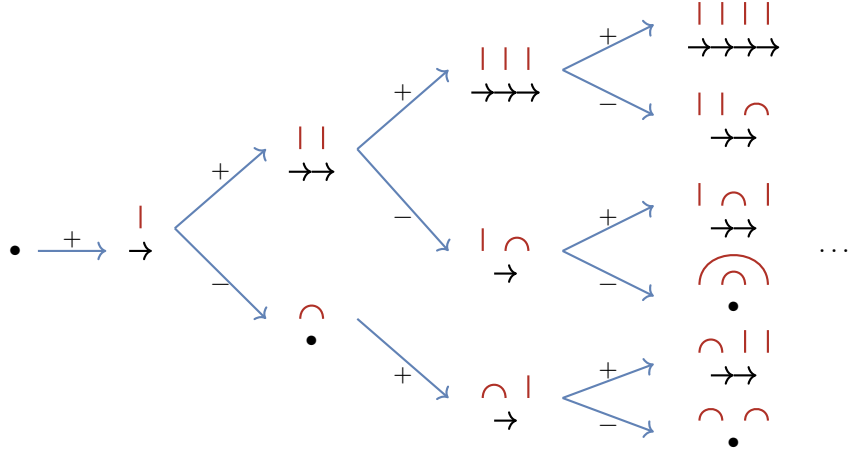
\[\begin{tikzpicture}[scale=0.6]
	\begin{pgfonlayer}{nodelayer}
		\node [style=none] (0) at (6.75, 3.5) {};
		\node [style=none] (1) at (7.25, 3.5) {};
		\node [style=none] (2) at (7.75, 3.5) {};
		\node [style=none] (3) at (8.25, 3.5) {};
		\node [style=none] (4) at (8.75, 3.5) {};
		\node [style=none] (5) at (7.25, 1.5) {};
		\node [style=none] (6) at (7.75, 1.5) {};
		\node [style=none] (7) at (8.25, 1.5) {};
		\node [style=none] (8) at (7.25, -0.5) {};
		\node [style=none] (9) at (7.75, -0.5) {};
		\node [style=none] (10) at (8.25, -0.5) {};
		\node [style=none] (11) at (7.75, -2.25) {$\bullet$};
		\node [style=none] (12) at (7.25, -3.75) {};
		\node [style=none] (13) at (7.75, -3.75) {};
		\node [style=none] (14) at (8.25, -3.75) {};
		\node [style=none] (15) at (7.75, -5.25) {$\bullet$};
		\node [style=none] (16) at (2, 2.5) {};
		\node [style=none] (17) at (2.5, 2.5) {};
		\node [style=none] (18) at (3, 2.5) {};
		\node [style=none] (19) at (3.5, 2.5) {};
		\node [style=none] (20) at (2.5, -1.5) {};
		\node [style=none] (21) at (3, -1.5) {};
		\node [style=none] (22) at (2.5, -4.75) {};
		\node [style=none] (23) at (3, -4.75) {};
		\node [style=none] (24) at (-1, 0.75) {};
		\node [style=none] (25) at (-1.5, 0.75) {};
		\node [style=none] (26) at (-2, 0.75) {};
		\node [style=none] (28) at (-5, -1) {};
		\node [style=none] (29) at (-5.5, -1) {};
		\node [style=none] (30) at (-1.5, -3) {$\bullet$};
		\node [style=none] (31) at (-4.5, -0.5) {};
		\node [style=none] (32) at (-2.5, 1.25) {};
		\node [style=none] (33) at (-2.5, -2.5) {};
		\node [style=none] (34) at (-0.5, 1.25) {};
		\node [style=none] (35) at (1.5, 3) {};
		\node [style=none] (36) at (1.5, -1) {};
		\node [style=none] (37) at (4, 3) {};
		\node [style=none] (38) at (6, 4) {};
		\node [style=none] (39) at (6, 2) {};
		\node [style=none] (40) at (4, -1) {};
		\node [style=none] (41) at (6, 0) {};
		\node [style=none] (42) at (6, -2) {};
		\node [style=none] (43) at (4, -4.25) {};
		\node [style=none] (44) at (6, -3.5) {};
		\node [style=none] (45) at (6, -5) {};
		\node [style=none] (46) at (-0.5, -2.5) {};
		\node [style=none] (47) at (1.5, -4.25) {};
		\node [style=none] (48) at (-5.25, -0.5) {};
		\node [style=none] (49) at (-5.25, 0) {};
		\node [style=none] (50) at (-1.75, 1.25) {};
		\node [style=none] (51) at (-1.25, 1.25) {};
		\node [style=none] (52) at (-1.75, 1.75) {};
		\node [style=none] (53) at (-1.25, 1.75) {};
		\node [style=none] (54) at (-1.75, -2.5) {};
		\node [style=none] (55) at (-1.25, -2.5) {};
		\node [style=none] (56) at (2.25, 3) {};
		\node [style=none] (57) at (2.75, 3) {};
		\node [style=none] (58) at (3.25, 3) {};
		\node [style=none] (59) at (2.25, 3.5) {};
		\node [style=none] (60) at (2.75, 3.5) {};
		\node [style=none] (61) at (3.25, 3.5) {};
		\node [style=none] (62) at (2.25, -4.25) {};
		\node [style=none] (63) at (2.75, -4.25) {};
		\node [style=none] (64) at (3.25, -4.25) {};
		\node [style=none] (66) at (3.25, -3.75) {};
		\node [style=none] (67) at (2.75, -1) {};
		\node [style=none] (68) at (3.25, -1) {};
		\node [style=none] (69) at (2.25, -1) {};
		\node [style=none] (70) at (2.25, -0.5) {};
		\node [style=none] (71) at (7, 4) {};
		\node [style=none] (72) at (7, 4.5) {};
		\node [style=none] (73) at (8, 4) {};
		\node [style=none] (74) at (8, 4.5) {};
		\node [style=none] (75) at (7.5, 4) {};
		\node [style=none] (76) at (7.5, 4.5) {};
		\node [style=none] (77) at (8.5, 4) {};
		\node [style=none] (78) at (8.5, 4.5) {};
		\node [style=none] (79) at (8, -3.25) {};
		\node [style=none] (80) at (8, -2.75) {};
		\node [style=none] (81) at (8.5, -3.25) {};
		\node [style=none] (82) at (8.5, -2.75) {};
		\node [style=none] (83) at (7, -3.25) {};
		\node [style=none] (84) at (7.5, -3.25) {};
		\node [style=none] (85) at (8.5, 0) {};
		\node [style=none] (86) at (8.5, 0.5) {};
		\node [style=none] (87) at (7, 0) {};
		\node [style=none] (88) at (7, 0.5) {};
		\node [style=none] (89) at (7.5, 0) {};
		\node [style=none] (90) at (8, 0) {};
		\node [style=none] (91) at (7.5, -1.75) {};
		\node [style=none] (92) at (8, -1.75) {};
		\node [style=none] (93) at (7, -1.75) {};
		\node [style=none] (94) at (8.5, -1.75) {};
		\node [style=none] (95) at (7, -4.75) {};
		\node [style=none] (96) at (7.5, -4.75) {};
		\node [style=none] (97) at (8, -4.75) {};
		\node [style=none] (98) at (8.5, -4.75) {};
		\node [style=none] (99) at (7, 2) {};
		\node [style=none] (100) at (7.5, 2) {};
		\node [style=none] (101) at (8, 2) {};
		\node [style=none] (102) at (8.5, 2) {};
		\node [style=none] (103) at (7, 2.5) {};
		\node [style=none] (104) at (7.5, 2.5) {};
		\node [style=none] (105) at (10, -1) {$\dots$};
		\node [style=none] (106) at (0.5, 2.5) {$+$};
		\node [style=none] (107) at (0.5, -0.25) {$-$};
		\node [style=none] (108) at (-3.5, 0.75) {$+$};
		\node [style=none] (109) at (-3.5, -1.75) {$-$};
		\node [style=none] (110) at (0.5, -3.75) {$+$};
		\node [style=none] (111) at (5, 3.75) {$+$};
		\node [style=none] (112) at (5, 2.25) {$-$};
		\node [style=none] (113) at (5, -0.25) {$+$};
		\node [style=none] (114) at (5, -1.75) {$-$};
		\node [style=none] (115) at (5, -3.65) {$+$};
		\node [style=none] (116) at (5, -4.85) {$-$};
		\node [style=none] (117) at (-7.5, -1) {};
		\node [style=none] (118) at (-6, -1) {};
		\node [style=none] (119) at (-6.75, -0.75) {$+$};
		\node [style=none] (120) at (-8, -1) {$\bullet$};
	\end{pgfonlayer}
	\begin{pgfonlayer}{edgelayer}
		\draw [style=arrow] (29.center) to (28.center);
		\draw [style=arrow] (26.center) to (25.center);
		\draw [style=arrow] (25.center) to (24.center);
		\draw [style=arrow] (16.center) to (17.center);
		\draw [style=arrow] (17.center) to (18.center);
		\draw [style=arrow] (18.center) to (19.center);
		\draw [style=arrow] (20.center) to (21.center);
		\draw [style=arrow] (22.center) to (23.center);
		\draw [style=arrow] (0.center) to (1.center);
		\draw [style=arrow] (1.center) to (2.center);
		\draw [style=arrow] (2.center) to (3.center);
		\draw [style=arrow] (3.center) to (4.center);
		\draw [style=arrow] (5.center) to (6.center);
		\draw [style=arrow] (6.center) to (7.center);
		\draw [style=arrow] (8.center) to (9.center);
		\draw [style=arrow] (9.center) to (10.center);
		\draw [style=arrow] (12.center) to (13.center);
		\draw [style=arrow] (13.center) to (14.center);
		\draw [style=blue arrow] (31.center) to (32.center);
		\draw [style=blue arrow] (31.center) to (33.center);
		\draw [style=blue arrow] (34.center) to (35.center);
		\draw [style=blue arrow] (34.center) to (36.center);
		\draw [style=blue arrow] (37.center) to (38.center);
		\draw [style=blue arrow] (37.center) to (39.center);
		\draw [style=blue arrow] (40.center) to (41.center);
		\draw [style=blue arrow] (40.center) to (42.center);
		\draw [style=blue arrow] (43.center) to (44.center);
		\draw [style=blue arrow] (43.center) to (45.center);
		\draw [style=blue arrow] (46.center) to (47.center);
		\draw [style=redline] (48.center) to (49.center);
		\draw [style=redline] (52.center) to (50.center);
		\draw [style=redline] (53.center) to (51.center);
		\draw [style=redline, bend left=90, looseness=2.00] (54.center) to (55.center);
		\draw [style=redline] (56.center) to (59.center);
		\draw [style=redline] (57.center) to (60.center);
		\draw [style=redline] (58.center) to (61.center);
		\draw [style=redline] (64.center) to (66.center);
		\draw [style=redline, bend left=90, looseness=2.00] (62.center) to (63.center);
		\draw [style=redline, bend left=90, looseness=2.00] (67.center) to (68.center);
		\draw [style=redline] (70.center) to (69.center);
		\draw [style=redline] (71.center) to (72.center);
		\draw [style=redline] (73.center) to (74.center);
		\draw [style=redline] (75.center) to (76.center);
		\draw [style=redline] (77.center) to (78.center);
		\draw [style=redline] (79.center) to (80.center);
		\draw [style=redline] (81.center) to (82.center);
		\draw [style=redline, bend left=90, looseness=2.00] (83.center) to (84.center);
		\draw [style=redline] (85.center) to (86.center);
		\draw [style=redline] (87.center) to (88.center);
		\draw [style=redline, bend left=90, looseness=2.00] (89.center) to (90.center);
		\draw [style=redline, bend left=90, looseness=2.00] (91.center) to (92.center);
		\draw [style=redline, bend left=90, looseness=1.50] (93.center) to (94.center);
		\draw [style=redline] (103.center) to (99.center);
		\draw [style=redline] (104.center) to (100.center);
		\draw [style=redline, bend left=90, looseness=1.75] (101.center) to (102.center);
		\draw [style=redline, bend left=90, looseness=1.75] (95.center) to (96.center);
		\draw [style=redline, bend left=90, looseness=1.75] (97.center) to (98.center);
		\draw [style=blue arrow] (117.center) to (118.center);
	\end{pgfonlayer}
\end{tikzpicture}
\]
		So a cap diagram $\euler{x}$ can be represented by a unique sequence $s_\euler{x} = s_1s_2\dots s_n$ of $+$ and $-$ (the edges of the branching graph leading to it); e.g. $s_{\on{cap}\otimes \on{cap}} = +-+-$.
		Such a sequence should be interpreted as 
		$$\euler{x} = ((\dots((\on{id}_{\obj{0}}\odot_{h_1} \on{id}_{\obj{1}})\odot_{h_2} \on{id}_{\obj{1}})\dots ) \odot_{h_n} \on{id}_{\obj{1}}), \quad \text{where } h_j = \begin{cases}
			0 \quad s_j = + \\
			1 \quad s_j = -
		\end{cases}.$$
		On the other hand, under the bijection $\Phi_n$, we have $\Phi_n(\euler{x}) = (\dots(((\mathbb{1})^{s_1})^{s_2} \dots))^{s_n}$
		which, by definition of $C^\pm$, is the irreducible summand with highest weight element $b_{h_1}\otimes b_{h_2}\otimes \dots \otimes b_{h_n}$.
		
		Before we complete the prove that $F$ is coboundary, we shall need the following two lemmas.
		
		\begin{lemma} \label{zComp}
			Given $\euler{y}\in\mathscr{D}_{\mobijd{n}}$ and $h\in \{0,1\}$, we have $\Phi_{n+1}\left(\on{id}_{\obj{1}} \odot_h \euler{y}\right) = \begin{cases}
				\Phi_n(\euler{y})_+ \quad h=0\\
				\Phi_n(\euler{y})_- \quad h=1
			\end{cases}$.
		\end{lemma}

        \begin{proof}
            Since the equality in question is that of two irreducible components of $B^{\otimes n+1}$, it suffices to show that they have the same highest weight element.
	Keeping the notation of the proof of \autoref{coboundaryF}, let $s_\euler{y} = s_1s_2\dots s_n$ be a $\pm$ sequence describing the unique path to $\euler{y}$ in the branching graph of $\Phi_n$, and $h_1h_2\dots h_n$ the corresponding binary sequence.
	By definition, the highest weight element of $\Phi_n(\euler{y})= (\dots((\mathbb{1})^{s_1})^{s_2} \dots)^{s_n}$ is $b_{h_1}\otimes b_{h_2}\otimes \dots \otimes b_{h_n}$.
	Thus, the highest weight element of $\Phi_n(\euler{y})_+$ is given by $b_0\otimes b_{h_1}\otimes b_{h_2}\otimes \dots \otimes b_{h_n}$ which is also the highest weight element of $(\dots(((\mathbb{1})^+)^{s_1})^{s_2} \dots)^{s_n} = \Phi_{n+1}(\on{id}_{\obj{1}}\odot_0 \euler{y})$.
	
	The highest weight element of $\Phi_n(\euler{y})_-$ is given by $b_0\otimes f(b_{h_1}\otimes b_{h_2}\otimes \dots \otimes b_{h_n})$, which, by a repeated application of \autoref{crysTensor}, is equal to $ b_0\otimes b_{h_1}\otimes \dots \otimes f(b_{h_i}) \otimes \dots \otimes b_{h_n}$, where $i$ is minimum in $\{1,\dots, n\}$ such that $h_i=0$ and $b_{h_{i+1}}\otimes \dots \otimes b_{h_n}$ is a highest weight element of $B^{\otimes n-i}$.
	The component of $B^{\otimes n-i}$ generated by the highest weight element $b_{h_{i+1}}\otimes \dots \otimes b_{h_n}$ is precisely that which corresponds to the diagram $\euler{y}_{>i}$ under $\Phi_{n-i}$. And the condition that $b_{h_{i+1}}\otimes \dots \otimes b_{h_n}$ is of highest weight guarantees that $\mathtt{hk}_i(\euler{y})=0$; 
    indeed, if $\mathtt{hk}_i(\euler{y})\neq 0$, we would have $h_{i+1}=1$, which according to \autoref{crysTensor} implies $e(b_{h_{i+1}}\otimes \dots \otimes b_{h_n}) = b_0\otimes b_{h_{i+2}}\otimes \dots \otimes b_{h_n} \neq 0$, meaning that $b_{h_{i+1}}\otimes \dots \otimes b_{h_n}$ is not of highest weight. 
    It follows by minimality of $i$ that $i$ is the leftmost through strand of $\euler{y}$.
	Furthermore, $b_0\otimes b_{h_1}\otimes \dots \otimes f(b_{h_i}) \otimes \dots \otimes b_{h_n} = b_0\otimes b_{h_1}\otimes \dots \otimes b_1 \otimes \dots \otimes b_{h_n}$, where $b_0$ and $b_1$ occur in the first and $(i+1)^\text{th}$ places, respectively.
    On the other hand, $\on{id}_{\obj{1}}\odot_1 \euler{y}$ is by definition the diagram resulting from hooking the first strand to the leftmost through strand, namely $i$ of $\euler{y}$, so that $(1,i+1)\in K(\on{id}_{\obj{1}}\odot_1 \euler{y})$. 
    Therefore, the highest weight element of $\Phi_{n+1}(\on{id}_{\obj{1}}\odot_1 \euler{y})$ is also $b_0\otimes b_{h_1}\otimes \dots \otimes b_1 \otimes \dots \otimes b_{h_n}$, where $b_0$ and $b_1$ occur in the first and $(i+1)^\text{th}$ places, respectively.
\end{proof}
		
		\begin{lemma} \label{ProjEmb}
			Given $\euler{x}\in\mathscr{D}_{\mobijd{n}}$ with $\mathtt{th}(\euler{x}) = k$,
			\begin{itemize}
				\item $F(\euler{j}_k\circ \euler{x}):B^{\otimes n}\to B^{\otimes k}$ is the projection map onto the summand $\Phi_n(\euler{x})$.
				\item $F(\overline{\euler{x}}\circ \euler{j}_k):B^{\otimes k}\to B^{\otimes n}$ is the embedding of the top summand $B_k$ of $B^{\otimes k}$ onto the summand $\Phi_n(\euler{x})$.
			\end{itemize}
		\end{lemma}

\begin{proof}
	Note that a basis for $B^{\otimes n}$ is given by elements of the form $b_{i_1}\otimes b_{i_2}\otimes \dots \otimes b_{i_n}$ where $i_j\in\{0,1\}$, and to simplify notation, let us write such an element simply as $i_1\otimes i_2\otimes\dots\otimes i_n$.
	Recall that $K(\euler{x})$ is the set of pairs $(r,s)$ such that there is a cap in in $\euler{x}$ joining the $r^\text{th}$ and $s^\text{th}$ strands.
	By definition of $F(\on{cap})$, we have $F(\euler{x})(i_1\otimes \dots \otimes i_n) = 0$ unless $(i_r,r_s)=(0,1)$ whenever $(r,s)\in K(\euler{x})$.
	Moreover, $F(\euler{x})$ is equal to the element resulting from deleting every $i_r$ and $i_s$ from $i_1\otimes \dots \otimes i_n$ (leaving the rest unchanged) whenever $(r,s)\in K(\euler{x})$.
	It follows that $F(\euler{x})$ is the projection map $B^{\otimes n}\to B^{\otimes k}$ projecting onto the product of the factors corresponding to through strands of $\euler{x}$.
	
	Next, observe that, by functoriality and additivity of $F$ as well as the properties of the Jones--Wenzl projectors in \autoref{JWuniqueness}, the morphism $F(\euler{j}_k)$
	\begin{enumerate}
		\item can be written as $\on{id}_{B^{\otimes k}}+x$, where $x$ is in the ideal $F(\mathtt{th}_{<k}(k,k))$;
		\item is an idempotent;
		\item is annihilated by $F(\on{cup}_{\{i\},k})$ and $F(\on{cap}_{\{i\},k})$ (by \autoref{annihilation} and functoriality).
	\end{enumerate}
	Since $F$ is an equivalence, crystal morphisms satisfying all of these three properties are unique by \autoref{JWuniqueness}.
	However, the projection onto the top summand $B^{\otimes k}\to B^{\otimes k}$ also satisfies all three properties, and hence $F(\euler{j}_k)$ is the projection onto the top summand $B_k\subseteq B^{\otimes k}$.
	It follows that $F(\euler{j}_k\circ \euler{x})$ is the projection map onto the irreducible summand whose highest weight element is $i_1\otimes \dots \otimes i_n$ with $(i_r,i_s)=(0,1)$ whenever $(r,s)\in K(\euler{x})$.
	But that is exactly the summand $\Phi_n(\euler{x})$. This proves the first claim of the lemma.
	The second claim is proved by a dual and completely analogous argument. 
\end{proof}
	
	\begin{theorem} \label{coboundaryF}
		The equivalence $F:\mathbf{CrysTL} \to \mathfrak{sl}_2\mathbf{-Crys}$ is a coboundary functor.
	\end{theorem}
	
	\begin{proof}
		Let $\sigma^{\mathcal{TL}}$ and $\sigma^{\mathfrak{sl}_2}$ denote the coboundary structures of $\mathbf{CrysTL}$ and $\mathfrak{sl}_2\mathbf{-Crys}$, respectively; and, to simplify notation, let $\sigma^{\mathcal{TL}}_{m,n}=\sigma^{\mathcal{TL}}_{\mobijs{m},\mobijs{n}}$ and $\sigma^{\mathfrak{sl}_2}_{m,n}=\sigma^{\mathfrak{sl}_2}_{B^{\otimes m},B^{\otimes n}}$.
		Since the structure morphisms of $F$ are trivial, the diagram of \autoref{CobFunctor} reduces to showing that $F\left(\sigma^{\mathcal{TL}}_{A,B}\right)=\sigma^{\mathfrak{sl}_2}_{F(A),F(B)}$ for all objects $A$ and $B$ of $\mathbf{CrysTL}$.
		By naturality of the coboundary structure, it suffices to show that $F$ intertwines the coboundary structures before applying the Cauchy completion, i.e. we wish to show $F\left(\sigma^{\mathcal{TL}}_{m,n}\right)=\sigma^{\mathfrak{sl}_2}_{m,n}$ for all $m,n\in \mathbb N$.
		Since all the coboundary morphisms $\sigma_{m,n}$ can be recursively recovered from $\sigma_{1,n}$ (see \autoref{CobCats}), it suffices to show that $F\left(\sigma^{\mathcal{TL}}_{1,n}\right)=\sigma^{\mathfrak{sl}_2}_{1,n}$ for all $n\in \mathbb N$.
		
		Let us examine the action of $\sigma^{\mathfrak{sl}_2}_{1,n}$ on connected components.
		By naturality of the coboundary structure, the map $\sigma^{\mathfrak{sl}_2}_{1,n}$ must send the connected components of $B\otimes C$ to the components of $C\otimes B$ for every connected component $C$ of $B^{\otimes n}$, since $\sigma^{\mathfrak{sl}_2}_{1,n}$ commutes with projections onto any component.
		It follows that $\sigma^{\mathfrak{sl}_2}_{1,n}$ maps $C_+$ to $C^+$ and $C_-$ to $C^-$ (identically) for every component $C$ of $B^{\otimes n}$, since weights must be preserved.
		
		Observe now that, by \autoref{ProjEmb}, for any $\euler{x}\in \mathscr{D}_{\mobijd{n+1}}$ with $\mathtt{th}(\euler{x})= k$, the map 
		$$F(\overline{\kappa_{1,n}(\euler{x})}\circ \euler{j}_k\circ \euler{x}) = F(\overline{\kappa_{1,n}(\euler{x})}\circ \euler{j}_k) \circ F(\euler{j}_k \circ \euler{x}):B^{\otimes n+1}\to B^{\otimes n+1}$$
		takes the component $\Phi_{n+1}(\euler{x})$ identically to the component $\Phi_{n+1}(\kappa_{1,n}(\euler{x}))$ and kills every other component. 
		Hence, their sum $F(\sigma_{1,n}^{\mathcal{TL}})$ permutes all the summands of $B^{\otimes n+1}$ via $\Phi_{n+1}(\euler{x})\mapsto \Phi_{n+1}(\kappa_{1,n}(\euler{x}))$.
		
		To finish the proof, we need to show that, writing $\Phi_{n+1}(\euler{x}) = C_\pm$ for a connected component $C$ of $B^{\otimes n}$, we get $\Phi_{n+1}(\kappa_{1,n}(\euler{x})) = C^\pm$, whence the actions of $F(\sigma_{1,n}^{\mathcal{TL}})$ and $\sigma_{1,n}^{\mathfrak{sl}_2}$ on all summands agree. 
		Indeed, writing $\euler{x}=\on{id}_{\obj{1}}\odot_h \euler{y}$ for some $\euler{y} \in \mathscr{D}_{\mobijd{n}}$ and $h\in \{0,1\}$, we see that $\Phi_{n+1}(\euler{x})=\Phi_n(\euler{y})_\pm$ by \autoref{zComp}.
		Now, 
		$$\Phi_{n+1}(\kappa_{1,n}(\euler{x})) = \Phi_{n+1}(\kappa_{1,n}(\on{id}_{\obj{1}}\odot_h \euler{y})) = \Phi_{n+1}(\euler{y} \odot_h \on{id}_{\obj{1}}) = \Phi_n(\euler{y})^\pm.$$    
	\end{proof}
	
    \begin{remark}
        In \cite{HKam}, a commutor for $\uqsl$ is defined for $q \neq 0$, which is reminiscent of that for the crystal case. 
        It would be interesting to give a diagrammatic description of that coboundary structure on $\mathbf{Rep}(\uqsl)$, building on our description. 
        However, in contrast to the crystal case, the coefficients appearing in the linear combinations in such a description would likely be complicated, and beyond the scope of this article.
    \end{remark}

    \begin{remark}
        A crucial ingredient in our diagrammatic construction of the commutor $\sigma$ in \autoref{sec:commutor} is the permutation $\kappa_{m,n}$ of the set $\mathscr{D}_{\mobijd{m+n}}$ of cap diagrams (equivalently, of the summands of $B^{\otimes m+n}$). 
        Finding a representation-theoretic interpretation of such permutations is, to the best of the authors' knowledge, an open problem.
    \end{remark}
    
	\section{Fiber Functors}\label{subsec:fiberfunctors}

    In this section, we assume $\Bbbk = \mathbb{C}$.
    
    Recall that a {\it fiber functor} on a multiring category $\mathcal{C}$ over $\Bbbk$ is a faithful, exact, $\Bbbk$-linear monoidal functor $U: \mathcal{C} \rightarrow \mathbf{vec}$. 
    We denote the category of fiber functors on $\mathcal{C}$, whose morphisms are the monoidal transformations of fiber functors, by $\mathbf{Fib}(\mathcal{C})$.
    
	\begin{remark}
		If $\mathcal{C}$ is rigid and $\mathbb{1}$ is simple, e.g. if $\mathcal{C}$ is a tensor category, then an exact $\Bbbk$-linear monoidal functor $\mathcal{C} \rightarrow \mathbf{vec}_{\Bbbk}$ is automatically faithful, see \cite[Corollaire~2.10]{De}. 
        For a counterexample in the non-rigid case, see \cite[Example~2.20]{DSPS}.
	\end{remark}

    \begin{lemma}\label{fiberCauchy}
        Any $\Bbbk$-linear monoidal functor on $\mathbf{CrysTL}$ is a fiber functor. 
        Moreover, there are equivalences between the category of fiber functors on $\mathbf{CrysTL}$, the category $\mathbf{MonFun}_{\Bbbk}(\mathcal{TL}_{0}(\Bbbk),\mathbf{vec})$, and that of fiber functors on $\mathfrak{sl}_2\mathbf{-Crys}$.
    \end{lemma}

    \begin{proof}
        By \autoref{cor:semisimple}, $\mathbf{CrysTL}$ is semisimple, so any $\Bbbk$-linear functor from it is exact. 
        By \autoref{tensorFunctorsAreFaithful}, such a functor is also necessarily faithful.
        Thus $\mathbf{Fib}(\mathbf{CrysTL}) = \mathbf{MonFun}_{\Bbbk}(\mathbf{CrysTL},\mathbf{vec})$. 

        For the latter claim, we have found equivalences
        \[
        \mathbf{MonFun}_{\Bbbk}(\mathcal{TL}_{0}(\Bbbk), \mathbf{vec}) \simeq \mathbf{MonFun}_{\Bbbk}(\mathbf{CrysTL},\mathbf{vec}) = \mathbf{Fib}(\mathbf{CrysTL}) \simeq \mathbf{Fib}(\mathfrak{sl}_2\mathbf{-Crys}).
        \]
    \end{proof}

    A detailed account of fiber functors on $\mathbf{Rep}(\uqsl)$ is given in \cite{EO}, which uses the results of \cite{Tur} to conclude that $\mathbf{Rep}(\uqsl)$ is the free monoidal category on a self-dual object of categorical dimension $[2]_{q}$. We give a slight reformulation of one of the observations in \cite{EO}, \cite{Tur}.
    \begin{definition}
     Let $\mathcal{C}$ be a rigid monoidal category, with a fixed rigid structure consisting of choices of right dual objects $X^{\vee}$ for $X \in \mathcal{C}$, and evaluation and coevaulation morphisms $\on{ev}_{X}: \mathbb{1} \rightarrow X^{\vee} \otimes X$ and $\on{coev}_{X}: X \otimes X^{\vee} \rightarrow \mathbb{1}$. 
     We denote by $\mathcal{E}(\mathcal{C})$ the category whose objects are triples $(X, \eta_{X}: \mathbb{1} \rightarrow X \otimes X, \varepsilon_{X}: X \otimes X \rightarrow \mathbb{1})$ such that $(X \otimes \varepsilon_{X}) \circ (\eta_{X} \otimes X) = \on{id}_{X} = (\varepsilon_{X} \otimes X) \circ (X \otimes \eta_{X})$. 
     For morphisms, we set
     \[
     \on{Hom}_{\mathcal{E}(\mathcal{C})(X,\eta_{X},\varepsilon_{X}),(X,\eta_{X},\varepsilon_{X})} = \setj{f \in \on{Hom}_{\mathcal{C}}(X,Y) \; | \; (f \otimes f) \circ \eta_{X} = \eta_{Y} \text{ and } \varepsilon_{Y} \circ (f \otimes f) = \varepsilon_{X}}.
     \]

     Let $\mathcal{Y}(\mathcal{C})$ be the category whose objects are pairs $(X,\Phi_{X})$, where $X \in \mathcal{C}$ and $\Phi_{X}: X \rightarrow X^{\vee}$ is an isomorphism. For morphisms, we set
     \[
     \on{Hom}_{\mathcal{Y}(\mathcal{C})}((X,\Phi_{X}), (Y,\Phi_{Y})) = \setj{g \in \on{Hom}_{\mathcal{C}}(X,Y) \; | \; \Phi_{X} = g^{\vee} \circ \Phi_{Y} \circ g \text{ and } \Phi_{X}^{-1} = g \circ \Phi_{Y}^{-1} \circ g^{\vee}}.
     \]
     Clearly, $\mathcal{Y}(\mathcal{C})$ is a groupoid, and the latter condition can be replaced by $\Phi_{Y} = (g^{-1})^{\vee} \circ \Phi_{X} \circ g^{-1}$.
    \end{definition}

 The following can be used to prove \cite[Theorem~2.1]{EO}, \cite[Chapter~XII]{Tur}.
   \begin{lemma}\label{Turaev}
       The assignments
       \[
       \begin{aligned}
           \mathcal{E}(\mathcal{C}) &\longleftrightarrow \mathcal{Y}(\mathcal{C}) \\
           (X, \eta_{X}, \varepsilon_{X}) &\longmapsto (X, (X\otimes \varepsilon_{X})\circ  (\on{coev}_{X} \otimes X)) \\
           f &\longmapsto f \\
           (X, (\Phi_{X}^{-1} \otimes X) \circ \on{coev}_{X}, \on{ev}_{X} \circ (X \otimes \Phi_{X})) &\longmapsfrom (X,\Phi_{X}) \\
           g &\longmapsfrom g
       \end{aligned}
       \]
       define mutually quasi-inverse equivalences of categories.
   \end{lemma}
   The category $\mathcal{E}(\mathcal{C})$ is that of monoidal functors and transformations from the category $\mathcal{B}$ obtained by removing the circle evaluation relation from the definition of the ordinary Temperley--Lieb category for $q \neq 0$. 
   Thus in particular the category of fiber functors on $\mathcal{TL}_{q}(\Bbbk)$ embeds in it. 
   On the other hand, for $\mathcal{C} = \mathbf{vec}_{\Bbbk}$, the category $\mathcal{Y}(\mathcal{C})$ is the groupoid of non-degenerate bilinear forms. 
   Further, in this case the choice of the rigid structure is inessential since the set of rigid structures on $\mathcal{C}$ is a trivial torsor on the group of monoidal automorphisms $\on{Aut}_{\otimes}(\on{Id}_{\mathcal{C}})$ of the identity functor on $\mathcal{C}$.

   Given a monoidal functor $F: \mathcal{TL}_{0}(\Bbbk) \rightarrow \mathcal{C}$, one can still define a functor analogous to that from $\mathcal{E}(\mathcal{C})$ to $\mathcal{Y}(\mathcal{C})$ using the morphism
   $F(1\otimes \on{cap}) \circ F(1)\otimes \on{coev}_{F(1)} \in \mathcal{C}(F(1), F(1)^{\vee})$, but this is merely an analogue of translation between bilinear forms and linear maps $V \rightarrow V^{\ast}$.
  Crucially, $F(1\otimes \on{cap}) \circ F(1)\otimes \on{coev}_{F(1)} \in \mathcal{C}(F(1), F(1)^{\vee})$ is not invertible, so we do not have an equation analogous to
   \begin{equation}\label{unitcounit}
    \left((X\otimes \varepsilon_{X})\circ  (\on{coev}_{X} \otimes X)\right)^{-1} = (X \otimes \on{ev}_{X}) \circ (\eta_{X} \otimes X),
   \end{equation}
  so $F(\on{cap})$ and $F(\on{cup})$ do not determine each other. 
  Indeed, in what follows we will consider the space of fiber functors with fixed space $F(1)$ and bilinear form $F(\on{cap})$.

 Given a bilinear map $\mathsf{b}$ on a space $V$, let $\mathcal{L}(\mathsf{b})$ denote the {\it left radical of $\mathsf{b}$}, given by $\setj{v \in V \; | \; \mathsf{b}(v,-) = 0}$, and denote by $\mathcal{R}(\mathsf{b})$ the right radical of $\mathsf{b}$.
 
  \begin{lemma}\label{lem:e0}
      There is an equivalence of categories
      \[
      \mathbf{Fib}(\mathbf{CrysTL}) \simeq \mathcal{E}_{0},
      \]
      where the objects of $\mathcal{E}_{0}$ are triples $(V,\mathsf{b},\mathsf{t})$, where $\mathsf{b}$ is a bilinear form on $V$, and $\mathsf{t} \in \mathcal{R}(\mathsf{b})\otimes \mathcal{L}(\mathsf{b})$ is such that $\mathsf{b}(\mathsf{t}) = 1$. 
      A morphism $f \in \on{Hom}_{\mathcal{E}_{0}}((V,\mathsf{b},\mathsf{t}),(V',\mathsf{b}',\mathsf{t}'))$ is a linear map which is a morphism of bilinear spaces from $\mathsf{b}$ to $\mathsf{b}'$ such that $(f \otimes f)(\mathsf{t}) = \mathsf{t}'$. 
  \end{lemma} 

  \begin{proof}
      Using \autoref{fiberCauchy}, it suffices to observe that
      \[
      \begin{aligned}
      \mathbf{MonFun}_{\Bbbk}(\mathcal{TL}_{0}(\Bbbk),\mathbf{vec}) &\rightarrow \mathcal{E}_{0} \\
      U &\mapsto (U(\obj{1}), U(\on{cap}), U(\on{cup})) \\
      \gamma &\mapsto \gamma_{\obj{1}}
      \end{aligned}
      \]
      is an equivalence. And indeed, the circle evaluation is satisfied if and only if $\mathsf{b}(\mathsf{t}) = 1$, and zig-zag relations are satisfied if and only if $\mathsf{t} \in V \kotimes \mathcal{L}(\mathsf{b})$ and $\mathsf{t} \in \mathcal{R}(\mathsf{b}) \kotimes V$ respectively, showing that $\mathsf{t} \in \mathcal{R}(\mathsf{b}) \kotimes \mathcal{L}(\mathsf{b})$. 
      For the morphisms, this follows from the presentation defining $\mathcal{TL}_{0}(\Bbbk)$.
  \end{proof}

In this section we will consider orthogonal decompositions of bilinear forms. 
To emphasize that the category of bilinear forms is not additive, we use the symbol $\perp$ for such decompositions, rather than the symbol $\oplus$, indicating the decomposition of the underlying spaces.
\begin{corollary}
   Let $(V,\mathsf{b},\mathsf{t}) \in \mathcal{E}_{0}$, and let $(V',\mathsf{b}',\mathsf{t})$ be a triple consisting of a space $V'$, a bilinear form $\mathsf{b}'$ on $V'$ and $\mathsf{t} \in \mathcal{R}(W) \otimes \mathcal{L}(W)$ be such that $\mathsf{b}'(\mathsf{t}') = 0$. 
   The triple $(V',\mathsf{b}',\mathsf{t}') \triangleright (V,\mathsf{b},\mathsf{t}) := (V' \oplus V, \mathsf{b}' \perp \mathsf{b}, \mathsf{t}+\mathsf{t}')$ is an object of $\mathcal{E}_{0}$. 
   We refer to it as the {\it inflation} of $(V,\mathsf{b},\mathsf{t})$ by $(V',\mathsf{b}',\mathsf{t})$. 
\end{corollary}

Observe that for $q \neq 0$, we would need $\mathsf{b},\mathsf{b}'$ to be non-degenerate, and the resulting triple would specify a fiber functor for parameter value $q + q'$ rather than $q$.

\begin{proposition}\label{notagroupoid}
   The projection $\pi: V' \oplus V \twoheadrightarrow V $ and the injection $\iota: V \hookrightarrow V' \oplus V$ are morphisms of $\mathcal{E}_{0}$. In particular, $\mathbf{Fib}(\mathbf{CrysTL})$ is not a groupoid.
\end{proposition}

\begin{proof}
    This follows immediately from $(\pi \otimes \pi)\circ (\mathsf{t} + \mathsf{t}') = \mathsf{t}$, and $(\mathsf{b}\perp \mathsf{b}')(\mathsf{t}+\mathsf{t}') = \mathsf{b}(\mathsf{t})  + \mathsf{b}'(\mathsf{t}') = 1$.
\end{proof}

\autoref{notagroupoid} gives a different proof of the fact that $\mathbf{CrysTL}$ is not rigid, in view of \cite[Lemma~3.4]{BV}. We give a Hopf-theoretic interpretation of this.
\begin{proposition}\label{prop:HopfModules}
   Let $\mathcal{C}$ be a ring category admitting fiber functors $U,U'$ and a non-invertible monomorphism $\upsilon: U \hookrightarrow U'$. Then $\mathcal{C}$ is not rigid.
\end{proposition}

\begin{proof}
   Recall that the functors $U(-)^{\ast}$ and $U'(-)^{\ast}$ are of the form $\on{Hom}_{\mathcal{C}}(-,A)^{\ast}$ and $\on{Hom}_{\mathcal{C}}(-,A)^{\ast}$ for objects $A,A' \in \on{Ind}(\mathcal{C})$ -- see e.g. \cite[Theorem~1.10.1]{EGNO}. 
   The oplax monoidal structures on $U,U'$ respectively imply that $\on{Hom}_{\mathcal{C}}(-,A)$ and $\on{Hom}_{\mathcal{C}}(-,A)$ are lax monoidal, and thus so are their $\on{Ind}$-extensions $\on{Hom}_{\on{Ind}(\mathcal{C})}(-,A)$ and $\on{Hom}_{\on{Ind}(\mathcal{C})}(-,A)$ (see \cite[Proposition~2.27]{SZ} for details). 
   Since $\on{Ind}(\mathcal{C})$ is locally finitely presentable, we have a monoidal equivalence $\on{Ind}(\mathcal{C}) \simeq \on{Lex}_{\Bbbk}(\mathcal{C}^{\on{op}},\mathbf{Vec})$ between the $\on{Ind}$-completion of $\mathcal{C}$ and the category of left exact presheaves on $\mathcal{C}$, where the latter is endowed with Day convolution. 
   The lax monoidal structures on $\on{Hom}_{\on{Ind}(\mathcal{C})}(-,A)$ and $\on{Hom}_{\on{Ind}(\mathcal{C})}(-,A)$ correspond to monoid structures on them in $\on{Lex}_{\Bbbk}(\mathcal{C}^{\on{op}},\mathbf{Vec})$ (see \cite[Example~3.2.2]{Da}, \cite[Proposition 22.1]{MMSS}). 
   Thus $A$ and $A'$ are monoids in $\on{Ind}(\mathcal{C})$ and $\upsilon$ gives rise to a monoid epimorphism $\overline{\upsilon}: A \twoheadrightarrow A'$. 
   It is easy to see that then $A' \simeq A/\on{Ker}(\overline{\upsilon})$ as monoids. Thus, we find a full, faithful, finitary functor $\Upsilon: \on{mod}_{\mathcal{C}}(A/I) \hookrightarrow \on{mod}_{\mathcal{C}}(A)$, preserving compact objects.

   However, since $U$ is a fiber functor, we may write $\mathcal{C} = \on{mod}U(A)$. 
   Then the category $\on{mod}_{\mathcal{C}}(A)$ is precisely the category of Hopf modules for the Hopf algebra $U(A)$. 
   Thus, if we assume $\mathcal{C}$ to be a tensor category, and hence $A$ to be Hopf, then by the fundamental theorem of Hopf modules we find $\on{mod}_{\mathcal{C}}(A) \simeq \mathbf{Vec}$. 
   Then, $\Upsilon$ must be an equivalence: the image of its restriction to compact object defines a non-zero Cauchy complete subcategory of $\mathbf{vec}$, which must coincide with $\mathbf{vec}$. 
   This implies that $AI = \on{Im}(A \otimes I \rightarrow A \otimes A \xrightarrow{\mu_{A}} A) = 0$. 
   Thus, so is the image of the inclusion $I \hookrightarrow A$, obtained by precomposing the morphism defining $AI$ with $\eta_{A} \otimes I$, where $\eta_{A}$ is the unit of $A$. 
   Thus $I = 0$ and so $\upsilon$ is invertible.
\end{proof}

We now continue our study of the category $\mathbf{Fib}(\mathbf{CrysTL})$.
Recall from \cite{Ga} that the isometry classes of indecomposable degenerate bilinear forms are given by nilpotent Jordan blocks, $J_{1} =
\left(
\begin{smallmatrix}
    0
\end{smallmatrix}
\right)$, $J_{2} = 
\left(
\begin{smallmatrix}
    0 & 1 \\
    0 & 0
\end{smallmatrix}
\right)$ and so on. 
We denote the bilinear form associated to $J_{n}$ by $\mathsf{J}_{n}$. 
Again following \cite{Ga}, we find that a given bilinear form $\mathsf{b}$ decomposes as 
\[
\mathsf{b} = \mathsf{b}_{\on{ndeg}} \perp \mathsf{J}_{1}^{\perp m_{1}} \perp \mathsf{J}_{2}^{\perp m_{2}} \perp \cdots \perp \mathsf{J}_{k}^{\perp m_{k}},
\]
where the isomorphism class of the non-degenerate form $\mathsf{b}_{\on{ndeg}}$ and the multiplicities $m_{1},\ldots, m_{k}$ are uniquely determined.

Denoting by $e_{1},\ldots,e_{n}$ the standard basis vectors, we see that $\mathcal{R}(\mathsf{J}_{n}) = \on{Span}\setj{e_{1}}$ and $\mathcal{L}(\mathsf{J}_{n}) = \on{Span}\setj{e_{n}}$. 
In particular, $(\mathsf{J}_{n})|_{\mathcal{R}(\mathsf{J}_{n}) \otimes \mathcal{L}(\mathsf{J}_{n})} \neq 0$ implies $n = 2$.

\begin{proposition}\label{prop:inflationreduction}
 Any fiber functor is isomorphic to a functor of the form $(V',\mathsf{b}',\mathsf{t}') \triangleright (\Bbbk^{\oplus 2m},\mathsf{J}_{2}^{\perp m}, \mathsf{t})$, where $\mathsf{t} \in \mathcal{R}(\mathsf{J}_{2}^{\perp m}) \otimes \mathcal{L}(\mathsf{J}_{2}^{\perp m})$ and $m > 0$.
\end{proposition}

\begin{proof}
    Let $(V, \mathsf{b},\mathsf{t}) \in \mathcal{E}_{0}$. 
    Choose a decomposition $\mathsf{b} = \mathsf{b}_{2} \perp \mathsf{b}'$
    so that $\mathsf{J}_{2}$ does not appear in the decomposition of $\mathsf{b}'$, and $\mathsf{b}_{2} \simeq \mathsf{J}_{2}^{\perp m}$. 
    Write $V = V' \oplus V_{2}$ for the resulting direct sum decomposition of $V$.
    Write $\mathsf{t} = \mathsf{t}_{2} + \mathsf{t}'$ for the resulting decomposition of $\mathsf{t}$. 
    Then $\mathsf{b}'|_{\mathcal{R}(\mathsf{b}') \otimes \mathcal{L}(\mathsf{b}')} = 0$, and thus $\mathsf{b}(\mathsf{t}') = 0$, and $\mathsf{b}'(\mathsf{t}') = 0$. 
    Further, $\mathsf{t}_{2}\neq 0$ as $\mathsf{b}(\mathsf{t}_{2} + \mathsf{t}') \neq 0$. Thus $m > 0$. 
    We find $(V,\mathsf{b}, \mathsf{t}) = (V',\mathsf{b}',\mathsf{t}') \triangleright (V_{2}, \mathsf{b}_{2}, \mathsf{t}_{2}) \simeq (V',\mathsf{b}',\mathsf{t}') \triangleright (\Bbbk^{\oplus 2m},\mathsf{J}_{2}^{\perp m}, \mathsf{t})$.
\end{proof}

For a fixed vector space $V$ and bilinear form $\mathsf{b}$ on $V$,
we will now study the problem of determining the isomorphism classes of fiber functors $U$ satisfying $U(1) = V$ and $U(\on{cap}) = \mathsf{b}$. 
As seen in \autoref{unitcounit}, for generic non-zero $q$, there is only one such isomorphism class. We will see that for $q=0$ this is not the case.

Let $\on{Aut}(\mathsf{b})$ denote the isometry group of $\mathsf{b}$. The following is an immediate consequence of \autoref{lem:e0}:
  \begin{corollary}
  Let $V$ be a finite-dimensional vector space and let $\mathsf{b}$ be a bilinear form on $V$. Then 
      \[
      \setj{U \in \mathbf{Fib}(\mathbf{CrysTL}) \; | \; U(1) = V , U(\on{cap}) = \mathsf{b}}\!/\!\simeq = (\mathcal{R}(\mathsf{b})\otimes \mathcal{L}(\mathsf{b}))/\left\langle \mathsf{t} \sim \mathsf{t}' \text{ if }\exists \varphi \in \on{Aut}(\mathsf{b}): (\varphi \otimes \varphi)(\mathsf{t}) = \mathsf{t}' \right\rangle.
      \] 
  \end{corollary}

 We say that $\mathsf{b}$ is {\it of even type} if no odd-dimensional Jordan blocks appear in its orthogonal decomposition into indecomposable forms. 
 If the only Jordan blocks appearing are those of the form $J_{2}$, we say that $\mathsf{b}$ is {\it of homogeneous even type.} 
 We will also frequently decompose $\mathsf{b}$ as $\mathsf{b}_{\on{ndeg}} \perp \mathsf{b}_{\on{tdeg}}$, the orthogonal sum of its {\it non-degenerate part} and {\it totally degenerate part}.
 
 We now recall some of the results of \cite[Sections 4,5]{Dj}. Observe that we use the opposite convention for Jordan blocks to \cite{Dj}, where such blocks are subdiagonal. 
 By \cite[Proposition~5.1]{Dj}, if $\mathsf{b}$ is of even type, then the decomposition $\mathsf{b} = \mathsf{b}_{\on{ndeg}} \perp \mathsf{b}_{\on{tdeg}}$ is unique, given by a generalized eigenspace decomposition. 
 In that case, $\on{Aut}(\mathsf{b}) = \on{Aut}(\mathsf{b}_{\on{tdeg}}) \times \on{Aut}(\mathsf{b}_{\on{ndeg}})$, and since $\mathcal{R}(\mathsf{b}), \mathcal{L}(\mathsf{b}) \subseteq \mathsf{b}_{\on{tdeg}}$, we can reduce to study the orbits of the action of $\on{Aut}(\mathsf{b}_{\on{tdeg}})$ on $\mathcal{R}(\mathsf{b}) \otimes \mathcal{L}(\mathsf{b})$.

 An even-dimensional indecomposable degenerate bilinear form $\mathsf{d}$ given by a nilpotent Jordan block decomposes as a direct sum of two $\on{Aut}(\mathsf{d})$-invariant subspaces $\mathcal{L}^{\infty}(\mathsf{b})$ and $\mathcal{R}^{\infty}(\mathsf{b})$. 
 We have $\mathcal{L}^{\infty}(\mathsf{J}_{m}) = \setj{e_{2},e_{4},\ldots, e_{2m}}$ and $\mathcal{R}^{\infty}(\mathsf{J}_{m}) = \setj{e_{1},e_{3},\ldots, e_{2m-1}}$. Ordering the standard basis as $(e_{1},e_{3},\ldots,e_{2m-1}, e_{2},\ldots,e_{2m})$, the matrix of the bilinear form turns to 
 $
 \left(
 \begin{smallmatrix}
     0 & I_{m} \\
     J_{m} & 0
 \end{smallmatrix}
 \right)
 $ 
 where $I_{m}$ is the $m\times m$ identity matrix and $J_{m}$ an $m\times m$ nilpotent Jordan block. 
 More generally, for a totally degenerate form with no odd-dimensional components, we can find a basis with respect to which the bilinear form's matrix is  $
 \left(
 \begin{smallmatrix}
     0 & I \\
     J & 0
 \end{smallmatrix}
 \right)
 $ where $I$ is an identity matrix and $J$ is a direct sum of Jordan blocks with eigenvalue zero. With respect to this basis, the matrix of an isometry is of the form
 \begin{equation}\label{blockmatrix}
 \left(
 \begin{smallmatrix}
     A & 0 \\
     0 & (A^{-1})^{T}
 \end{smallmatrix}
 \right),
 \end{equation}
 where $AJ = JA$. 
 In particular, $A$ is a direct sum of polynomials in the respective blocks of $J$. 
 We may formulate the previous observation as follows:
 \begin{lemma}\label{lem:factorRL}
     The action of $\on{Aut}(\mathsf{b})$ on $\mathcal{R}(\mathsf{b}) \otimes \mathcal{L}(\mathsf{b})$ factors through the action of $\on{GL}(\mathcal{R}(\mathsf{b}))$, where the latter is given by
     $A \mapsto (A\cdot-) \otimes ((A^{-1})^{T} \cdot -)$, following \autoref{blockmatrix}. 
 \end{lemma}
 
 We may identify the actions of $\on{Aut}(\mathsf{b})$ and of $\on{Aut}_{\Bbbk[x]}(\mathcal{R}^{\infty}(\mathsf{b}))$ on
 $\mathcal{R}^{\infty}(\mathsf{b})$, where we endow $\mathcal{R}^{\infty}(\mathsf{b})$ with the $\Bbbk[x]$-module structure given by $J$. The right radical $\mathcal{R}(\mathsf{b})$ is the socle of that module. 

 Write $R$ for the $\Bbbk[x]$-module $\mathcal{R}^{\infty}(\mathsf{b})$. Then $R = \bigoplus_{i=1}^{n} R^{(i)}$ where $R^{(i)} = \setj{v \in R \; | \; x^{i}v = 0 \text{ and } x^{i-1}v \neq 0}$. This is the grading on $R$ associated to the radical filtration on it.
 The basis we have previously chosen for $R$ is such that the matrix of $x\cdot- :R^{(k)}\rightarrow R^{(k-1)}$ is of the form
 $
 \left(
 \begin{smallmatrix}
    I \\ 
    0
 \end{smallmatrix}
 \right)
 $, where $I$ is an identity matrix. 

 Let $\mathsf{b}$ be of even type, and let $\mathsf{b}_{2i} = \mathsf{J}_{2i}^{m_{2i}}$ be the block of $\mathsf{b}$ consisting of its Jordan blocks of size $2i$. 
 Further, let $\mathsf{R}_{2i} = \mathcal{R}(\mathsf{b}_{2i})$. In particular, $\mathcal{R}(\mathsf{b}) = \mathsf{R}_{2} \perp \mathsf{R}_{4} \perp \cdots \perp \mathsf{R}_{2d}$.
\begin{lemma}
 The image of the homomorphism $\on{Aut}(\mathsf{b}) \rightarrow \on{GL}(\mathcal{R}(V))$ consists of the matrices of the form
 \[
 \begin{pmatrix}
     \on{GL}(\mathsf{R}_{2}) & 0 & \cdots & 0 \\
     \on{Hom}_{\Bbbk}(\mathsf{R}_{2},\mathsf{R}_{4}) & \on{GL}(\mathsf{R}_{4}) & \ddots & 0 \\
     \vdots & \ddots & \on{GL}(\mathsf{R}_{j}) & 0 \\
     \on{Hom}_{\Bbbk}(\mathsf{R}_{2},\mathsf{R}_{2d}) & \on{Hom}_{\Bbbk}(\mathsf{R}_{4},\mathsf{R}_{2d}) & \cdots & \on{GL}(\mathsf{R}_{2d})
 \end{pmatrix}
 \]
\end{lemma}

\begin{proof}
    Clearly, it suffices to check the case $\mathsf{b} = \mathsf{J}_{2k} \perp \mathsf{J}_{2m}$, with $k \leq m$. 
    Using the identification of the actions of $\on{Aut}(\mathsf{b})$ and of $\on{Aut}_{\Bbbk[x]}(\mathcal{R}^{\infty}(\mathsf{b}))$, showing that for $g \in \on{Aut}(\mathsf{b})$ and $k < m$, we have $g(\mathsf{R}_{2m}) \cap \mathsf{R}_{2k} = \setj{0}$ can be done by proving that for the $\Bbbk[x]$-modules given by $J_{2k}$ and $J_{2m}$, a morphism $\upsilon: J_{2m} \rightarrow J_{2k}$ sends $\on{Soc}(J_{2m})$ to zero. 
    This follows by $\upsilon$ being non-injective, but can also be seen by $\on{Soc}(J_{2m}) = \on{Rad}^{2m}(J_{2m})$ and $\upsilon(\on{Rad}^{2m}(J_{2m})) \subseteq \on{Rad}^{2m}(J_{2k}) = \setj{0}$.

    It suffices to show that given $\gamma_{k} \in \on{GL}(\on{Soc}(J_{2k})), \gamma_{m} \in \on{GL}(\on{Soc}(J_{2m})$ and $\delta \in \on{Hom}_{\Bbbk}(J_{2k}, J_{2m})$, there is an automorphism of $\xi: J_{2k} \oplus J_{2m}$ such that $\xi|_{\on{Soc}(J_{2k} \oplus J_{2m})}$ is given by 
    $\left(
    \begin{smallmatrix}
        \gamma_{k} & 0 \\ \delta & \gamma_{m}
    \end{smallmatrix}
    \right)$.
     The automorphism we define is graded, in the bases we considered previously (where $x\cdot = \left(
 \begin{smallmatrix}
    I \\ 
    0
 \end{smallmatrix}
 \right)$) it is given by by $J_{k}^{(l)} \oplus J_{m}^{(l)} \xrightarrow{
     \left(
     \begin{smallmatrix}
        \gamma_{k} & 0 \\ \delta & \gamma_{m}
    \end{smallmatrix}
     \right)} J_{k}^{(l)} \oplus J_{m}^{(l)}$
     for $l \leq k$ and by
     $J_{k}^{(l)} \oplus J_{m}^{(l)} \xrightarrow{
     \left(
     \begin{smallmatrix}
        0 & 0 \\ 0 & \gamma_{m}
    \end{smallmatrix}
     \right)} J_{k}^{(l)} \oplus J_{m}^{(l)}$ for $k < l \leq m$, as then $J_{2k}^{(l)} = \setj{0}$. 
     Using the matrix description of the action of $x$ it is easy to see that these assignments commute with the action.
\end{proof}

\begin{corollary}
    If $\mathsf{b}$ is of even type, then the morphism $\on{Aut}(\mathsf{b}) \rightarrow \on{GL}(\mathcal{R}(\mathsf{b}))$ is surjective if and only if $\mathsf{b}$ is of homogeneous even type.
\end{corollary}

\begin{corollary}
    If $\mathsf{b}$ is of homogeneous even type, then
    $\big(\mathcal{R}(\mathsf{b}) \otimes \mathcal{L}(\mathsf{b})\big)/\on{Aut}(\mathsf{b}) = \big(\mathcal{R}(\mathsf{b}) \otimes \mathcal{L}(\mathsf{b})\big)/\on{GL}(\mathcal{R}(V))$,
    where the latter action is that described in \autoref{blockmatrix}.
\end{corollary}

\begin{corollary}\label{cor:GIT}
    Let $\mathsf{b}$ be of homogeneous even type. 
    The action of $\on{GL}(\mathcal{R}(V))$ on $\setj{\mathtt{t} \in \mathcal{R}(\mathsf{b}) \otimes \mathcal{L}(\mathsf{b}) \; | \; \mathsf{b}(\mathsf{t}) = 1}$ is isomorphic to that of $\on{GL}_{n}(\Bbbk)$ on $\setj{A \in \on{Mat}_{n\times n}(\Bbbk) \; | \; \on{Tr}(A) = 1}$. 
    The dimension of the associated affine GIT quotient is $n-1$.
\end{corollary}

\begin{proof}
    The action described in \autoref{blockmatrix} is clearly isomorphic to $\on{GL}_{n}(\Bbbk)$ acting on $\Bbbk^{n} \otimes \Bbbk^{n}$ by letting $g \in \on{GL}_{n}(\Bbbk)$ map $v \otimes w$ to $(g\cdot v) \otimes (g^{-1})^{T}\cdot w$. 
    The map $\Bbbk^{n} \otimes \Bbbk^{n} \rightarrow \on{Mat}_{n\times n}(\Bbbk)$ sending $v \otimes w$ to $vw^{T}$ is a $\on{GL}_{n}(\Bbbk)$-module isomorphism to the conjugation action. 
    The bases $e_{1}^{\mathcal{R}},\ldots,e_{n}^{\mathcal{R}}$ for $\mathcal{R}(V)$ and $e_{1}^{\mathcal{R}},\ldots,e_{n}^{\mathcal{R}}$ for $\mathcal{L}(V)$ are such that $\mathsf{b}(e_{i}^{\mathcal{R}},e_{j}^{\mathcal{L}}) = \delta_{ij}$, so $\mathsf{b}(v \otimes w) = \on{Tr}(vw^{T})$ and the trace condition follows. 
    The affine GIT quotient associated to the conjugation action of $\on{GL}_{n}(\Bbbk)$ is well-known: a matrix $A$ is sent to the orbit of the diagonal matrix whose entries are the generalized eigenvalues of $A$, with multiplicities, and for the $\on{GL}_{n}(\Bbbk)$-invariant regular functions we have $\mathcal{O}(\on{Mat}_{n\times n}(\Bbbk)^{\on{GL}_{n}(\Bbbk)} = \Bbbk[c_{1},\ldots,c_{n}]$, the coefficients of the characteristic polynomial. 
    Since $c_{1} = -\on{Tr}$, the dimension of $\setj{A \in \on{Mat}_{n\times n}(\Bbbk) \; | \; \on{Tr}(A) = 1}$ is $n-1$. 
\end{proof}
Note that the latter part of \autoref{cor:GIT} can also be shown by observing that $\setj{A \in \on{Mat}_{n\times n}(\Bbbk) \; | \; \on{Tr}(A) = 1}$ is an isomorphic $\on{GL}_{n}(\Bbbk)$-module to $\mathfrak{sl}_{n}(\Bbbk)$.

The dimension computed in \autoref{cor:GIT} is to be compared with \cite{EO}, quantifying the difference between the case $q\neq 0$, where determining $U(1)$ and $U(\on{cap})$ specifies the fiber functor uniquely, whereas in our case we get an affine space of closed orbits of solutions. 
In the extreme case $\mathsf{b} = \mathsf{J}_{2}^{\perp m}$, the dimension of $U(1)$ is $2m$, while the dimension of the GIT quotient is $m-1$.

	

\end{document}